\documentclass{siamltex}
 \usepackage[breaklinks,colorlinks=true,linkcolor=blue,citecolor=red,backref=page]{hyperref}
\usepackage{amsfonts} 
\usepackage{amssymb,amsmath} 
\usepackage{graphicx,color}

\DeclareMathAlphabet{\itbf}{OML}{cmm}{b}{it}

\def\bq{{{\itbf q}}}

\def\br{{{\itbf r}}}

\def\by{{{\itbf y}}}
\def\bx{{{\itbf x}}}

\def\bk{{{\itbf k}}}

\def\bzeta{{\boldsymbol{\zeta}}}
\def\bxi{{\boldsymbol{\xi}}}

\def\eps{{\varepsilon}}

\newcommand{\RR}{\mathbb{R}}

\newcommand{\EE}{\mathbb{E}}

\begin{document}  
 
 \title{Speckle Memory Effect in the Frequency Domain and Stability in Time-Reversal Experiments} 

\author{Josselin Garnier\footnotemark[1] 
\and Knut S\O lna\footnotemark[2]}  

\maketitle

\author{Josselin
Garnier\thanks{\footnotesize Centre de Math\'ematiques Appliqu\'ees, Ecole Polytechnique,
91128 Palaiseau Cedex, France (josselin.garnier@polytechnique.edu)} 
\and Knut S\o lna\thanks{\footnotesize Department of Mathematics, 
University of California, Irvine CA 92697
(ksolna@math.uci.edu)}
}

\renewcommand{\thefootnote}{\arabic{footnote}}

\begin{abstract}
When waves propagate through a complex   medium like the turbulent atmosphere the wave field becomes
incoherent and the wave intensity forms a complex speckle pattern. 
In this paper we study a speckle memory effect in the frequency domain and some of its consequences.
This effect means that certain properties of the speckle pattern produced by wave transmission through a randomly scattering medium 
is  preserved when shifting the frequency of the illumination.
The speckle memory effect  is  characterized 
via  a detailed novel analysis of the fourth-order moment of the random paraxial Green's function at four different frequencies.
We arrive at a precise   characterization of the frequency memory effect and what governs the strength of the memory. 
As an application we quantify the statistical stability of time-reversal wave refocusing through a randomly scattering medium
in the paraxial or beam regime. Time reversal refers to the situation when a transmitted wave field is recorded on a time-reversal mirror then
time reversed  and sent back into the complex medium. The reemitted wave field then refocuses at the original source point.  
We compute the mean of the refocused wave and identify a novel
quantitative description of its variance
in terms of the radius of the time-reversal mirror, the size of its elements, the source bandwidth
and  the statistics of the random medium fluctuations.
 \end{abstract}



\pagestyle{myheadings}
\thispagestyle{plain}
\markboth{J. Garnier and K. S\o lna}
{Speckle Memory Effect in the Frequency Domain}

  
\section{Introduction}
For imaging or communication purposes  it is important to understand how waves propagate through a randomly scattering medium. 
The quantities of interest can generally be expressed in terms of statistical averages.
Usually the first- and second-order moments of the Green's function are sufficient to characterize them.
However in some circumstances fourth-order moments are needed, for instance
for scintillation problems \cite{fps,garniers3} or the analysis of intensity correlation-based imaging \cite{bertolotti,katz12}.
For imaging with  narrow or broad band signals it is also important to characterize multifrequency moments \cite{borcea}.  
We consider here the 
paraxial regime corresponding  to high-frequency and long-range propagation of a wave beam.    
The paraxial regime is physically relevant and it models  many situations, for instance laser beam propagation \cite{andrews,strohbehn}
or underwater acoustics \cite{tappert}.
The equations that govern the evolution of the fourth-order moments in the paraxial regime
have been known for a long time \cite{uscinski}-\cite[Sec.~20.18]{ishimaru}.
The solution  of the fourth-order moment problem 
was recently  analyzed and discussed in \cite{garniers3,garniers4} 
when  the four Green's functions involved in the fourth-order moment are evaluated at  the same frequency.
In this paper, we extend this result to the case when the four Green's functions have different  frequencies.
This new result makes it possible to analyze a number of configurations in wave propagation and imaging.  
Here we consider  two main motivating applications:

- The first motivating application  is time-harmonic wave focusing through a random medium.
Wavefront-shaping-based schemes 
\cite{popoff14,rotter,vellekoop10,vellekoop07,vellekoop08} 
have indeed attracted attention in recent years, particularly because of their potential applications
for focusing and imaging through scattering media.
The primary goal is to focus monochromatic light through a layer of strongly scattering material.
This is a challenging problem as multiple scattering of waves scrambles the transmitted light into random interference intensity patterns called speckle patterns \cite{goodman}.
This is shown in Figure \ref{fig:0a}(a): without control of the source the intensity of the transmitted
field  forms a complex speckle pattern. 
However, by using a spatial light modulator (SLM) before the scattering medium, 
it is possible to focus light as first demonstrated in \cite{vellekoop07}. 
Indeed, the elements of the SLM can impose phase shifts, and an optimization scheme makes it possible to choose the phase shifts so as to maximize the 
intensity transmitted at one target point  behind the scattering medium. This is shown in Figure \ref{fig:0a}(b). 
The optimal phase shifts depend on the medium and they are equal to the opposite phases of the field emitted by a point source at the target point 
and recorded in the plane of the SLM \cite{mosk12}.
In other words, the wavefront-shaping optimization procedure is equivalent to 
phase conjugation or time reversal. This is illustrated in Figure \ref{fig:0b} which describes a time-reversal experiment.
A time-reversal experiment consists of two steps and it is based on the use of a special device, a time-reversal mirror (TRM), that is used as an array of receivers in the first step and as an array of sources in the second step.
The first step is described in picture (a): a point source emits
a wave that propagates through a scattering medium and that is recorded by the TRM.
The second step is described in picture (b): the recorded signals are time-reversed and reemitted into the same medium by the TRM, and the reemitted waves then focus at the original  source point. At a single frequency this process corresponds to
phase conjugation or reemission of  the complex conjugate of the recorded wave field by the TRM;
with some abuse of notation we refer to this process 
as time-harmonic time reversal. 
It has been shown that the speckle memory effect \cite{feng88,freund88}
allows to focus on a neighboring point close to the original target point \cite{vellekoop10,vellekoop07,vellekoop08},
which opens the way for  the transmission of spatial patterns \cite{garniers5,garniers6,garniers7,popoff10}.
Indeed, one the main manifestations of the spatial memory effect is the following one:
By applying an appropriate and deterministic spatial phase modulation to the conjugated source field in the second step of the time-reversal experiment (Figure \ref{fig:0b}(b)) one can achieve  that 
the focusing  (red spot) in the bottom right plot is shifted.
By properly composing such 
modulated source fields one can transmit a pattern, see \cite{garniers5} for a detailed discussion.  
A main  question  we want to address here is whether such  speckle memory effects
can be exploited also  in the frequency domain.
In fact, we show that it is  possible to focus a time-harmonic signal with a different
frequency than the one 
of the field recorded  by the TRM in Figure \ref{fig:0b}(a). One can even focus a broadband pulse
and this opens the way to the transmission of short pulses, see 
 \cite{mosk12} for experimental verification of the frequency memory effect. 
 The process then  corresponds to using and processing the reference phase-conjugated field 
in Figure \ref{fig:0b}(b) in order to focus coherently time-harmonic waves with slightly shifted frequencies.  The reference  field or 
a `guide star' field may then be used over a frequency band to obtain focusing for pulses.   
The theoretical description of such a frequency memory effect has  so far been an open question. 
In Section \ref{sec:refocusth} we give a quantitative description of the effect of a frequency shift on refocusing,
which is directly related to the speckle memory effect in the frequency domain. We show that the speckle pattern 
is only slightly changed when  shifting the frequency so that we can use the same source phase field over
a range of frequencies and still obtain focusing for all frequencies in the band. 
A main result presented in Section \ref{sec:refocusth} is that the width $\Omega$ of the frequency 
band  for which  we can use the same recorded and conjugated field at the TRM and still achieve focusing 
is determined  by the speckle coherence frequency $\Omega_{\rm spec}$:
\begin{equation}
  \Omega \lesssim  \Omega_{\rm spec} := \frac{\ell_{\rm par}} {L T}      , 
\end{equation}
 where $T=L/c_o$ is the travel time over the distance $L$ from the source to the TRM for 
a background wave speed $c_o$  and $\ell_{\rm par}$ is the paraxial distance introduced in (\ref{def:lpar}) below.
The paraxial distance  corresponds to the travel distance at which the paraxial description of the wave beam 
in the random medium breaks down and  is inversely  proportional to a 
measure of the lateral scattering strength in the random medium.
 It follows that for longer propagation distances and stronger medium
fluctuations the frequency band at which the frequency memory holds becomes narrower 
since  the  speckle pattern then becomes  more sensitive to a shift in the source frequency.       
  \begin{figure}
\begin{center}
\begin{picture}(330,160)
\put(0,117){{\bf (a)}}
\put(10,80){\includegraphics[width=7.7cm]{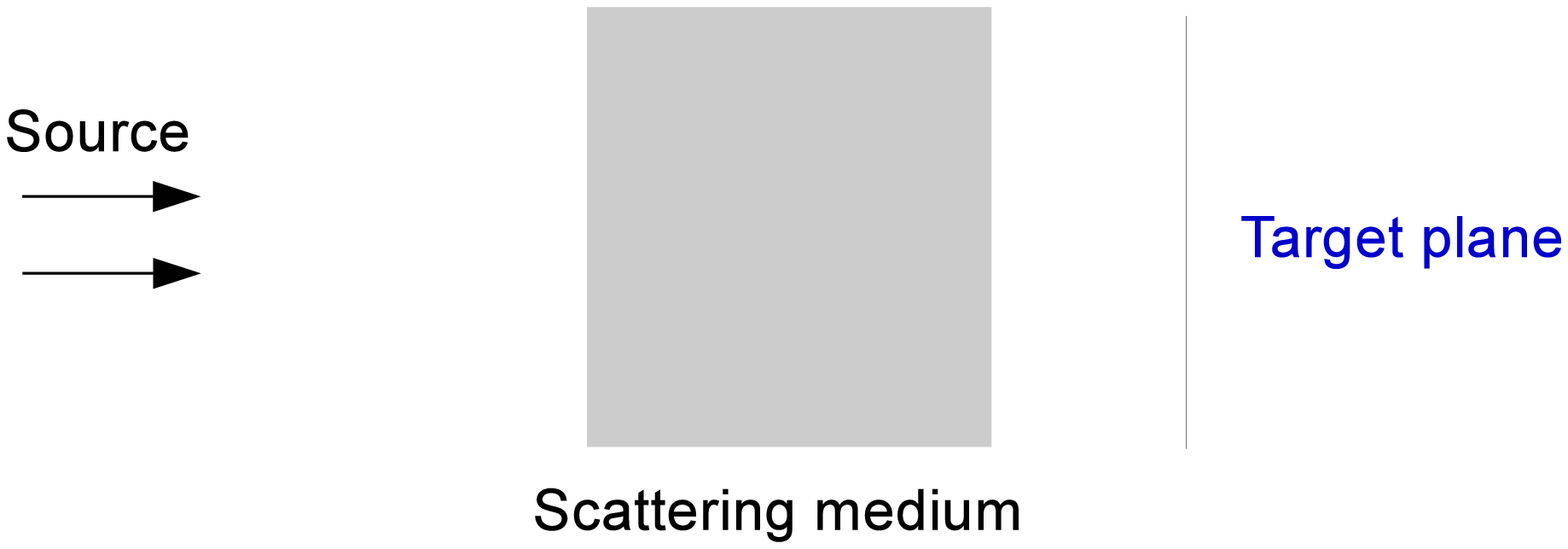}}
\put(225,84){\includegraphics[width=3.3cm]{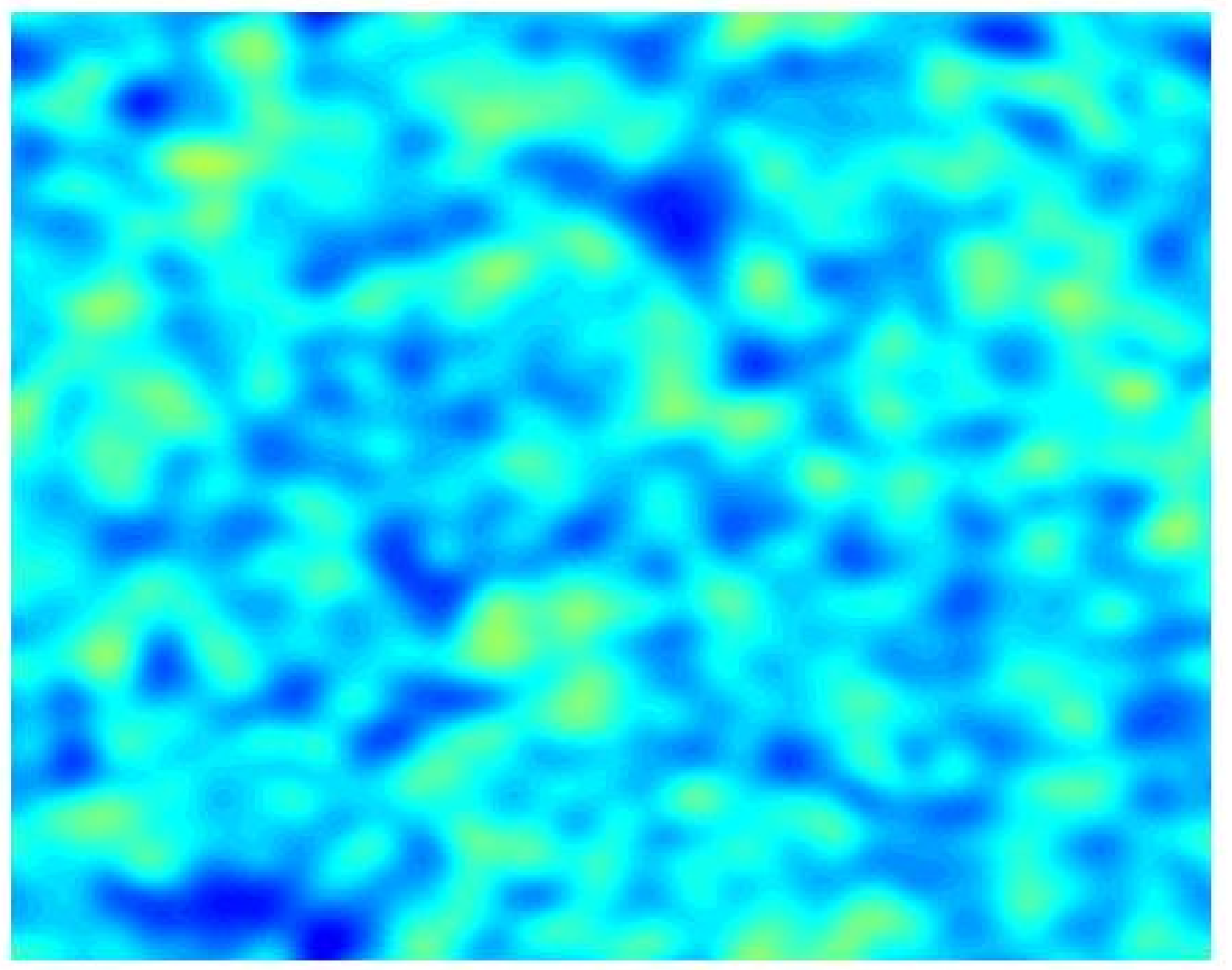}}
\put(0,37){{\bf (b)}} 
\put(10,0){\includegraphics[width=7.7cm]{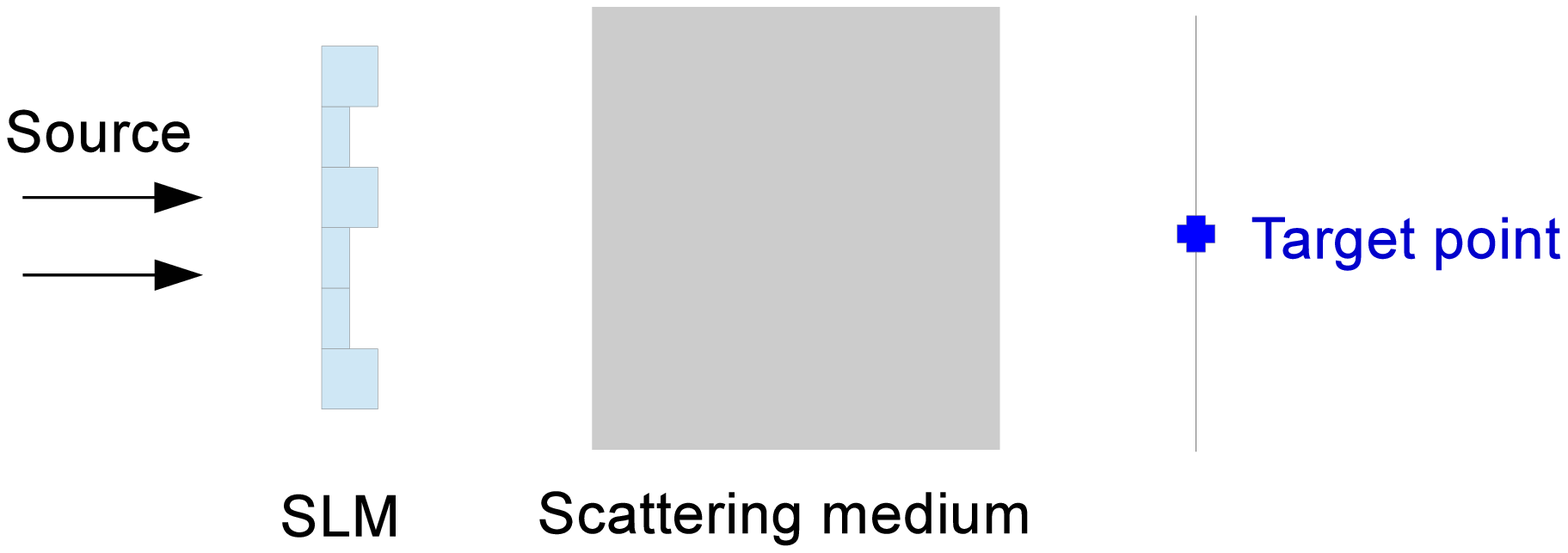}}
\put(225,4){\includegraphics[width=3.3cm]{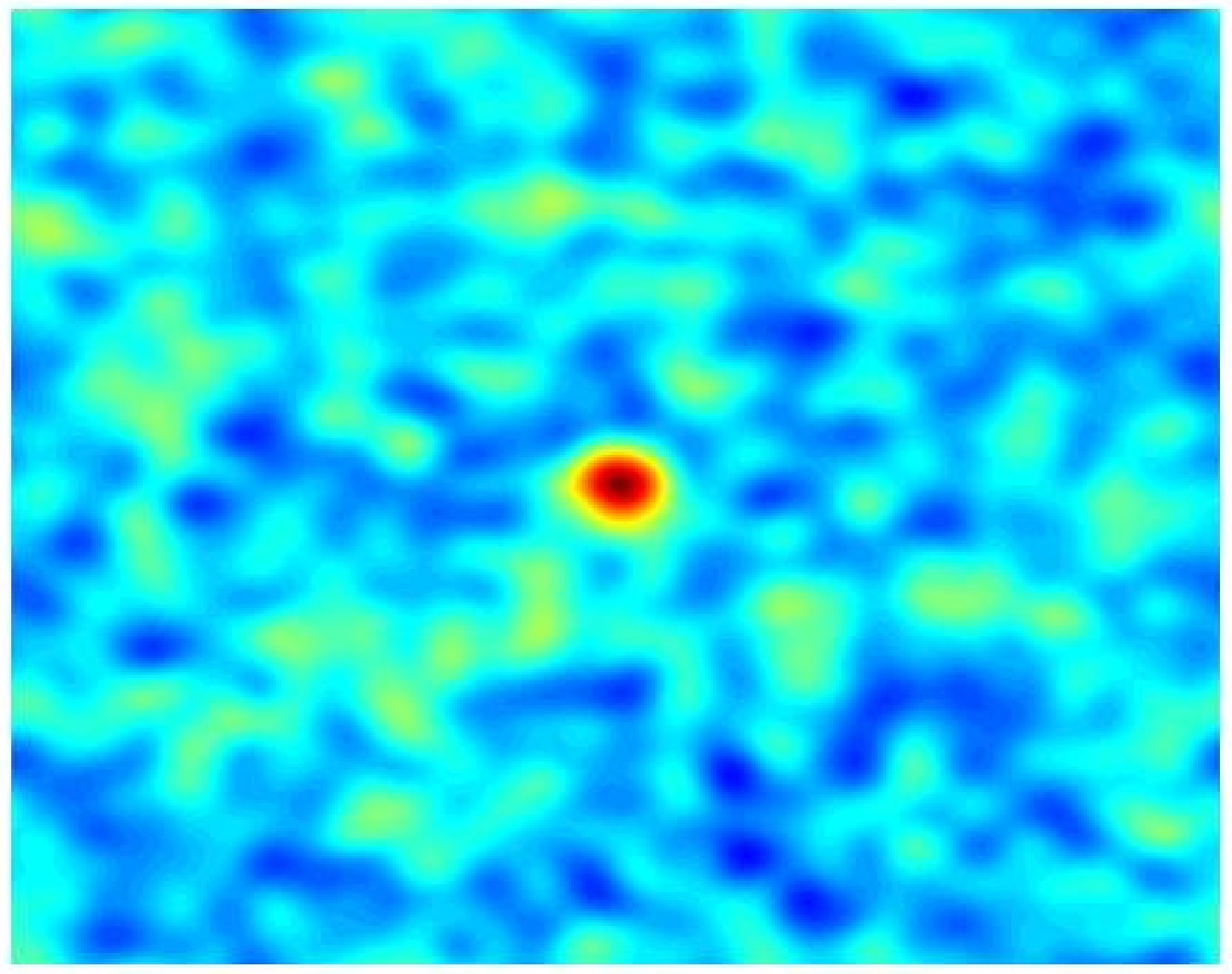}}
\end{picture}
\end{center}
\caption{Focusing wave through a scattering medium. Without any control 
one gets a speckle pattern in the target plane (a).
With a spatial light modulator (SLM) one can focus on a target point by imposing appropriate phase shifts (b) [From \cite{garniers5}].
\label{fig:0a} 
}
\end{figure}
\begin{figure}
\begin{center}
\begin{picture}(330,160)
\put(0,117){{\bf (a)}}
\put(10,80){\includegraphics[width=7.7cm]{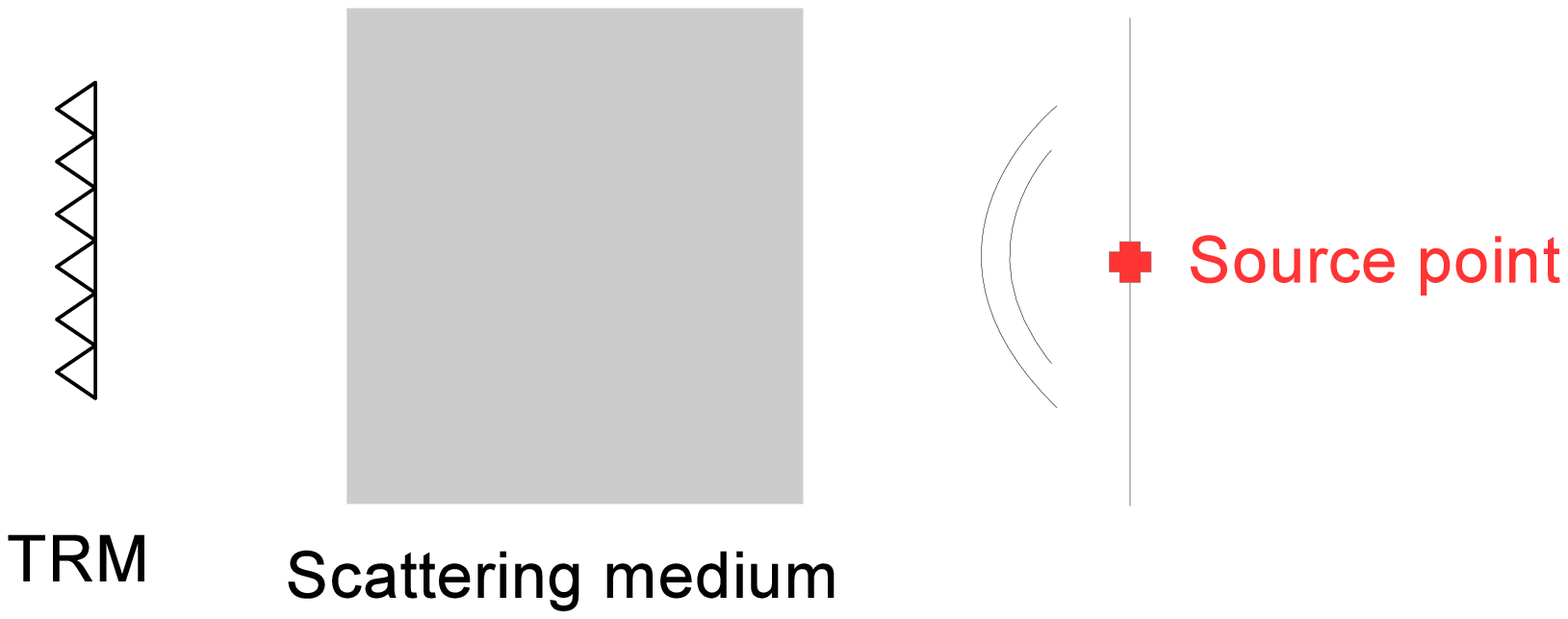}}
\put(0,37){{\bf (b)}} 
\put(10,0){\includegraphics[width=7.7cm]{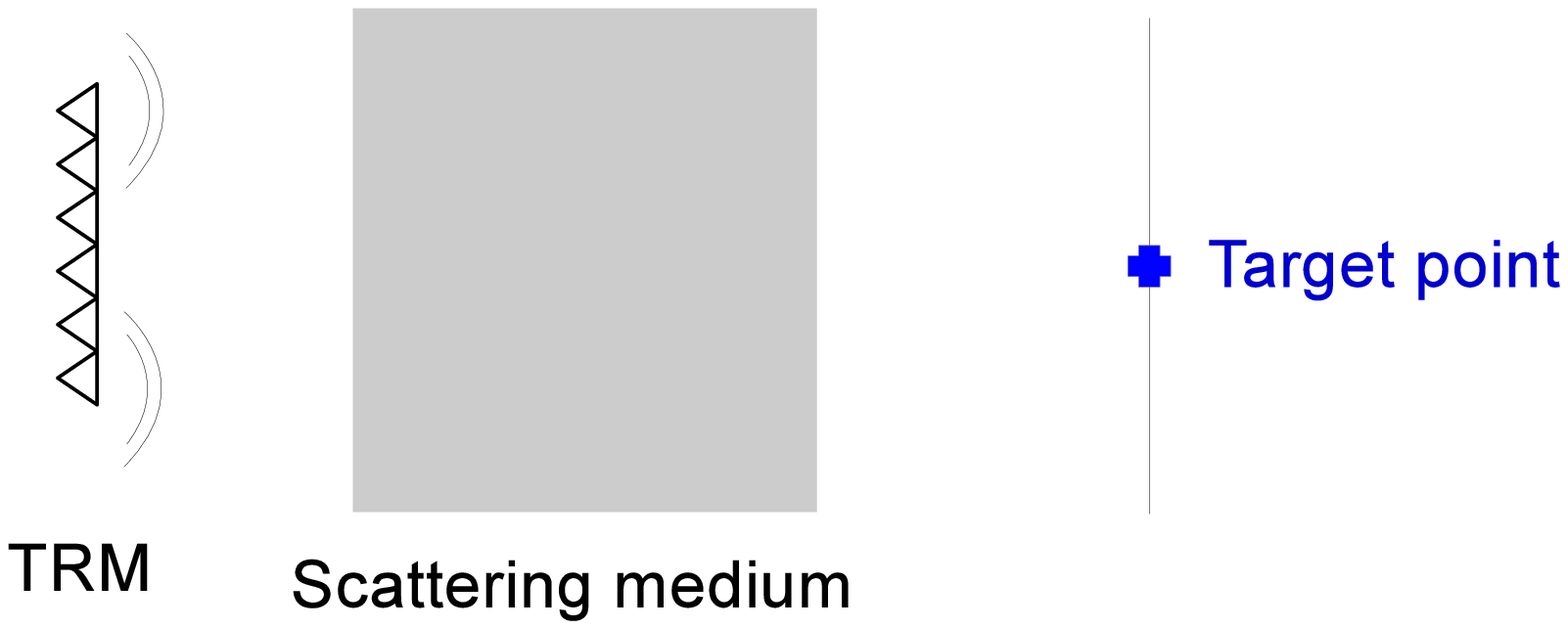}}
\put(205,-4){ \includegraphics[width=4.0cm]{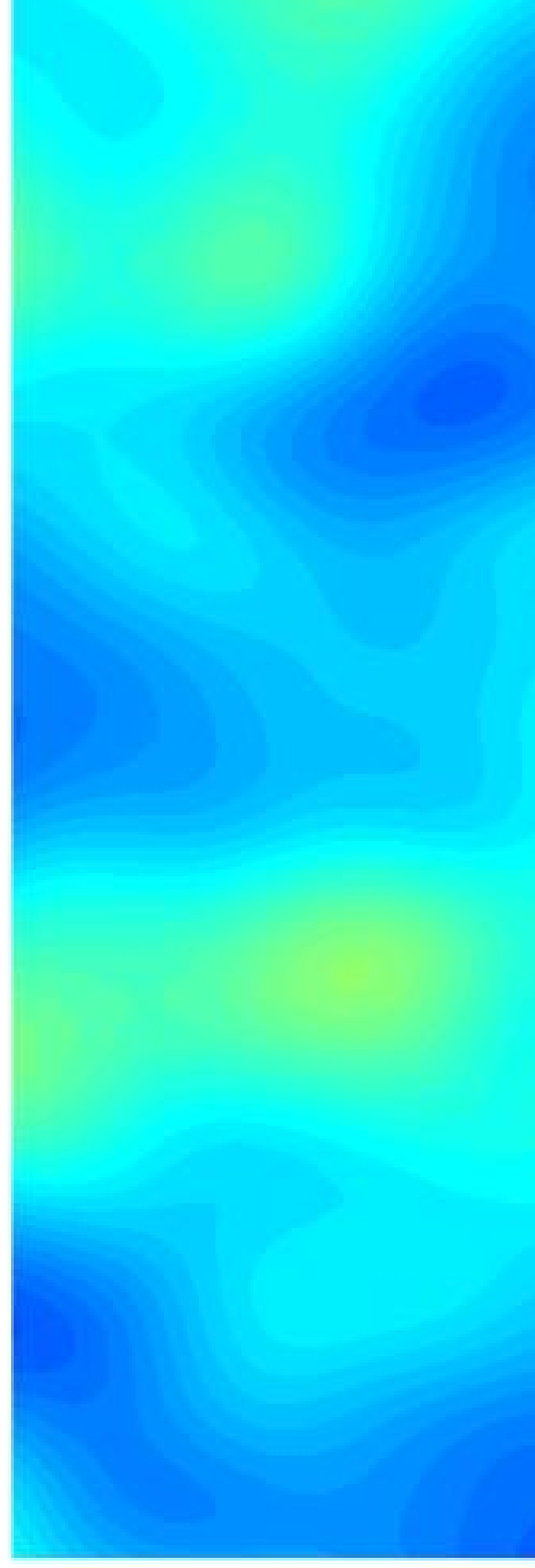}}
\end{picture}
\end{center}
\caption{Time-reversal experiment through a scattering medium. In the first step of the experiment (a)
a time-harmonic point source emits a wave that propagates through the scattering medium and is recorded by the time-reversal mirror (TRM) used as an array of receivers. 
In the second step of the experiment (b) the TRM is used as an array of sources, it
emits the complex-conjugated recorded field, and the wave refocuses at the original source location (the cross in the right image
stands for the original source location; the focal spot is centered at the cross) [From \cite{garniers5}].
\label{fig:0b} 
}
\end{figure}
 
- The second motivation for our multi-frequency analysis is statistical stability in  time reversal. Time reversal for waves in random media has indeed 
been studied theoretically, numerically, and experimentally (see the review \cite{fink00}).
As mentioned above when a wave is emitted by a point source and recorded by a TRM, which then reemits the time-reversed recorded signals, then in general the 
wave refocuses on the original source location, see Figure \ref{fig:0b}. 
It moreover turns out that refocusing is enhanced when the medium is randomly scattering,
and that the time-reversed refocused wave is statistically stable, in the sense that its shape
depends on the statistical properties of the random medium, but not on its particular realization.
The phenomenon of focusing enhancement has been analyzed quantitatively \cite{blomgren,book1,lerosey07,PRS04}.
Statistical stability of time-reversal refocusing for broadband pulses is usually qualitatively proved by invoking the fact that the
time-reversed refocused wave is the superposition of many independent frequency components, which gives the self-averaging
property in the time domain \cite{blomgren,PRS04}.
However, so far, there has not been a fully satisfactory analysis of the statistical stability phenomenon, because
it involves the evaluation of a fourth-order moment of the Green's function of the random wave equation.
This problem has been addressed in \cite{ishimaru07} in a situation similar to  the one addressed in this paper,
but using the circular complex Gaussian assumption for the evaluation of the fourth-order moments
that are needed for the analysis. Here we will not make use of this assumption, rather we will 
prove that the fourth-order moments can be computed and this allows us to give a detailed analysis
of the statistical stability of the time-reversed refocused wave.
In Section \ref{sec:refocusbb} we quantify time-reversal refocusing and stability as functions of the size of the TRM, the size of its elements,
the source bandwidth, and the statistical properties of the random medium.
The main results can summarized as follows: if the bandwidth $B$ of the source is small
so that $B\ll \Omega_{\rm spec}$ 
and also if the scattering is  strong enough so that 
 the spreading of the beam is large relative to its original width,
then the signal-to-noise ratio  (SNR) of the refocused wave 
is roughly  equal to  the  number of elements  in the TRM:
 \begin{equation}
    {\rm SNR}  \simeq N \hbox{~~for~~} N:= \left(\frac{r_0}{\rho_0}\right)^2  ,
\end{equation}
 with $r_0$ being the size of the TRM and $\rho_0$ the size of the elements.  
 If the bandwidth $B$ of the source is large so that $B \gg \Omega_{\rm spec}$ 
 and if scattering is strong,    then   
 \begin{equation}
    {\rm SNR}   \simeq \left( \frac{N}{8} \right) \left( \frac{B}{\Omega_{\rm spec}} \right)  . 
\end{equation}
This shows that the source bandwidth improves the statistical stability of the refocused wave,
provided it is larger than the speckle coherence frequency.
 This then quantifies the usual assertion found 
in the literature that the profile of the time-reversed field is self-averaging by independence of the frequency
components of the wave field and  clarifies  the hypotheses which ensure that such a result is valid.  
We remark here also that in the strongly scattering situation and small mirror elements it is a classic
result that the time-reversal refocusing resolution $R$  can be expressed as the Rayleigh
resolution formula $R\approx \lambda L / A_{\rm eff}$ evaluated at the central wavelength $\lambda$ and at the scattering-enhanced aperture $ A_{\rm eff}$ scaling with propagation distance as $L^{3/2}$ \cite{garniers5}. In the notation introduced here this means that
\begin{eqnarray}
   R \approx \lambda  \sqrt{\frac{ \ell_{\rm par}}{ L} }   ,
\end{eqnarray}   
where we need $ \ell_{\rm par} >  L$ for the paraxial
  approximation to be valid.  Note that this resolution measure is independent of the actual
TRM radius.   

 
The paper is organized as follows.
First in Section \ref{sec:0} we outline  the main setting with scalar waves propagating  
in a random medium and summarize the main result regarding the  paraxial approximation that we use,
the solution of the  It\^o-Schr\"odinger equation.  
In Section \ref{sec:1} we describe the two main applications that we have introduced:
time-harmonic refocusing and broadband time reversal.
In Sections \ref{sec:3}-\ref{sec:4} we study in detail the second- and fourth-order moments of the paraxial Green's function
at different frequencies and how we get successively simpler expressions for the moments
by making further assumptions regarding the scaling regime. 
We quantify the focusing properties of the two main applications in terms of resolution and stability
in Sections \ref{sec:refocusth}-\ref{sec:refocusbb}. In Appendix \ref{app:a} we discuss 
in more detail the scaling regime that we use and how it relates to the   It\^o-Schr\"odinger equation
that is fundamental to our asymptotic moment analysis.      

\section{Paraxial Waves in Random  Media }\label{sec:0} 
 We consider  scalar waves and 
assume the governing equation:
\begin{eqnarray}\label{eq:wave}
 (\partial_{z}^2 + \Delta_\bx) u
-\frac{n^2(z,\bx)}{c_o^2} \partial_{t}^2 u 
 =0 ,
\end{eqnarray}
for $(z,\bx) \in \RR \times \RR^{2} $,  the space coordinates.
 In (\ref{eq:wave})  $n(z,\bx)$ is the local index of refraction
that we model as random and we assume radiation conditions at infinity.
We remark that even though the scalar wave equation is simple and linear, the relation 
between the statistics of the index of refraction and the statistics of the 
wave field is highly 
nontrivial and nonlinear.
Originally motivated by elastic problems in geophysics, we assume that
the privileged propagation axis is the $z$-direction and will consider
beam waves propagating into the  $z$-direction, thus corresponding to the 
horizontal direction in Figures \ref{fig:0a} and \ref{fig:0b}.  
We model moreover the complex medium as a random medium 
and do this by letting  the local index of refraction in (\ref{eq:wave}) be 
parameterized  by
\begin{eqnarray}\label{eq:rm}
  n^2 (z,\bx) =   1 +     \nu(z,\bx)   ,  
\end{eqnarray} 
for $\nu$ being the centered random medium fluctuations.
We assume that $\nu$ is a stationary zero-mean 
random field  that is mixing in $z$ and with integrable correlations. 
 
It is now convenient to Fourier transform  in time:
\begin{eqnarray}\label{eq:deffourier}
\hat{u}(\omega,z,\bx) = \int_\RR
{u}(t,z,\bx) \exp \big( i \omega t \big) {d}t .
\end{eqnarray}
We then obtain the  Helmholtz or reduced wave equation :
\begin{eqnarray}\label{eq:helm}
(\partial_{z}^2+\Delta_\bx) \hat{u} + 
\frac{\omega^2}{c_o^2}n^2(z,\bx) \hat{u}=0 ,
\end{eqnarray}
with  $k=\omega/c_o$ being  the free space wavenumber. 
   
A particular solution of (\ref{eq:helm}) in the case of a homogeneous
medium $n \equiv 1$ is a  
plane wave propagating in the $z$ direction:
\begin{eqnarray*}
\hat{u}= \exp \Big( i\frac{\omega }{c_o}  z \Big) .
\end{eqnarray*}
We make the ansatz of a slowly-varying envelope 
around a plane wave going into the $z$-direction  
\begin{equation}
\hat{u}(\omega,z,\bx)  = 
\exp \Big( i\frac{\omega }{c_o}  z \Big) v(\omega,z,\bx)  .
\end{equation}
In the white-noise paraxial regime (which holds when the wavelength is much smaller than the correlation length of the medium and the beam radius, which are themselves much smaller than the propagation distance)  we can then model $v$ 
in terms of the solution of the following It\^o-Schr\"odinger equation:
\begin{equation}\label{eq:model0}
2 i k d v +  \Delta_\bx v \, dz +   k^2 
 v\circ dB(z,\bx) =0 .
\end{equation}
  In Appendix \ref{app:a} we discuss in detail the scaling assumptions of the
white-noise paraxial regime leading to the model  (\ref{eq:model0})  for computing 
moments of waves emitted from sources satisfying the scaling assumptions 
as outlined in the appendix.    
We remark that  the symbol $\circ$ stands for the Stratonovich stochastic integral,
  $B(z,\bx)$ is a real-valued  Brownian field over $[0,\infty) \times \RR^2$ with  covariance
 \begin{equation}
 \label{defB}
\EE[   {B}(z,\bx)  {B}(z',\bx') ] =  
 {\min\{z, z'\}}   {C}(\bx - \bx')   ,
\end{equation}
and $C$ is determined by the two-point statistics of the fluctuations of the random medium
as
\begin{equation}\label{def:C2} 
    C(\bx) = \int_{\RR} \EE[\nu(0,{\bf 0}) \nu(z,\bx) ] dz  ,
\end{equation} 
with $\nu$ being the random medium fluctuations in (\ref{eq:rm}).  
 Note therefore that in particular the width of $C$ is the correlation  length   of the medium fluctuations.
The  It\^o-Schr\"odinger
equation was analyzed for the first time in \cite{dawson84} and it was derived from first principles
by a multiscale analysis of the wave equation in a random medium in \cite{garniers1}.
 The model (\ref{eq:model0}) leads to closed equations for wave field moments of all orders.
 We discuss in the appendix  the first-order moment equation that is readily solvable.
 The second-order one-frequency moment equations are also explicitly solvable, while
 the fourth-order equations are not explicitly solvable in the white-noise
 paraxial regime, neither in the one-frequency nor in the multi-frequency cases. 
 However, in a secondary scaling regime 
 that we denote the scintillation regime we will be able to  solve both
 the second-order and fourth-order multi-frequency moments.   
 We will push through this moment analysis in Section \ref{sec:4}. 
 Before this, 
in Section \ref{sec:1},  we discuss  the detailed modeling of the two applications which motivates the particular form of the second- and fourth-order 
 multi-frequency  moments that we consider. In Section \ref{sec:3} we express these moments
 in terms of the moments of  the Green's function associated with the It\^o-Schr\"odinger equation
 (\ref{eq:model0}). 
   
\section{Time-Reversal Experiment}
\label{sec:1}%
We assume that a TRM is located in the plane $z=0$. 
The  radius of the mirror is $R_{\rm m}$ and the radius of its elements is $\rho_0$.

\begin{figure}
\begin{center}
\begin{tabular}{c}
\includegraphics[width=4.0cm]{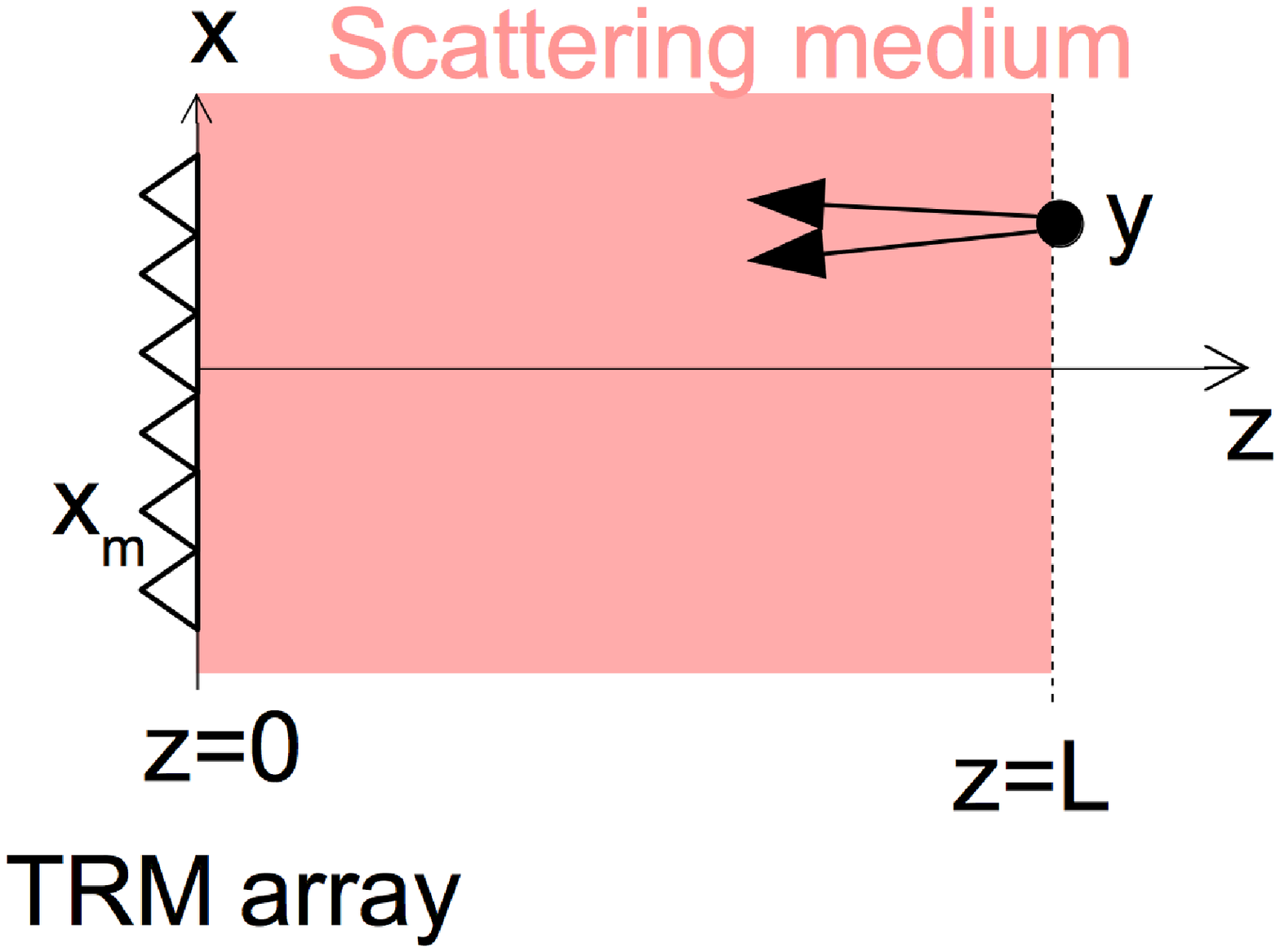}  
\hspace*{1.cm}
\includegraphics[width=4.1cm]{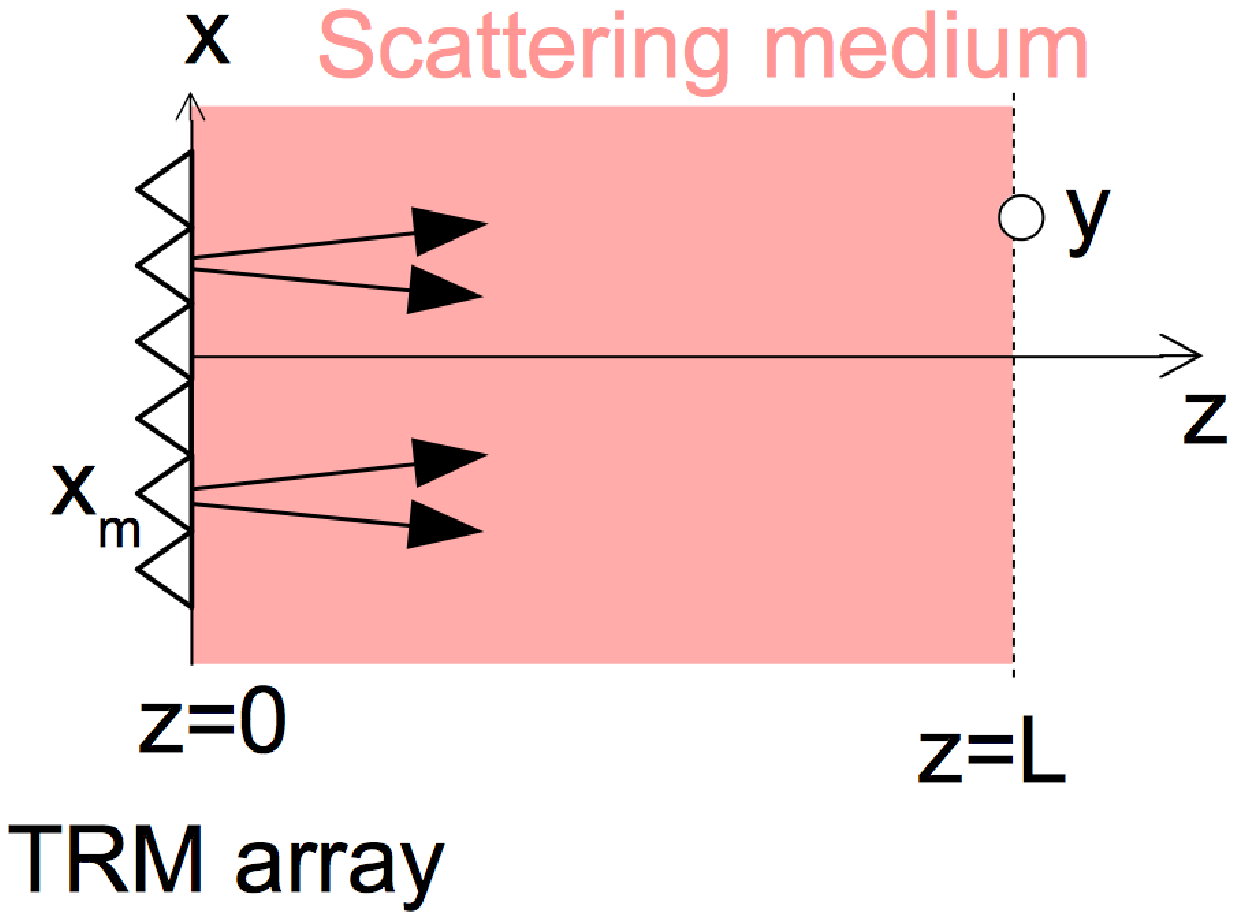}  
\end{tabular}
\end{center}
\caption{Time-reversal experiment. Left: first step of the experiment (a point source transmits from $(\by,L)$ and the TRM in the plane $z=0$  is used as an array of receivers).
Right: second step of the experiment (the TRM is used as an array of sources).}
\label{fig:setup}
\end{figure}

\subsection{Time-Harmonic Refocusing Experiment}
\label{sec:1th}%
In the first step of the time-harmonic time-reversal experiment, 
a point  source localized at $(\by,L)$ emits a time-harmonic signal at frequency $\tilde{\omega}$ (see Figure \ref{fig:setup}).
The TRM is used as an array of receivers 
and records the wave emitted by the point source. The size $\rho_0$ of the elements of the TRM is taken into account 
in the form of a Gaussian smoothing kernel with radius $\rho_0$. 
We denote  the time-harmonic Green's function 
from $(\bx_m,0)$ to $(\by,L)$  by $\hat{\cal G}( L , \by, \bx_m)$ 
(which is equal to the Green's function from $(\by,L)$
to $(\bx_m,0)$ by reciprocity), the recorded field at $(\bx_m,0)$ can then be expressed as:
\begin{equation}
\label{def:urec}
\hat{u}_{\rm rec}( \bx_m ; \by) =
\frac{1}{2 \pi \rho_0^2} \int_{\RR^2} \hat{\cal G}(\tilde{\omega},L , \by, \bx_m+\bx') \exp \Big( - \frac{|\bx'|^2}{2 \rho_0^2} \Big) d\bx'  .
\end{equation}

In the second step of the experiment, 
the TRM is used as an array of sources. It emits
the complex-conjugated (time-reversed) recorded field $\overline{\hat{u}_{\rm rec}}$ at frequency $\omega$,
which can be different from $\tilde{\omega}$.
The field observed in the plane $z=L$ at the point $(\bx,L)$ has the form
\begin{equation}
\label{eq:trfield1}
\hat{u}_{\rm tr}( \bx; \by  ) =
 \int_{\RR^2}
   \hat{u}_{\rm em} (\omega, \bx;  \bx_m) 
 \exp \Big( - \frac{|\bx_m|^2}{R_{\rm m}^2} \Big)  \overline{ \hat{u}_{\rm rec}( \bx_m ; \by) } 
  d\bx_m .
\end{equation}
Here we have assumed that the TRM has a radius $R_{\rm m}$ and
can be modeled by a Gaussian spatial cut-off function.  Moreover, 
we again take into account the size $\rho_0$ of the elements of the TRM by considering that
from any point $(\bx_m,0)$ the  TRM  can transmit from an element with radius $\rho_0$ and 
with a Gaussian form,
which generates the following field at point $(\bx,L)$:
\begin{equation}
\label{def:uem}
\hat{u}_{\rm em} (\omega, \bx;\bx_m ) = \frac{1}{2 \pi \rho_0^2} 
\int_{\RR^2} \hat{\cal G}(\omega,L,\bx,\bx_m+\bx')  \exp \Big( - \frac{|\bx'|^2}{2 \rho_0^2} \Big)d\bx'  .
\end{equation}

The time-reversed field observed in the plane $z=L$ can therefore be expressed as
\begin{eqnarray}
\nonumber
\hat{u}_{\rm tr}(\bx ;\by ) &=& \frac{4 \omega \tilde{\omega}}{c_o^2} K_0
\iint_{\RR^2 \times \RR^2} 
\exp \Big( - \frac{|\bx'|^2}{r_0^2} - \frac{|\by'|^2}{4 \rho_0^2} \Big) \\
&& \times
\hat{\cal G}\big(\omega,L,\bx, \bx'+\frac{\by'}{2}\big) 
\overline{\hat{\cal G}\big(\tilde{\omega},L,\by,\bx'-\frac{\by'}{2}\big) }
  d\bx' d\by'   ,
  \label{eq:base}
\end{eqnarray}
with   
\begin{equation}
\label{def:r0}
K_0 = \frac{c_o^2(r_0^2 -\rho_0^2)}{16 \pi  \omega \tilde{\omega} \rho_0^2 r_0^2}
,\quad \quad 
r_0^2 = R_{\rm m}^2 + \rho_0^2  .
\end{equation}  
From now on we will take $K_0=1$ as this multiplicative factor does not play any role in what follows.

The goal of the forthcoming analysis is to quantity the refocusing properties in terms of resolution and stability, and 
to make it precise for which frequency offset $\omega-\tilde{\omega}$  it is possible to observe refocusing.

{\bf Remark.}
In this paper we model the global shape of the TRM and the local shape of the elements of the TRM by soft Gaussian cut-off functions,
instead of hard cut-off functions such as ${\bf 1}_{[0,R_{\rm m}]}(|\bx_m|)$ or ${\bf 1}_{[0,\rho_0]}(|\bx'|)$,
 because this makes it possible to get simpler expressions.
This does not affect qualitatively the results.

\subsection{Broadband Time-Reversal Experiment}
\label{sec:tr}
In the first step of the broadband time-reversal experiment, 
a point  source localized at $(\by,L)$ emits a short pulse $f(t)$  (see Figure \ref{fig:setup}).
The pulse has central frequency $\omega_0$ and bandwidth $B$.
The TRM in the plane $z=0$ is used as an array of receivers 
and records the wave emitted by the point source around the expected arrival time $L/c_o$:
\begin{align}
\nonumber
{u}_{\rm rec}\big( t, \bx_m;\by \big) =&
\frac{1}{(2 \pi)^2 \rho_0^2} \int_\RR \int_{\RR^2} \hat{\cal G}(\omega,L , \by,\bx_m+\bx') \\
& \times \exp \Big( - \frac{|\bx'|^2}{2 \rho_0^2} -i \omega 
\big( \frac{L}{c_o} + t\big)   \Big)
\hat{f}(\omega)  d\bx' d\omega .
\label{eq:urect}
\end{align}

In the second step of the experiment, 
the TRM is used as an array of sources. It emits
the time-reversed recorded field. We observe the field around the original source location
$(\by,L)$ and around the expected arrival time $L/c_o$ to study the wave refocusing:
\begin{align}
\nonumber
{u}_{\rm tr} (t ,\bx ;\by ) =&\frac{1}{2 \pi} \int_\RR \int_{\RR^2}
\hat{u}_{\rm em} (\omega, \bx;\bx_m)  \exp \Big( - \frac{|\bx_m|^2}{R_{\rm m}^2}  \Big)
\overline{ \hat{u}_{\rm rec}(\omega,\bx_m; \by) } \\
& \times \exp \Big(   -i \omega 
\big( \frac{L}{c_o} + t\big)   \Big) d\bx_m d\omega ,
\label{eq:utrt0}
\end{align}
with $\hat{u}_{\rm em} (\omega, \bx;\bx_m) $ defined by (\ref{def:uem}) 
and $ \hat{u}_{\rm rec}(\omega,\bx_m; \by)$ being the Fourier transform of ${u}_{\rm rec}( t, \bx_m;\by)$
given by (\ref{eq:urect}).
We aim at characterizing the 
statistical stability of the refocused wave, in terms of the number of elements
of the TRM and in terms of bandwidth of the pulse, as well as the refocusing resolution.

We consider the case when the bandwidth $B$ of $f$ is smaller than its central frequency $\omega_0$, 
for instance, when 
 the source is a modulated Gaussian with central frequency $\omega_0$ and bandwidth $B$:
\begin{eqnarray}
\label{eq:modgaus}
\hat{f}(\omega) &=& \frac{\sqrt{2\pi}}{B } \exp \Big(- \frac{(\omega-\omega_0)^2}{2 B^2} \Big)  .
\end{eqnarray}

The time-reversed field observed in the plane $z=L$ around the expected arrival time $L/c_o$ can   be expressed as
\begin{eqnarray}
\nonumber
{u}_{\rm tr}(t,\bx ;\by ) &=& \frac{4 \omega_0^2}{2\pi c_o^2} K_0
\int_\RR \iint_{\RR^2 \times \RR^2}
\exp \Big( - \frac{|\bx'|^2}{r_0^2} - \frac{|\by'|^2}{4 \rho_0^2} -i \omega t \Big)  \overline{\hat{f}(\omega)} \\
&& \times
\hat{\cal G}\big(\omega,L,\bx,\bx'+\frac{\by'}{2}\big) 
\overline{\hat{\cal G}\big(\omega,L, \by, \bx'-\frac{\by'}{2} \big) }
  d\bx' d\by' d\omega  ,
  \label{eq:base:t}
\end{eqnarray}
with $K_0 = [c_o^2(r_0^2 -\rho_0^2)]/[16 \pi  \omega_0^2 \rho_0^2 r_0^2]$.
Corresponding to the situation above we will take $K_0=1$ below.

The goal of the forthcoming analysis is to quantity the refocusing properties in terms of resolution and stability, and 
to clarify the role of the source bandwidth as well as the parameters of the TRM.

\section{The Green's Function in the White-noise Paraxial Regime}
\label{sec:3}
In the white-noise paraxial regime the Green's function $\hat{\cal G}$ is of the form \cite{garniers4}
$$
\hat{\cal G} (\omega,L,\bx, \by\big)  = \frac{i c_o}{2\omega} e^{i \frac{\omega}{c_o} L}  \hat{G}(\omega, L,\bx,\by) ,
$$
where $\omega/c_o$ is the homogeneous wavenumber and 
 the function $\hat{G}$ is the solution of the It\^o-Schr\"odinger equation 
 introduced in (\ref{eq:model0}):
\begin{equation}
\label{eq:model}
 d \hat{G}(\omega,z,\bx,\by)   =     
          \frac{ i c_o}{2\omega} \Delta_{\bx}   \hat{G}(\omega,z,\bx,\by) dz
   +   \frac{i \omega}{2c_o}   \hat{G} (\omega,z,\bx,\by ) \circ  d{B}(z,\bx) 
  , 
\end{equation}
with the initial condition in the plane $z=0$:
$ \hat{G} (\omega,z= 0,\bx,\by )  = \delta(\bx-\by)$.

In this context, the time-reversed field (\ref{eq:base})
observed at $(\bx,L)$ when the original source is at $(\by,L)$
 in the wave refocusing experiment of Section \ref{sec:1th} is
\begin{equation}
\hat{u}_{\rm tr}( \bx ;\by ) =  
\iint_{\RR^2\times \RR^2}  \!
\exp \Big( - \frac{|\bx'|^2}{r_0^2} - \frac{|\by'|^2}{4 \rho_0^2} \Big) 
\hat{G}\big(\omega,L,\bx,\bx'+\frac{\by'}{2}\big) 
\overline{\hat{G}\big(\tilde{\omega},L,\by,\bx'-\frac{\by'}{2}\big) }
  d\bx' d\by'   .
\end{equation}
The mean time-reversed field is
\begin{equation}
{\cal M}_1(\bx;\by) = \EE \big[ \hat{u}_{\rm tr}( \bx ;\by )\big]   ,
\end{equation}
and it can be expressed as
\begin{eqnarray*}
{\cal M}_1(\bx;\by) &=&
\iint_{\RR^2 \times \RR^2}
\exp \Big( - \frac{|\bx'|^2}{r_0^2} - \frac{|\by'|^2}{4 \rho_0^2} \Big)  \\
&& \times
\EE \Big[ \hat{G}\big(\omega,L,\bx,\bx'+\frac{\by'}{2}\big) 
\overline{\hat{G}\big(\tilde{\omega},L,\by,\bx'-\frac{\by'}{2}\big) } \Big]
  d\bx' d\by'   .
\end{eqnarray*}
The covariance function of the time-reversed field is
\begin{eqnarray}
{\cal M}_2(\bx,\tilde{\bx};\by) &=& \EE \big[ \hat{u}_{\rm tr}( \bx ;\by )
\overline{\hat{u}_{\rm tr}( \tilde{\bx} ;\by )} \big] ,
\end{eqnarray}
and it can be expressed as 
\begin{eqnarray*}
\nonumber
{\cal M}_2(\bx,\tilde{\bx};\by) &=&
\iint_{\RR^2 \times \RR^2}
\exp \Big( - \frac{|\bx_1'|^2+|\bx_2'|^2}{r_0^2} - \frac{|\by_1'|^2+|\by_2'|^2}{4 \rho_0^2} \Big) \\
&& \hspace*{-0.6in}
\times \EE \Big[ \hat{G}\big(\omega,L,\bx,\bx_1'+\frac{\by_1'}{2}\big) \hat{G}\big(\tilde{\omega},L,\by,\bx_2'+\frac{\by_2'}{2}\big) \\
&& \hspace*{-0.6in} \times
\overline{\hat{G}\big(\tilde{\omega},L,\by,\bx_1'-\frac{\by_1'}{2}\big) } \overline{\hat{G}\big(\omega,L,\tilde{\bx},\bx_2'-\frac{\by_2'}{2}\big) } \Big]
  d\bx_1' d\by_1' 
  d\bx_2' d\by_2' .
\end{eqnarray*}
These expressions show that we need to study the second- and fourth-order moments
of the random paraxial Green's functions at different frequencies.


\section{The Moments of the Green's Function}
\label{sec:4}
This section contains the detailed analysis of the second- and fourth-order moments 
that are needed to study the time-reversed field. 

  
\subsection{The Second-order Moment}\label{sec:mom2}
Let us consider two frequencies $\omega, \tilde{\omega}$.
We consider the second-order moment:
\begin{eqnarray}
\nonumber
M_1(z,\bx,\by) &=&
\iint_{\RR^2 \times \RR^2}
\exp \Big( - \frac{|\bx'|^2}{r_0^2} - \frac{|\by'|^2}{4 \rho_0^2} \Big)  \\
&&\times
\EE \Big[ \hat{G}\big(\omega,z,\bx,\bx'+\frac{\by'}{2}\big) 
\overline{\hat{G}\big(\tilde{\omega},z,\by,\bx'-\frac{\by'}{2}\big) } \Big]
  d\bx' d\by'   .
  \label{def:mom1}
\end{eqnarray}
$M_1$ satisfies the system:
\begin{equation}
\label{eq:mom11}
\frac{\partial M_{1} }{\partial z} = \frac{i c_o}{2}  \big(  \frac{1}{\omega} \Delta_{\bx}
- \frac{1}{\tilde{\omega}} \Delta_{\by} \big) M_{1}  + \frac{1}{8 c_o^2} \big( 2 \omega \tilde{\omega}C(\bx-\by)- (\omega^2+\tilde{\omega}^2)C({\bf 0})  \big) M_{1}    ,
\end{equation}
starting from 
$$
 M_{1} (z=0,\bx,\by) = \exp \Big( - \frac{|\bx+\by|^2}{4 r_0^2} - \frac{|\bx-\by|^2}{ 4 \rho_0^2} \Big)  .
$$
 
\subsection{The Fourt-order Moment}\label{sec:mom4}
Let us consider four frequencies $\omega_1,\omega_2,\tilde{\omega}_1$, $\tilde{\omega}_2$.
We consider the fourth-order moment
\begin{eqnarray}
\nonumber
 M_2(z,\bx_1,\bx_2,\by_1,\by_2) &=&
\iint_{\RR^2 \times \RR^2 \times \RR^2 \times \RR^2}
\exp \Big( - \frac{|\bx_1'|^2+|\bx_2'|^2}{r_0^2} - \frac{|\by_1'|^2+|\by_2'|^2}{4 \rho_0^2} \Big) \\
\nonumber
&& \hspace*{-0.6in}
\times \EE \Big[ \hat{G}\big(\omega_1,z,\bx_1,\bx_1'+\frac{\by_1'}{2}\big) \hat{G}\big(\omega_2,z,\bx_2,\bx_2'+\frac{\by_2'}{2}\big) \\
&&\hspace*{-0.6in} \times
\overline{\hat{G}\big(\tilde{\omega}_1,z,\by_1,\bx_1'-\frac{\by_1'}{2}\big) } \overline{\hat{G}\big(\tilde{\omega}_2 ,z,\by_2,\bx_2'-\frac{\by_2'}{2}\big) } \Big]
  d\bx_1' d\by_1' 
  d\bx_2' d\by_2' .
\label{def:mom2}
\end{eqnarray}
It satisfies
\begin{eqnarray}
\nonumber
\frac{\partial M_2}{\partial z} &=& \frac{i c_o}{2}  \Big( \frac{1}{\omega_1} \Delta_{\bx_1}
+ \frac{1}{\omega_2}  \Delta_{\bx_2}
- \frac{1}{\tilde{\omega}_1}   \Delta_{\by_1}- \frac{1}{\tilde{\omega}_2} \Delta_{\by_2} \Big) 
M_{2} \\
&& + \frac{1}{4c_o^2} U_{2} \big( \bx_1,\bx_2,\by_1,\by_2 \big)
 M_{2}  ,
 \label{def:generalmoment}
\end{eqnarray}
with the generalized potential
\begin{eqnarray}
\nonumber
 U_{2}\big(  \bx_1,\bx_2,\by_1,\by_2 \big)  &=&
\omega_1 \tilde{\omega}_1 {C}(\bx_1-\by_1) +
\omega_1 \tilde{\omega}_2 {C}(\bx_1-\by_2) +
\omega_2 \tilde{\omega}_1 {C}(\bx_2-\by_1) \\
\nonumber
&&+ 
\omega_2 \tilde{\omega}_2{C}(\bx_2-\by_2) -  \omega_1 \omega_2 {C}( \bx_1-\bx_2)
-\tilde{\omega}_1 \tilde{\omega}_2 {C}( \by_1-\by_2) \\
&& -
\frac{\omega_1^2+\omega_2^2 +\tilde{\omega}_1^2 +\tilde{\omega}_2^2}{2}
{C}({\bf 0}) \, ,
\end{eqnarray}
and it starts from
\begin{eqnarray*}
 M_{2} (z=0,  \bx_1,\bx_2,\by_1,\by_2) &=&
 \exp \Big( - \frac{|\bx_1+\by_1|^2+|\bx_2+\by_2|^2}{4 r_0^2} \Big)\\
 && \times \exp \Big(- \frac{|\bx_1-\by_1|^2+|\bx_2-\by_2|^2}{ 4 \rho_0^2} \Big) .
\end{eqnarray*}

\subsection{The Scintillation Regime}
\label{sec:regime}%
In this paper we  address a regime which can be considered
as a particular case of the paraxial white-noise regime:  the scintillation regime. 
The scintillation regime is valid if  the correlation length of the medium 
(i.e., the transverse correlation length of the Brownian field $B$)
is smaller than the radius of the TRM and the size of the TRM elements.
If the correlation length is our reference length, this means that 
in this regime the covariance function $C^\eps$,
the radius of the TRM $r_0^\eps$, the TRM element size $\rho_0^\eps$,
and the propagation distance $L^\eps$
are of the form
\begin{equation}
\label{sca:sci}
C^\eps(\bx)= \eps C (\bx) , \quad \quad r_0^\eps = \frac{r_0}{\eps}, \quad \quad \rho_0^\eps = \frac{\rho_0}{\eps},
\quad \quad L^\eps = \frac{L}{\eps} .
\end{equation}
Here $\eps$ is a small dimensionless parameter and we will study the limit $\eps \to 0$.

Note that in Subsection \ref{sec:regime:2} we will address a slightly different version of the scintillation regime, which is (\ref{sca:sci}) in which the size of the TRM elements
is of the same order as the correlation length of the medium: $\rho_0^\eps=\rho_0$.

\subsection{The Second-order Moment in the Scintillation Regime}
Let us consider the second-order moment (\ref{def:mom1}) in the scintillation regime (\ref{sca:sci}).
We assume that the two frequencies are close to each other and
we parameterize them as
\begin{eqnarray*}
\omega = \omega_0+\eps \Omega  ,
\quad \quad 
\tilde{\omega} =\omega_0- \eps \Omega .
\end{eqnarray*}
We parameterize the two points $\bx$ and $\by$ as
 $$
 \br= \frac{\bx+\by}{2} , \quad \quad \bq = \bx-\by, 
 $$
 We consider a long propagation distance of the form $z/\eps$.

In the variables $(z/\eps,\bq,\br)$ 
the function $M_{1}^\eps$ satisfies the equation:
\begin{eqnarray}
\nonumber
\frac{\partial M_{1}^\eps}{\partial z} &=& 
\frac{i c_o}{\omega_0 \eps}   \nabla_{\br}\cdot \nabla_{\bq}
 M_{1}^\eps 
 -\frac{i c_o \Omega}{\omega_0^2}  \big( \frac{1}{4} \Delta_\br +\Delta_\bq\big)
 M_{1}^\eps
+ \frac{\omega_0^2}{4 c_o^2} \big( {C}( \bq) -  {C}({\bf 0})  \big) M_{1}^\eps   ,
\label{eq:M10}
\end{eqnarray}
starting from
$$
M_{1}^\eps(z=0,\bq,\br) = \exp \Big( - \eps^2 \frac{  |\br|^2}{r_0^2}  
- \eps^2 \frac{ |\bq|^2}{4\rho_0^2} \Big) ,
$$
and where we have not written terms of order $\eps$.
The Fourier transform (in $\bq$ and $\br$) of the second-order moment
of the paraxial Green's function 
is defined by:
\begin{eqnarray}
\hat{M}_{1}^\eps\big( \frac{z}{\eps} ,\bxi,\bzeta \big) 
&=& 
\iint_{\RR^2\times \RR^2} M_{1}^\eps\big( \frac{z}{\eps} ,\bq,\br\big) 
\exp  \big(- i\bq \cdot \bxi- i\br\cdot \bzeta \big) d\br d\bq 
. 
\end{eqnarray}
It satisfies
\begin{eqnarray*}
\nonumber
 \frac{\partial \hat{M}^\eps_1}{\partial z} 
&=&
- \frac{i c_o}{\omega_0 \eps} \bxi \cdot \bzeta  \hat{M}^\eps_1+
\frac{ic_o \Omega}{\omega_0^2}
 \big( \frac{1}{4} |\bzeta|^2+ |\bxi|^2 
\big)  \hat{M}^\eps_1 \\
&&+\frac{\omega_0^2}{4 (2\pi)^2 c_o^2} 
\int_{\RR^2} \hat{C}(\bk) \Big[  - \hat{M}^\eps_1(  \bxi ,\bzeta )  +
\hat{M}^\eps_1 (  \bxi-\bk , \bzeta)    \Big] d \bk .
\end{eqnarray*}
Let us absorb the rapid phase in the function
\begin{equation}
\label{eq:renormhatM1}
\widetilde{M}^\eps_1 \big( \frac{z}{\eps} ,\bxi,\bzeta \big) = 
\hat{M}_1^\eps \big( \frac{z}{\eps} ,\bxi ,\bzeta\big)
 \exp \Big( \frac{i c_o z}{\omega_0 \eps} \bxi \cdot \bzeta  \Big)
 .
\end{equation}
In the scintillation regime the rescaled function $\widetilde{M}^\eps_1$  satisfies the equation with fast phases
\begin{eqnarray}
\nonumber
 \frac{\partial \widetilde{M}^\eps_1}{\partial z} 
&=&
\frac{ic_o \Omega}{\omega_0^2}
 \big( \frac{1}{4} |\bzeta|^2+ |\bxi|^2 
\big)  \widetilde{M}^\eps_1 \\
&&+\frac{\omega_0^2}{4 (2\pi)^2 c_o^2} 
\int_{\RR^2} \hat{C}(\bk) \Big[  -  \widetilde{M}^\eps_1(  \bxi ,\bzeta )  +
\widetilde{M}^\eps_1 (  \bxi-\bk , \bzeta) 
e^{i\frac{c_o z}{\eps \omega_0} \bk \cdot \bzeta}  \Big] d \bk ,
\label{eq:tildeN1eps}
\end{eqnarray}
starting from 
\begin{equation}
\label{eq:initialtildeM1eps}
\widetilde{M}^\eps_1(z=0,\bxi ,  \bzeta) = (2\pi)^4 \phi^\eps_{\sqrt{2} \rho_0} ( \bxi )
\phi^\eps_{r_0/\sqrt{2} } ( \bzeta  ) ,
\end{equation}
where we have denoted
\begin{equation}
\label{def:phiepsrho}
\phi^\eps_{\rho}(\bxi) = \frac{\rho^2}{2\pi \eps^2} \exp \Big( -\frac{\rho^2}{2 \eps^2} |\bxi|^2\Big) .
\end{equation}
Note that $\phi^\eps_{\rho}$ belongs to $L^1$ and has a $L^1$-norm equal to one,
and that it behaves like a Dirac distribution as $\eps \to 0$.

\begin{proposition}
\label{prop:1}%
The function $\widetilde{M}^\eps_1(z/\eps,\bxi , \bzeta  ) $ defined by (\ref{eq:renormhatM1}) can be expanded as
\begin{eqnarray}
\nonumber
 \widetilde{M}^\eps_1\big( \frac{z}{\eps} ,\bxi , \bzeta \big)  &=&
K(z)
\phi^\eps_{\sqrt{2} \rho_0} ( \bxi )
\phi^\eps_{r_0/\sqrt{2} } ( \bzeta  ) \\
&&+
\phi^\eps_{r_0/\sqrt{2} } ( \bzeta  ) 
A\big(z, \bxi  ,\frac{\bzeta}{\eps},\Omega\big) 
 + R^\eps_1 (z ,\bxi  ,  \bzeta )   ,
 \label{def:R1eps}
\end{eqnarray}
where  the function $K$ is defined by
\begin{equation}
\label{def:K}
K(z) = (2\pi)^4 \exp\Big(- \frac{\omega_0^2}{4 c_o^2} C({\bf 0}) z\Big) , 
\end{equation}
the function $(z,\bxi)\mapsto A(z,\bxi,\bzeta,\Omega)$ is the solution of 
\begin{eqnarray}
\nonumber
\partial_z A &=&  \frac{i c_o\Omega}{\omega_0^2} |\bxi|^2 A +\frac{\omega_0^2}{4(2\pi)^2c_o^2}
\int_{\RR^2} \hat{C}(\bk) \big[ A(\bxi-\bk) e^{\frac{i c_oz}{\omega_0} \bk \cdot \bzeta} -A(\bxi)\big] d\bk\\
&&
+\frac{\omega_0^2}{4  (2\pi)^2 c_o^2} K(z) \hat{C}(\bxi) e^{\frac{i c_oz}{\omega_0} \bxi \cdot \bzeta} ,
\label{def:A}
\end{eqnarray}
starting from $A(z=0,\bxi,\bzeta,\Omega)=0$,
and the function $R^\eps_1 $ satisfies
\begin{equation}
\sup_{z \in [0,Z]} \| R^\eps_1 (z,\cdot,\cdot) \|_{L^1(\RR^2\times \RR^2 )} 
\stackrel{\eps \to 0}{\longrightarrow}  0  ,
\end{equation}
for any $Z>0$.
\end{proposition}

\noindent
\begin{proof}
We introduce
\begin{eqnarray*}
\check{M}_1^\eps(z,\bxi,\bzeta) &=&
\widetilde{M}^\eps_1\big( \frac{z}{\eps} ,\bxi , \bzeta \big) 
\exp\Big( -i \frac{c_o\Omega}{\omega_0^2}\big(\frac{1}{4} |\bzeta|^2 + |\bxi|^2\big) z \Big) ,
\\
\check{A}^\eps(z,\bxi,\bzeta)&=&
A\big(z,\bxi,\frac{\bzeta}{\eps},\Omega) \exp\Big( -i \frac{c_o\Omega}{\omega_0^2}|\bxi|^2 z \Big) .
\end{eqnarray*}
We first note that, for any $\bzeta$, we have using Bochner's theorem
$$
\partial_z \| \check{A}^\eps(z,\cdot,\bzeta )\|_{L^1} \leq \frac{\omega_0^2}{2c_o^2} C({\bf 0}) \|\check{A^\eps}(z,\cdot,\bzeta )\|_{L^1}
+ \frac{\omega_0^2}{4 c_o^2} K(z) C({\bf 0}) ,
$$
which shows by Gronwall's lemma that
$$
\sup_{z\in [0,Z], \bzeta \in \RR^2} \partial_z \|\check{A^\eps}(z,\cdot,\bzeta)\|_{L^1} <  \infty.
$$
If we define the operator $\check{\cal L}^\eps$ from $L^1(\RR^2\times \RR^2)$ to $L^1(\RR^2\times \RR^2)$
$$
[ \check{\cal L}^\eps \check{M} ](\bxi,\bzeta) = 
\frac{\omega_0^2}{4(2\pi)^2 c_o^2}
\int_{\RR^2} \hat{C}(\bk)
\big[ 
\check{M}(\bxi-\bk,\bzeta) e^{ i \frac{c_o\Omega}{\omega_0^2} (|\bk|^2-2\bxi\cdot\bk) z + i \frac{c_o}{\eps \omega_0} \bk \cdot \bzeta z}
-
\check{M}(\bxi,\bzeta)\big]d\bk,
$$
whose norm is bounded by $\| \check{\cal L}^\eps\|_{L^1\to L^1} \leq \frac{\omega_0^2}{2c_o^2} C({\bf 0})$,
then we get from (\ref{eq:tildeN1eps}) that
$\check{M}_1^\eps$ satisfies the equation
$$
\partial_z \check{M}_1^\eps =\check{\cal L} ^\eps \check{M}_1^\eps .
$$
Denoting
\begin{eqnarray*}
\check{R}^\eps(z,\bxi,\bzeta) &=&
\check{M}_1^\eps(z,\bxi,\bzeta)  -
\check{N}^\eps(z,\bxi,\bzeta) , \\
\check{N}^\eps(z,\bxi,\bzeta) &=&
K(z)
\phi^\eps_{\sqrt{2} \rho_0} ( \bxi )
\phi^\eps_{r_0/\sqrt{2} } ( \bzeta  ) +
\phi^\eps_{r_0/\sqrt{2} } ( \bzeta  ) 
\check{A}^\eps (z, \bxi  ,\bzeta) ,
\end{eqnarray*}
we have
\begin{equation}
\label{eq:Repsproof}
\partial_z \check{R}^\eps = \check{\cal L}^\eps \check{R}^\eps +\check{S}^\eps  ,
\end{equation}
with
$$
\check{S}^\eps(z,\bxi,\bzeta)  = - \partial_z \check{N}^\eps(z,\bxi,\bzeta)  + \check{\cal L}^\eps \check{N}^\eps(z,\bxi,\bzeta) .
$$
The function $\check{S}^\eps$ is equal to
\begin{eqnarray*}
\check{S}^\eps(z,\bxi,\bzeta)  &=&
\frac{\omega_0^2}{4 (2\pi)^2 c_o^2} K(z) \phi^\eps_{r_0/\sqrt{2}}(\bzeta) 
e^{i \frac{c_o}{\omega_0 \eps} \bxi\cdot\bzeta z - i \frac{c_o \Omega}{\omega_0^2} |\bxi|^2 z} \\
&&\times
\Big[ 
\int_{\RR^2} \hat{C}(\bxi-\bk) \phi^\eps_{\sqrt{2} \rho_0} ( \bk )
e^{-i \frac{c_o}{\omega_0 \eps} \bk\cdot \bzeta z + i \frac{c_o \Omega}{\omega_0^2} |\bk|^2 z}
d\bk
- \hat{C}(\bxi)\Big] .
\end{eqnarray*}
Its $L^1$-norm can be evaluated as follows for  $ z\in [0, Z]$:
\begin{eqnarray*}
&&\| \check{S}^\eps(z,\cdot,\cdot)\|_{L^1}\\
&&=\frac{\omega_0^2K(z)}{4 (2\pi)^2 c_o^2} 
\iint d\bzeta d\bxi \phi^1_{r_0/\sqrt{2}} (\bzeta) 
 \Big| \int \hat{C}(\bxi-\eps \bk) \phi^1_{\sqrt{2} \rho_0} ( \bk )
e^{-i \eps \frac{c_o}{\omega_0} \bk\cdot \bzeta z + i \eps^2 \frac{c_o \Omega}{\omega_0^2} |\bk|^2 z}
d\bk
- \hat{C}(\bxi)\Big| \\
&&\leq \frac{\omega_0^2K(z)}{4 (2\pi)^2 c_o^2} \iiint d\bzeta d\bxi d\bk \phi^1_{r_0/\sqrt{2}} (\bzeta)  \phi^1_{\sqrt{2} \rho_0} ( \bk )
 \Big| \hat{C}(\bxi-\eps \bk)  
e^{-i \eps \frac{c_o}{\omega_0} \bk\cdot \bzeta z + i \eps^2 \frac{c_o \Omega}{\omega_0^2} |\bk|^2 z}
- \hat{C}(\bxi)\Big| \\
&&\leq\frac{\omega_0^2K(z)}{4 (2\pi)^2 c_o^2} 
\iint  d\bzeta  d\bk  \phi^1_{r_0/\sqrt{2}} (\bzeta) \phi^1_{\sqrt{2} \rho_0} ( \bk )
\Big[ \int \big| \hat{C}(\bxi-\eps \bk)  -\hat{C}(\bxi)\big| d\bxi\Big] \\
&&\quad + 
\frac{\omega_0^2 K(z)}{4 (2\pi)^2 c_o^2} 
\iiint d\bzeta d\bxi d\bk \phi^1_{r_0/\sqrt{2}} (\bzeta)  \phi^1_{\sqrt{2} \rho_0} ( \bk )
\hat{C}(\bxi) \big| e^{-i \eps \frac{c_o}{\omega_0} \bk\cdot \bzeta z + i \eps^2 \frac{c_o \Omega}{\omega_0^2} |\bk|^2 z} -1\big| \\
&&\leq\frac{\omega_0^2  (2\pi)^2}{4  c_o^2} 
\int   d\bk   \phi^1_{\sqrt{2} \rho_0} ( \bk )
\Big[ \int \big| \hat{C}(\bxi-\eps \bk)  -\hat{C}(\bxi)\big| d\bxi\Big] \\
&&\quad + 
\frac{(2\pi)^4 \omega_0^2 C({\bf 0})Z}{4  c_o^2} 
\iint d\bzeta  d\bk \phi^1_{r_0/\sqrt{2}} (\bzeta)  \phi^1_{\sqrt{2} \rho_0} ( \bk )
\big( \eps \frac{c_o}{\omega_0} |\bk| |\bzeta| +  \eps^2 \frac{c_o \Omega}{\omega_0^2} |\bk|^2  \big).
\end{eqnarray*}
The term within  the square brackets is bounded by $2(2\pi)^2 C({\bf 0})$ 
and goes to zero as $\eps \to 0$ for any $\bk$ (because $\hat{C}\in L^1$ and $\hat{C}$ is continuous,
as it is the inverse Fourier transform of an $L^1$-function), 
so the first term of the right-hand side goes to zero as $\eps \to 0$ by Lebesgue's dominated convergence theorem.
The second term of the right-hand side is of order $\eps$ and it goes to zero as $\eps \to 0$.
As a result,
$$
\sup_{z\in [0,Z]} \| \check{S}^\eps(z,\cdot,\cdot)\|_{L^1} \stackrel{\eps \to 0}{\longrightarrow} 0 .
$$
Integrating (\ref{eq:Repsproof}) and taking the $L^1$-norm, we find that for any $z\in [0,Z]$:
$$
\|  \check{R}^\eps(z,\cdot,\cdot)\|_{L^1}  \leq  \frac{\omega_0^2}{2c_o^2} C({\bf 0})
\int_0^z \|  \check{R}^\eps(z',\cdot,\cdot)\|_{L^1} dz' + 
\int_0^z \| \check{S}^\eps(z',\cdot,\cdot)\|_{L^1} dz'   .
$$
Applying Gronwall's lemma gives:
\begin{equation}
\label{eq:checkRepsproof}
\sup_{z\in [0,Z]} \|  \check{R}^\eps(z,\cdot,\cdot)\|_{L^1} \stackrel{\eps \to 0}{\longrightarrow} 0 .
\end{equation}
Finally, the residual  $R_1^\eps$ defined by (\ref{def:R1eps}) can be expressed as
\begin{eqnarray*}
R_1^\eps (z,\bxi,\bzeta)&=& 
\check{R}^\eps(z,\bxi,\bzeta)
\exp\Big( i \frac{c_o\Omega}{\omega_0^2}\big(\frac{1}{4} |\bzeta|^2 + |\bxi|^2\big) z \Big) \\
&&+K(z)
\phi^\eps_{\sqrt{2} \rho_0} ( \bxi )
\phi^\eps_{r_0/\sqrt{2} } ( \bzeta  )
\Big[ \exp\Big( i \frac{c_o\Omega}{\omega_0^2}\big(\frac{1}{4} |\bzeta|^2 + |\bxi|^2\big) z \Big)-1\Big] \\
&&+\phi^\eps_{r_0/\sqrt{2} } ( \bzeta  ) 
A\big(z, \bxi  ,\frac{\bzeta}{\eps},\Omega\big) 
\Big[ \exp\Big( i \frac{c_o\Omega}{\omega_0^2} \frac{1}{4} |\bzeta|^2  z \Big)-1\Big] .
\end{eqnarray*}
The $L^1$-norm in $(\bxi,\bzeta)$ of the first term of the right-hand side goes to zero by (\ref{eq:checkRepsproof}).
The $L^1$-norm of the second term is
$$
\iint d\bxi d\bzeta K(z)
\phi^1_{\sqrt{2} \rho_0} ( \bxi )
\phi^1_{r_0/\sqrt{2} } ( \bzeta  )
\Big| \exp\Big( i \frac{c_o\Omega \eps^2}{\omega_0^2}\big(\frac{1}{4} |\bzeta|^2 + |\bxi|^2\big) z \Big)-1\Big| 
$$
which is bounded by 
$$
\frac{ (2\pi)^4 c_o|\Omega| \eps^2}{\omega_0^2} Z \iint d\bxi d\bzeta 
\phi^1_{\sqrt{2} \rho_0} ( \bxi )
\phi^1_{r_0/\sqrt{2} } ( \bzeta  )
\big(\frac{1}{4} |\bzeta|^2 + |\bxi|^2\big) 
$$
which goes to zero as $\eps \to 0$.
The $L^1$-norm  of the third term is
$$
\iint d\bxi d\bzeta 
\phi^1_{r_0/\sqrt{2} } ( \bzeta  )
|A(z,\bxi,\bzeta)| \Big| \exp\Big( i \frac{c_o\Omega \eps^2}{\omega_0^2} \frac{1}{4} |\bzeta|^2  z \Big)-1\Big| 
$$
which is bounded by 
$$
\frac{c_o|\Omega| \eps^2}{4 \omega_0^2} Z \iint d\bxi d\bzeta 
\phi^1_{r_0/\sqrt{2} } ( \bzeta  )
|A(z,\bxi,\bzeta)|
  |\bzeta|^2 ,
$$
which goes to zero as $\eps \to 0$  because $\sup_{z\in [0,Z],\bzeta \in \RR^2} \| A(z,\cdot,\bzeta)\|_{L^1}$ is bounded.
This completes the proof of the proposition.
\end{proof}

We remark that $A$ defined by (\ref{def:A}) describes how energy is transferred from 
the coherent part to the incoherent part of the wave field and also  in between different lateral slowness modes. 
The first term in the right-hand side of (\ref{def:A}) captures the decorrelation due to frequency separation,   the second term captures 
 random forward scattering and transfer of incoherent energy  between different lateral slowness modes, and the third term
captures transfer of energy from the coherent part to the scattered  part of the wave field. 

\subsection{The Fourth-order Moment in the Scintillation Regime}
Let us consider the fourth-order moment (\ref{def:mom2})  in the scintillation regime (\ref{sca:sci}).
We assume that the four frequencies are close to each other and
we parameterize them as
\begin{eqnarray*}
\omega_1 = \omega_0+\eps (\Omega_1 + \Omega_2  +\Omega_3),
\quad \quad 
\omega_2 =\omega_0+\eps(-\Omega_1 + \Omega_2  -\Omega_3),\\
\tilde{\omega}_1 =\omega_0+\eps(\Omega_1 - \Omega_2  -\Omega_3), \quad \quad 
\tilde{\omega}_2 =\omega_0+\eps(-\Omega_1 - \Omega_2  +\Omega_3).
\end{eqnarray*}
We parameterize  the four points 
$\bx_1,\bx_2,\by_1,\by_2$ in (\ref{def:generalmoment}) in the special way:
\begin{eqnarray*}
\bx_1 = \frac{\br_1+\br_2+\bq_1+\bq_2}{2}, \quad \quad 
\by_1 = \frac{\br_1+\br_2-\bq_1-\bq_2}{2}, \\
\bx_2 = \frac{\br_1-\br_2+\bq_1-\bq_2}{2}, \quad \quad 
\by_2 = \frac{\br_1-\br_2-\bq_1+\bq_2}{2}.
\end{eqnarray*}
We consider a long propagation distance of the form $z/\eps$.
In the variables $(z/\eps,\bq_1,\bq_2,\br_1,\br_2)$ 
the function $M_{2}^\eps$ satisfies the equation:
\begin{eqnarray}
\nonumber
\frac{\partial M_{2}^\eps}{\partial z} &=& 
\frac{i c_o}{\omega_0 \eps} \big( \nabla_{\br_1}\cdot \nabla_{\bq_1}
+
 \nabla_{\br_2}\cdot \nabla_{\bq_2}
\big)  M_{2}^\eps 
-\frac{ic_o \Omega_1}{\omega_0^2}
 \big( \nabla_{\br_1}\cdot \nabla_{\bq_2}
+
 \nabla_{\br_2}\cdot \nabla_{\bq_1}
\big)  M_{2}^\eps\\
\nonumber
&&
-\frac{ic_o \Omega_2}{2\omega_0^2}
 \big( \Delta_{\br_1}+\Delta_{\br_2}
+
 \Delta_{\bq_2}+\Delta_{\bq_2}
\big)  M_{2}^\eps 
-\frac{ic_o \Omega_3}{\omega_0^2}
 \big( \nabla_{\br_1}\cdot \nabla_{\br_2}
+
 \nabla_{\bq_1}\cdot \nabla_{\bq_2}
\big)  M_{2}^\eps\\
&&
+ \frac{\omega_0^2}{4 c_o^2} {\cal U}_{2}(\bq_1,\bq_2,\br_1,\br_2) M_{2}^\eps   ,
\label{eq:M20}
\end{eqnarray}
with the generalized potential
\begin{eqnarray}
\nonumber
{\cal U}_{2}(\bq_1,\bq_2,\br_1,\br_2) &=& 
{C}(\bq_2+\bq_1) 
+
{C}(\bq_2-\bq_1) 
+
{C}(\br_2+\bq_1) 
+
{C}(\br_2-\bq_1) \\
&&- 
{C}( \bq_2+\br_2) - {C}( \bq_2-\br_2) - 2 {C}({\bf 0})  ,
\end{eqnarray}
and where we have not written terms of order $\eps$.
The initial condition for Eq.~(\ref{eq:M20}) is
$$
M_{2}^\eps(z=0,\bq_1,\bq_2,\br_1,\br_2) = \exp \Big( - \eps^2 \frac{  |\br_1|^2+ |\br_2|^2}{2r_0^2}
- \eps^2 \frac{ |\bq_1|^2+ |\bq_2|^2}{2\rho_0^2} \Big).
$$
The Fourier transform (in $\bq_1$, $\bq_2$, $\br_1$, and $\br_2$) of the fourth-order moment
of the paraxial Green's function 
is defined by:
\begin{eqnarray}
\nonumber
\hat{M}_{2}^\eps\big( \frac{z}{\eps} ,\bxi_1,\bxi_2,\bzeta_1,\bzeta_2\big) 
&=& 
\iint_{\RR^2\times \RR^2 \times \RR^2\times \RR^2} M_{2}^\eps\big( \frac{z}{\eps} ,\bq_1,\bq_2,\br_1,\br_2\big) \\
&& \hspace*{-0.4in}
\times
\exp  \big(- i\bq_1 \cdot \bxi_1- i\br_1\cdot \bzeta_1- i\bq_2 \cdot \bxi_2- i\br_2\cdot \bzeta_2\big) d\br_1d\br_2 d\bq_1d\bq_2
. \hspace*{0.3in} 
\end{eqnarray}
Let us absorb the rapid phase in the function
\begin{equation}
\label{eq:renormhatM2}
\widetilde{M}^\eps_2 \big( \frac{z}{\eps} ,\bxi_1,\bxi_2,\bzeta_1,\bzeta_2\big) = 
\hat{M}_2^\eps \big( \frac{z}{\eps} ,\bxi_1,\bxi_2,\bzeta_1,\bzeta_2\big)
 \exp \Big( \frac{i  c_oz}{\omega_0 \eps} (\bxi_2 \cdot \bzeta_2  +   \bxi_1 \cdot \bzeta_1) \Big)
 .
\end{equation}
In the scintillation regime   (\ref{sca:sci}) the rescaled function $\widetilde{M}^\eps_2$  satisfies the equation with fast phases
\begin{eqnarray}
\nonumber
 \frac{\partial \widetilde{M}^\eps_2}{\partial z} 
&=&
\frac{ic_o \Omega_1}{\omega_0^2}
 \big( \bxi_1\cdot \bzeta_2
+
 \bxi_2\cdot \bzeta_1\big)  \widetilde{M}^\eps_2\\
\nonumber
&&
+\frac{ic_o \Omega_2}{2\omega_0^2}
 \big(|\bxi_1|^2+|\bxi_2|^2 +|\bzeta_1|^2 +|\bzeta_2|^2 
\big)  \widetilde{M}^\eps_2 \\
\nonumber
&&
+\frac{ic_o \Omega_3}{\omega_0^2}
 \big(\bxi_1\cdot \bxi_2 +\bzeta_1\cdot \bzeta_2
\big) \widetilde{M}^\eps_2\\
\nonumber
&&+\frac{\omega_0^2}{4 (2\pi)^2 c_o^2} 
\int_{\RR^2} \hat{C}(\bk) \bigg[  - 2 \widetilde{M}^\eps_2(  \bxi_1 ,\bxi_2, \bzeta_1,\bzeta_2)   \\
\nonumber
&&+
\widetilde{M}^\eps_2 (  \bxi_1-\bk,\bxi_2-\bk,  \bzeta_1,\bzeta_2) 
e^{i\frac{c_o z}{\eps \omega_0} \bk \cdot  (\bzeta_2 + \bzeta_1)} \\
&& \nonumber
+ 
\widetilde{M}^\eps_2  (  \bxi_1-\bk,\bxi_2,  \bzeta_1,\bzeta_2-\bk) 
e^{i\frac{c_o z}{\eps \omega_0} \bk \cdot ( \bxi_2 +  \bzeta_1)}\\
\nonumber
&& +\widetilde{M}^\eps_2 (  \bxi_1+\bk,\bxi_2-\bk,  \bzeta_1,\bzeta_2) 
e^{i\frac{c_o z}{\eps \omega_0} \bk \cdot (\bzeta_2 -   \bzeta_1)}
\\
&& \nonumber
+ 
\widetilde{M}^\eps_2  (  \bxi_1+\bk,\bxi_2,  \bzeta_1,\bzeta_2-\bk) 
e^{i\frac{c_o z}{\eps \omega_0} \bk \cdot (\bxi_2 -  \bzeta_1)}
\\
\nonumber
&&
-
 \widetilde{M}^\eps_2  (  \bxi_1,\bxi_2-\bk, \bzeta_1, \bzeta_2-\bk) 
 e^{i\frac{c_o z}{\eps \omega_0} (  \bk \cdot (\bzeta_2+\bxi_2)-|\bk|^2 )} \\
 &&
- \widetilde{M}^\eps_2  ( \bxi_1,\bxi_2-\bk,  \bzeta_1,\bzeta_2+\bk) 
e^{i\frac{c_o z}{\eps \omega_0} ( \bk \cdot (\bzeta_2-\bxi_2)+|\bk|^2)}
\bigg] d \bk ,
\label{eq:tildeNeps}
\end{eqnarray}
starting from 
\begin{equation}
\label{eq:initialtildeM2eps}
\widetilde{M}^\eps_2 (z=0,\bxi_1,\bxi_2, \bzeta_1, \bzeta_2 ) = (2\pi)^8 \phi^\eps_{\rho_0} ( \bxi_1 )
\phi^\eps_{\rho_0} ( \bxi_2 )
\phi^\eps_{r_0} ( \bzeta_1 )\phi^\eps_{r_0} ( \bzeta_2 ) ,
\end{equation}
where $\phi^\eps_\rho$ is defined by (\ref{def:phiepsrho}).
The following result  shows that $\widetilde{M}^\eps_2$ exhibits a multi-scale behavior
as $\eps \to 0$, with some components evolving at the scale $\eps$ and 
some components evolving at the order one scale.
\begin{proposition}
\label{prop:3}%
The function $\widetilde{M}^\eps_2 (z/\eps,\bxi_1,\bxi_2, \bzeta_1,\bzeta_2 ) $ defined by (\ref{eq:renormhatM2}) can be expanded as
\begin{eqnarray}
\nonumber
&&\hspace*{-0.2in}
 \widetilde{M}^\eps_2\big( \frac{z}{\eps} ,\bxi_1,\bxi_2,  \bzeta_1,\bzeta_2 \big)  =
K(z)^2
\phi^\eps_{\rho_0} ( \bxi_1 )
\phi^\eps_{\rho_0} ( \bxi_2 )
\phi^\eps_{r_0} ( \bzeta_1 )
\phi^\eps_{r_0} ( \bzeta_2 ) \\
\nonumber
&& 
\hspace*{-0.1in}
+
\frac{K(z)}{2}
\phi^\eps_{\rho_0} \big( \frac{\bxi_1-\bxi_2}{\sqrt{2}}\big)
\phi^\eps_{r_0} ( \bzeta_1 )
\phi^\eps_{r_0} ( \bzeta_2 )
A\big(z, \frac{\bxi_2+\bxi_1}{2} ,\frac{\bzeta_2 + \bzeta_1}{\eps},\Omega_2+\Omega_3\big) \\
\nonumber
&& \hspace*{-0.1in}
+
\frac{K(z)}{2}
\phi^\eps_{\rho_0} \big( \frac{\bxi_1+\bxi_2}{\sqrt{2}}\big)
\phi^\eps_{r_0} ( \bzeta_1 )
\phi^\eps_{r_0} ( \bzeta_2 )
A \big(z, \frac{\bxi_2-\bxi_1}{2} ,\frac{\bzeta_2- \bzeta_1}{\eps} ,\Omega_2-\Omega_3\big) \\
\nonumber
&& \hspace*{-0.1in}
+
\frac{K(z)}{2}
\phi^\eps_{R_0} \big( \frac{\bxi_1-\bzeta_2}{\sqrt{2}}\big)
\phi^\eps_{r_0} ( \bzeta_1 ) \phi^\eps_{\rho_0} ( \bxi_2 )
A\big(z, \frac{\bzeta_2+\bxi_1}{2} ,\frac{\bxi_2+ \bzeta_1}{\eps} ,\Omega_2+\Omega_1 \big) \\
\nonumber
&& \hspace*{-0.1in}
+\frac{K(z)}{2}
\phi^\eps_{R_0} \big( \frac{\bxi_1+\bzeta_2}{\sqrt{2}}\big)
\phi^\eps_{r_0} ( \bzeta_1 ) \phi^\eps_{\rho_0} ( \bxi_2 )
A\big(z, \frac{\bzeta_2-\bxi_1}{2} ,\frac{\bxi_2- \bzeta_1}{\eps} , \Omega_2-\Omega_1  \big) \\
\nonumber
&&\hspace*{-0.1in}
+ \frac{1}{4}
\phi^\eps_{r_0} ( \bzeta_1 )\phi^\eps_{r_0} (\bzeta_2) 
A \big( z, \frac{\bxi_2+\bxi_1}{2},   \frac{\bzeta_2+ \bzeta_1}{\eps} ,\Omega_2+\Omega_3 \big) \\
\nonumber
&&\hspace*{1.1in} \times
A \big( z, \frac{\bxi_2-\bxi_1}{2},   \frac{\bzeta_2- \bzeta_1}{\eps} ,\Omega_2-\Omega_3 \big) \\
\nonumber
&&\hspace*{-0.1in} 
+ \frac{1}{4} \phi^\eps_{r_0} ( \bzeta_1 ) \phi^\eps_{\rho_0} (\bxi_2)
A \big( z, \frac{\bzeta_2+\bxi_1}{2},  \frac{\bxi_2+ \bzeta_1}{\eps}   , \Omega_2+\Omega_1\big)\\
\nonumber
&&\hspace*{1.1in}
\times A \big( z, \frac{\bzeta_2-\bxi_1}{2},  \frac{\bxi_2- \bzeta_1}{\eps}  , \Omega_2-\Omega_1\big)
\\
&&\hspace*{-0.1in}
 + R^\eps_2  (z ,\bxi_1,\bxi_2 ,  \bzeta_1 ,\bzeta_2 )   ,
\label{eq:propsci11}
\end{eqnarray}
where 
\begin{equation}
\frac{1}{R_0^2} = \frac{1}{2} \Big( \frac{1}{r_0^2} + \frac{1}{\rho_0^2} \Big) ,
\end{equation}
the function $K$ is defined by (\ref{def:K}), 
the function $(z,\bxi)\mapsto A(z,\bxi,\bzeta,\Omega)$ is the solution of (\ref{def:A}),
and the function $R^\eps_2$ satisfies
\begin{equation}
\sup_{z \in [0,Z]} \| R^\eps_2 (z,\cdot,\cdot,\cdot,\cdot) \|_{L^1(\RR^2\times \RR^2\times \RR^2\times \RR^2)} 
\stackrel{\eps \to 0}{\longrightarrow}  0  ,
\end{equation}
for any $Z>0$.
\end{proposition}

This result is an extension of Proposition  1 in \cite{garniers4}
 in which the case $r_0=\rho_0$ and $\Omega=0$ is addressed (whose proof follows the same lines as the one
 of Proposition~\ref{prop:3}).
It shows that, if we deal with an integral of $\widetilde{M}^\eps_2$ against a bounded function,
then we can replace $\widetilde{M}^\eps_2$ by the right-hand side of (\ref{eq:propsci11}) without the $R^\eps_2$ term
up to a negligible error when $\eps$ is small.  Note also that the result shows that the fourth-order moment 
$M_2$ can be expressed in terms of second-order moment $A$ in (\ref{def:A})  and in terms of the source field,
which can be seen as a `quasi-Gaussian property' \cite{garniers4}.

\subsection{The Strongly Scattering Regime}
Our goal is to find an explicit expression of the function $A$ defined by (\ref{def:A}).
The equation (\ref{def:A}) for $A$ (in which $\bzeta$ and $\Omega$ are frozen parameters)
can be solved exactly when $\Omega=0$:
\begin{equation}
\label{def:A0}
A(z,\bxi,\bzeta,0) =  \frac{K(z)}{(2\pi)^2} 
 \int_{\RR^2}  \Big[  \exp \Big( \frac{\omega_0^2}{4 c_o^2} \int_0^z C\big( \bx + \frac{c_o \bzeta}{\omega_0} z' \big) dz' \Big) -1\Big]
   \exp \big( -i \bxi\cdot \bx  \big)
 d\bx  .   
\end{equation}

When $\Omega \neq 0$ it is possible to find an approximate expression for $A(z,\bxi,\bzeta,\Omega)$
in the strongly scattering regime as we show below.
The strongly scattering regime corresponds to  
\begin{equation}\label{def:deep}
\omega_0^2 C({\bf 0})L/c_o^2 \gg 1,
\end{equation}
which means that the propagation distance $L$ is larger than the scattering mean free path defined by
\begin{equation}
\label{def:lsca}
\ell_{\rm sca} = \frac{8c_o^2 }{\omega_0^2 C({\bf 0})}.
\end{equation}
Indeed, the scattering mean free path is the characteristic decay length of the mean Green's function,
as shown by the form of the mean Green's function obtained by It\^o's formula:
$$
\EE [ \hat{G}(\omega,L,\bx,\by)  ] = \hat{G}_0(\omega,L,\bx,\by)  \exp\Big( - \frac{\omega_0^2 C({\bf 0})L}{8 c_o^2} \Big),
$$
where $\hat{G}_0$ is the homogeneous Green's function:
$$
\hat{G}_0(\omega,L,\bx,\by)
= \frac{  \omega}{2 i \pi c_o L}   \exp \Big(  i \frac{\omega |\bx-\by|^2}{2c_o L} \Big)    ,
$$
see also the discussion in Appendix \ref{app:a}.   

We assume that the medium fluctuations are isotropic and smooth enough so that the coefficient
\begin{equation}
\label{def:D}
D =  \frac{1}{(2\pi)^2} \int_{\RR^2} \hat{C}(\bk)|\bk|^2 d\bk 
\end{equation}
is finite. The coefficient $D$ is homogeneous to the inverse of a length.
This length is the paraxial length, ie, the propagation distance beyond which the 
paraxial approximation is not valid anymore.
Indeed, in the strongly scattering regime $L \gg \ell_{\rm sca}$ (which is equivalent to $\omega_0^2 C({\bf 0})L/c_o^2 \gg 1$),
the second moment of the Green's function is 
(see Proposition 12.7 \cite{noisebook}):
\begin{eqnarray}
\nonumber
&& \EE \big[ \hat{G}(\omega,L,\bx,\by) 
\overline{\hat{G}(\omega,L,\bx',\by) } \big] 
=
\hat{G}_0(\omega,L,\bx,\by) 
\overline{\hat{G}_0(\omega,L,\bx',\by) } 
 \exp \Big( -  \frac{ |\bx-\bx'|^2}{X_c^2(L)} \Big)   , 
\end{eqnarray}
where $X_c(L)$ is the correlation length of the wave field
\begin{equation}
X_c(L) = \frac{\sqrt{3} c_o}{\sqrt{D} \omega_0 \sqrt{L}} .
\end{equation}
When the correlation length of the field becomes of the order of the wavelength, ie, when $\omega_0 X_c(L) /c_o \sim 1$,
then the paraxial approximation is not valid anymore. The paraxial distance $\ell_{\rm par}$ such that 
$\omega_0 X_c(\ell_{\rm par}) /c_o =1$ is 
\begin{equation}\label{def:lpar}
\ell_{\rm par}= \frac{3}{D} .
\end{equation}
Note that $\ell_{\rm par} \gg \ell_{\rm sca}$. The ratio $\ell_{\rm par} / \ell_{\rm sca}$ is of the order of the square of the ratio
of the correlation length of the medium over the wavelength.

%

\begin{proposition}
\label{prop:1b}
When $L \gg \ell_{\rm sca}$, the function $A$ solution of (\ref{def:A}) 
can be approximated by 
the solution of the parabolic partial differential equation
\begin{equation}
\label{eq:tildeA}
\partial_z  {A}_{\rm s} = \frac{i c_o \Omega}{\omega_0^2} |\bxi|^2  {A}_{\rm s}+\frac{\omega_0^2 D}{16c_o^2} 
\Big[  \Delta_\bxi  {A}_{\rm s}-\frac{z^2 c_o^2}{ \omega_0^2} |\bzeta|^2  {A}_{\rm s} -2i\frac{z c_o}{\omega_0} \bzeta\cdot 
\nabla_\bxi  {A}_{\rm s}\Big] ,
\end{equation}
starting from $ {A}_{\rm s}(z=0,\bxi,\bzeta,\Omega) = (2\pi)^4 \delta(\bxi)$.
\end{proposition}
The approximation holds in the sense that, for any continuous and bounded function $f$ and for any $Z>0$:
$$
\int_{\RR^2} {f(\bxi)} {A}(Z,\bxi,\bzeta,\Omega) d\bxi 
\stackrel{L \gg \ell_{\rm sca}}{\simeq} 
\int_{\RR^2}  {f(\bxi)}  {A}_{\rm s}(Z,\bxi,\bzeta,\Omega) d\bxi .
$$

\noindent
\begin{proof}
In the proof we assume  that the correlation function of the medium is of the form
$C^\delta(\bx) = \delta^{-2} C(\delta \bx)$ and we study the convergence as $\delta \to 0$ of the solution of 
(\ref{def:A}). Note that the corresponding coefficient $\ell_{\rm sca}^\delta$ defined by (\ref{def:lsca})
is proportional to $\delta^{2}$ while the corresponding coefficient $D^\delta$ defined by (\ref{def:D})
is independent of $\delta$ in this scaling regime.

In the case $\Omega = 0$ the result can be obtained from the explicit expression (\ref{def:A0}).
By taking the limit $\delta \to 0$ and using the expansion $C^\delta(\bx) = \delta^{-2} C({\bf 0}) - D |\bx|^2/4 +o(1)$, one gets the function
$$
{A}_{\rm s}(z,\bxi,\bzeta,0) = (2\pi)^2
 \int_{\RR^2}    \exp \Big(-  \frac{\omega_0^2 D}{16 c_o^2} \int_0^z \big| \bx + \frac{c_o \bzeta}{\omega_0} z' \big| dz'
  -i \bxi\cdot \bx  \Big)
 d\bx  ,
$$ 
that is the solution of (\ref{eq:tildeA}) 
in the special case $\Omega=0$.

In the general case $\Omega \neq 0$ we use a probabilistic representation and invoke a diffusion-approximation theorem. 
First, we introduce 
$$
\tilde{A}^\delta(z,\bxi,\bzeta,\Omega) = {A}^\delta(z,\bxi,\bzeta,\Omega)
\exp \Big( - \frac{i c_o z}{\omega_0} \bxi\cdot \bzeta\Big)  +(2\pi)^4 \exp \Big( -\frac{\omega_0^2 C^\delta({\bf 0}) z}{4 c_o^2} \Big) \delta(\bxi) .
$$
It is the solution of 
\begin{eqnarray}
\nonumber
\partial_z \tilde{A}^\delta &=& \Big( \frac{i c_o\Omega}{\omega_0^2} |\bxi|^2 
- \frac{i c_o}{\omega_0} \bzeta\cdot \bxi \Big)\tilde{A}^\delta +\frac{\omega_0^2}{4(2\pi)^2c_o^2}
\int_{\RR^2} \hat{C}^\delta(\bk) \big[ \tilde{A}^\delta(\bxi-\bk)  -\tilde{A}^\delta(\bxi)\big] d\bk ,
\end{eqnarray}
starting from $\tilde{A}^\delta(z=0,\bxi,\bzeta,\Omega)=(2\pi)^4 \delta(\bxi)$.
Second we define the operators
\begin{eqnarray}
{\cal L}^\delta f(\bxi) &=&\frac{\omega_0^2}{4(2\pi)^2c_o^2}
\int_{\RR^2} \hat{C}^\delta(\bk) \big[ f(\bxi- \bk)  -f(\bxi)\big] d\bk ,\\
{\cal L}  f(\bxi) &=&\frac{\omega_0^2 D}{16 c_o^2} \Delta_\bxi f(\bxi) .
\end{eqnarray}
Since $\hat{C}^\delta(\bk) =\delta^{-4} \hat{C}(\bk/\delta)$,
the operator ${\cal L}^\delta$ can be written as 
$$
{\cal L}^\delta f(\bxi) = \frac{\omega_0^2}{4(2\pi)^2c_o^2 \delta^2}
\int_{\RR^2} \hat{C} (\bk) \big[ f(\bxi- \delta \bk)  -f(\bxi)\big] d\bk ,
$$
and it is  the infinitesimal generator of the random process $\boldsymbol{\Xi}^\delta(z) = \delta \,\boldsymbol{\Xi}(z/\delta^2)$ defined as
$$
\boldsymbol{\Xi}(z)  =  \boldsymbol{\Xi}(0)+ 
\sum_{k=1}^{N_z} {\itbf K}_k  ,
$$
where $N_z$ is a homogeneous Poisson point process with intensity ${\omega_0^2 C ({\bf 0})}/({4 c_o^2})$
 and $({\itbf K}_k )_{k \geq 1}$ is a sequence of
independent and identically distributed $\RR^2$-valued random variables with the probability density function
$$
 p_{{\itbf K}}(\bk) = \frac{\hat{C} (\bk)}{(2\pi)^2 C ({\bf 0})} .
$$
These random variables have mean zero and finite variance $D/C({\bf 0})$.
The compound Poisson 
process $\boldsymbol{\Xi}$ has independent and stationary increments, with the distribution characterized by
the characteristic function
\begin{equation}
\EE \big[ \exp \big( i \bx \cdot (\boldsymbol{\Xi}(z'+z)-\boldsymbol{\Xi}(z'))\big) \big] = \exp \Big( \frac{\omega_0^2 z}{4 c_o^2} \big( C(\bx)-C({\bf 0}) \big) \Big).
\end{equation}
Let us denote
$$
V(\bxi) = \frac{c_o\Omega}{\omega_0^2} |\bxi|^2 
- \frac{c_o}{\omega_0} \bzeta\cdot \bxi .
$$
For any continuous and bounded function $f$ and $Z>0$, the solution of 
\begin{eqnarray}
\nonumber
\partial_z \tilde{u}^\delta &=& iV(\bxi) \tilde{u}^\delta  - {\cal L}^\delta \tilde{u}^\delta  ,
\end{eqnarray}
with the terminal condition $\tilde{u}^\delta(z=Z,\bxi)=f(\bxi)$, can be expressed by Feynman-Kac formula as
$$
\tilde{u}^\delta(z,\bxi) = \EE  \Big[ f \big(\boldsymbol{\Xi}^\delta(Z) \big)  \exp \Big(- i \int_z^Z V(\boldsymbol{\Xi}^\delta(z')) dz'  \Big) 
\Big| \boldsymbol{\Xi}^\delta(z)=\bxi\Big] .
$$
We can check that
$$
\partial_z \int_{\RR^2} \overline{\tilde{u}^\delta(z,\bxi) }\tilde{A}^\delta(z,\bxi) d\bxi =0,
$$
therefore
\begin{eqnarray*}
\int_{\RR^2} \overline{f(\bxi)} \tilde{A}^\delta(Z,\bxi) d\bxi &=&  \int_{\RR^2} \overline{\tilde{u}^\delta(Z,\bxi) }\tilde{A}^\delta(Z,\bxi) d\bxi =
 \int_{\RR^2} \overline{\tilde{u}^\delta(0,\bxi) }\tilde{A}^\delta(0,\bxi) d\bxi \\
 &=& (2\pi)^4 \tilde{u}^\delta(0,{\bf 0}) .
\end{eqnarray*}
By Donsker's invariance principle 
the random process $\boldsymbol{\Xi}^\delta$ weakly converges (as a cadlag process) to a Brownian motion ${\itbf W}$
with generator $ {\cal L}$. This shows that $\tilde{u}^\delta(0,{\bf 0})$ converges
to $\tilde{u}(0,{\bf 0})$, where $\tilde{u}(z,\bxi)$ is defined by
$$
\tilde{u}(z,\bxi) = \EE  \Big[ f \big( {\itbf W}(Z) \big)  \exp \Big(- i \int_z^Z V({\itbf W}(z')) dz'\Big) \Big| {\itbf W}(z)=\bxi \Big] ,
$$
which is solution of 
\begin{eqnarray}
\nonumber
\partial_z \tilde{u}  &=&  iV(\bxi) \tilde{u}  - {\cal L} \tilde{u}   ,
\end{eqnarray}
with the terminal condition $\tilde{u} (z=Z,\bxi)=f(\bxi)$.
If we denote by $\tilde{A}$ the solution of
\begin{eqnarray}
\nonumber
\partial_z \tilde{A}  &=& iV(\bxi) \tilde{A}  + {\cal L} \tilde{A}   ,
\end{eqnarray}
with the initial condition  $\tilde{A} (z=0,\bxi)=(2\pi)^4 \delta(\bxi)$,
then we find that 
$$
\int_{\RR^2} \overline{f(\bxi)} \tilde{A}(Z,\bxi) d\bxi =
  \int_{\RR^2} \overline{\tilde{u}(Z,\bxi) }\tilde{A}(Z,\bxi) d\bxi =
 \int_{\RR^2} \overline{\tilde{u}(0,\bxi) }\tilde{A}(0,\bxi) d\bxi = (2\pi)^4 \tilde{u}(0,{\bf 0}) .
$$
This establishes that, for any continuous and bounded  function $f$ and $Z>0$,
$$
\int_{\RR^2} \overline{f(\bxi)} \tilde{A}^\delta(Z,\bxi) d\bxi 
\stackrel{\delta \to 0}{\longrightarrow} 
\int_{\RR^2} \overline{f(\bxi)} \tilde{A}(Z,\bxi) d\bxi ,
$$
which proves that $\tilde{A}^\delta$ converges to $\tilde{A}$.
By considering ${A}_{\rm s}(z,\bxi) = \tilde{A}(z,\bxi) \exp \big( \frac{i c_o z}{\omega_0} \bxi\cdot \bzeta\big)$ we find that ${A}_{\rm s}$ satisfies
(\ref{eq:tildeA}),
and $A^\delta $ converges to $A_{\rm s}$,
which is the desired result.
\end{proof}

Let us consider the partial inverse Fourier transform
\begin{equation}
\hat{A}_{\rm s}(z,\bx,\bzeta,\Omega) =
 \frac{1}{(2\pi)^2} \int_{\RR^2} {A}_{\rm s}(z,\bxi,\bzeta,\Omega) \exp(i \bxi\cdot\bx)d\bxi.
\end{equation}

\begin{proposition}
\label{prop:4}
The partial inverse Fourier transform $\hat{A}_{\rm s}(z,\bx,\bzeta,\Omega) $ has the form
\begin{equation}
\label{def:hatA}
\hat{A}_{\rm s}(z,\bx,\bzeta,\Omega) = (2\pi)^2 \exp \big[ - a_\Omega(z) -b_\Omega(z) |\bx|^2 -c_\Omega(z) \bx\cdot\bzeta - d_\Omega(z) |\bzeta|^2 \big],
\end{equation}
where 
\begin{eqnarray}
\label{def:aOmega}
a_\Omega(z)&=& \Psi_a   \Big(   \sqrt{\frac{D \Omega}{4 c_o}} z \Big) ,\\
b_\Omega(z) &=& \frac{\omega_0^2 D z}{16 c_o^2} 
\Psi_b \Big(  \sqrt{\frac{D \Omega}{4 c_o}} z \Big)  , \\
c_\Omega(z) &=& \frac{\omega_0 Dz^2}{16c_o} \Psi_c \Big(  \sqrt{\frac{D \Omega}{4 c_o}} z \Big) ,\\
d_\Omega(z) &=& \frac{D z^3}{48} \Psi_d \Big(   \sqrt{\frac{D \Omega}{4 c_o}} z \Big) ,
\label{def:dOmega}
\end{eqnarray}
with the functions $\Psi_{a,b,c,d}$ defined by
\begin{eqnarray}
\Psi_a(s) &=& \ln\Big[\cosh\big( e^{- i \frac{\pi}{4}}s \big)\Big],\\
\Psi_b(s)&=&\frac{\tanh(e^{- i \frac{\pi}{4}}s)}{e^{- i \frac{\pi}{4}}s},\\
\Psi_c(s) &=& 2i \frac{e^{- i \frac{\pi}{4}}s\tanh(e^{- i \frac{\pi}{4}}s) -1 +\cosh^{-1}(e^{- i \frac{\pi}{4}}s)}{ s ^2} ,\\
\Psi_d(s) &=& 1 - \frac{3i}{s^3} \int_0^s \big( e^{- i \frac{\pi}{4}}s'\tanh(e^{- i \frac{\pi}{4}}s') -1 + ( \cosh(e^{- i \frac{\pi}{4}}s'))^{-1}\big)^2 ds' .
\end{eqnarray}
\end{proposition}

The real parts of the functions $\Psi_{a,b,c,d}$ are plotted in Figure \ref{fig:rePsi}. Note that they are positive valued.
By Propositions \ref{prop:1} and \ref{prop:3} this result gives a complete and explicit expression of the second-order and 
fourth-order moment in the strongly scattering regime $L \gg \ell_{\rm sca}$.

\begin{proof}
By Proposition \ref{prop:1b}, $\hat{A}_{\rm s}$ is solution of 
$$
\partial_z \hat{A}_{\rm s} = - \frac{i c_o \Omega}{\omega_0^2} \Delta_\bx \hat{A}_{\rm s} - \frac{\omega_0^2 D}{16c_o^2} 
\Big[ |\bx|^2  + \frac{z^2 c_o^2}{ \omega_0^2}  |\bzeta|^2  +2\frac{z c_o}{\omega_0} \bzeta\cdot \bx
 \Big] \hat{A}_{\rm s} ,
$$
starting from $\hat{A}_{\rm s}(z=0,\bx,\bzeta,\Omega) =(2\pi)^2 $.
The solution has the form (\ref{def:hatA})
where $(a_\Omega,b_\Omega,c_\Omega,d_\Omega)$ is the solution of the system of ordinary differential equations:
\begin{eqnarray}
\label{eq:odeaOmega1}
\frac{{\rm d}a_\Omega}{{\rm d}z} &=& -i \frac{4 c_o \Omega}{\omega_0^2} b_\Omega ,\\
\frac{{\rm d}b_\Omega}{{\rm d}z} &=& \frac{\omega_0^2 D}{16 c_o^2} +  i \frac{4 c_o \Omega}{\omega_0^2} b_\Omega^2 ,\\
\frac{{\rm d}c_\Omega}{{\rm d}z} &=&  \frac{\omega_0 D z}{8c_o} + i \frac{4 c_o \Omega}{\omega_0^2} b_\Omega c_\Omega ,\\
\frac{{\rm d}d_\Omega}{{\rm d}z} &=&\frac{D z^2}{16} + i \frac{c_o \Omega}{\omega_0^2} c_\Omega^2 ,
\label{eq:odeaOmega4}
\end{eqnarray}
starting from $(a_\Omega,b_\Omega,c_\Omega,d_\Omega)(z=0)=(0,0,0,0)$.
We have $(a_{-\Omega},b_{-\Omega},c_{-\Omega},d_{-\Omega})(z)=(\overline{a_\Omega},\overline{b_\Omega},\overline{c_\Omega},\overline{d_\Omega})(z)$ 
and by solving the system, we obtain the desired result.
\end{proof}

\begin{figure}
\begin{center}
\begin{tabular}{c}
\includegraphics[width=6.0cm]{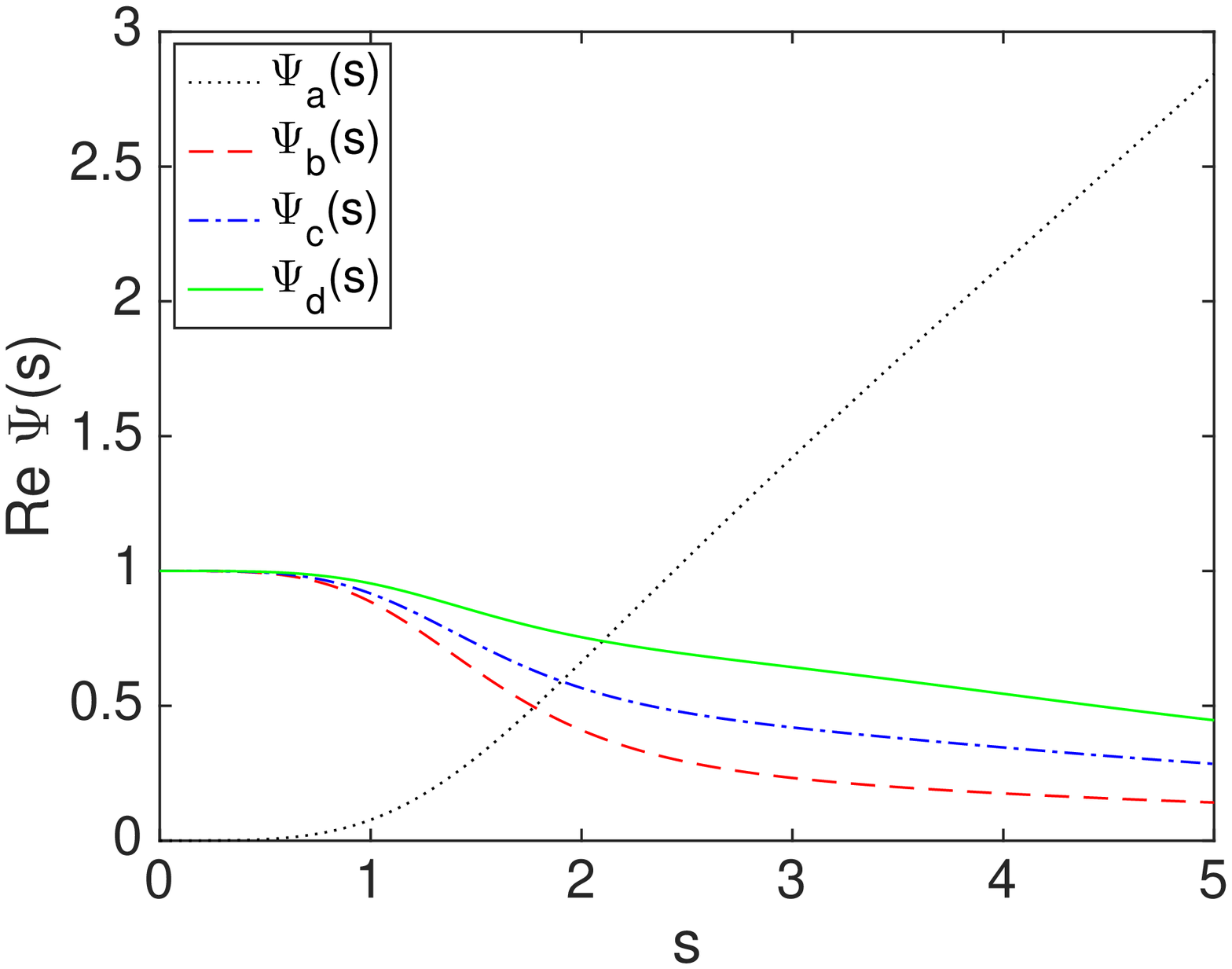}  
\includegraphics[width=6.0cm]{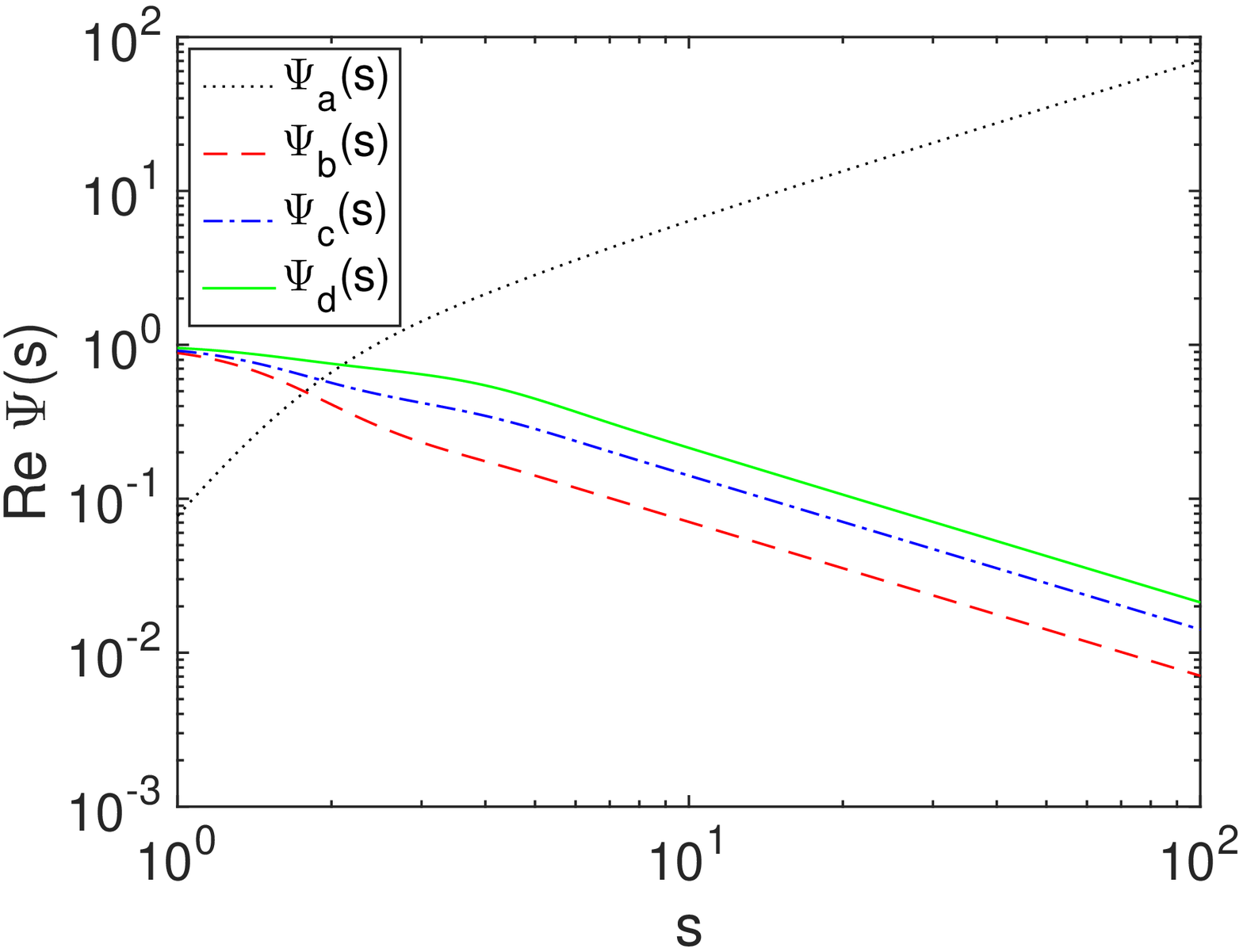}  
\end{tabular}
\end{center}
\caption{Real values of the functions $\Psi_{a,b,c,d}$ in linear scale (left) and log scale (right).}
\label{fig:rePsi}
\end{figure}

When  $D \Omega z^2 / c_o \ll 1$, we can use Taylor series expansions of the functions $\Psi_{a,b,c,d}$ 
to obtain
\begin{eqnarray*}
a_\Omega(z)&\simeq &  \hbox{} - i \frac{D \Omega z^2}{8c_o} + \frac{D^2 \Omega^2z^4}{192 c_o^2}
 +O\Big(\frac{D^3 \Omega^3z^6}{c_o^3}\Big) ,\\
b_\Omega(z) & \simeq &  \frac{\omega_0^2 D z}{16 c_o^2} 
\Big( 1 + i \frac{D \Omega z^2}{12 c_o} 
 - \frac{D^2 \Omega^2z^4}{120 c_o^2}   +O\Big(\frac{D^3 \Omega^3z^6}{c_o^3}\Big) \Big)  , \\
c_\Omega(z) &\simeq &  \frac{\omega_0 D z^2}{16 c_o} \Big( 1 +
i \frac{D \Omega z^2}{16 c_o}  -\frac{7 D^2 \Omega^2 z^4}{1152 c_o^2}+O\Big(\frac{D^3 \Omega^3z^6}{c_o^3}\Big) \Big)   , \\
d_\Omega(z) &\simeq & \frac{D z^3}{48} \Big( 1 + i \frac{3D \Omega z^2}{80 c_o} 
-\frac{3 D^2 \Omega^2 z^4}{896 c_o^2}+O\Big(\frac{D^3 \Omega^3z^6}{c_o^3}\Big) \Big) 
.
\end{eqnarray*}
The leading-order terms (with $\Omega=0$) 
are consistent with the limit of (\ref{def:A0}) in the strongly scattering regime $L \gg \ell_{\rm sca}$.

When $D \Omega z^2 / c_o \gg 1$, we can use asymptotic expressions for the functions 
$\Psi_{a,b,c,d}$ to obtain
\begin{eqnarray*}
a_\Omega(z)&\simeq & e^{ -i \frac{\pi}{4}} \sqrt{\frac{D \Omega}{4 c_o}} z - \ln 2 ,\\
b_\Omega(z) & \simeq & e^{i \frac{\pi}{4}} \sqrt{\frac{\omega_0^4 D}{64 c_o^3 \Omega}}, \\
c_\Omega(z) &\simeq &  e^{i \frac{\pi}{4}} \sqrt{\frac{\omega_0^2 D}{16 c_o \Omega}} z -\frac{i \omega_0}{2\Omega} , \\
d_\Omega(z) &\simeq & e^{i \frac{\pi}{4}} \sqrt{\frac{c_o D}{64 \Omega}} z^2 - \frac{i c_o z}{4 \Omega}
- e^{3 i \frac{\pi}{4}} \sqrt{\frac{ c_o^3}{4D \Omega^3}}  ,
\end{eqnarray*}
up to terms of relative order $\exp( - \sqrt{D \Omega  / (8c_o) } z)$. 
Note that, for the asymptotic expansion of $d_\Omega$, we  used the fact that
$\int_0^\infty 2 (s-1) [s(\tanh(s)-s+\cosh^{-1}(s)] + [s(\tanh(s)-s+\cosh^{-1}(s)]^2 ds=1$ in order to compute the $O(1)$-term.
Compared to the small (or vanishing) $\Omega$ case,
we can see that the growth rate in $z$ of the coefficients are very different.
This will have dramatic impact in the analysis of the refocused wave that we carry out in the next sections.

\subsection{The Scintillation Regime Revisited}
\label{sec:regime:2}%
In the scintillation regime (\ref{sca:sci}) addressed in the previous section, 
the TRM element size $\rho_0^\eps$ is assumed to be of order $\eps^{-1}$,
that is to say, larger than the correlation length of the medium.
We can also address the case where the TRM element size $\rho_0^\eps$ is of the same order as
the correlation length of the medium:
\begin{equation}
\label{sca:sci2}
C^\eps(\bx)= \eps C (\bx) , \quad \quad r_0^\eps = \frac{r_0}{\eps}, \quad \quad \rho_0^\eps = \rho_0 ,
\quad \quad L^\eps = \frac{L}{\eps} .
\end{equation}
The previous analysis can be revisited in the revised scintillation regime (\ref{sca:sci2}) and we get the following results.
\begin{proposition}
\label{prop:1:2}%
In the scintillation regime (\ref{sca:sci2}),
the function $\widetilde{M}^\eps_1(z/\eps,\bxi , \bzeta  ) $ defined by (\ref{eq:renormhatM1}) can be expanded as
\begin{eqnarray}
\nonumber
 \widetilde{M}^\eps_1\big( \frac{z}{\eps} ,\bxi , \bzeta \big)  &=&
\phi^\eps_{r_0/\sqrt{2} } ( \bzeta  ) 
A\big(z, \bxi  ,\frac{\bzeta}{\eps},\Omega\big) 
 + R^\eps_1 (z ,\bxi  ,  \bzeta )   ,
 \label{def:R1eps:2}
\end{eqnarray}
where the function $(z,\bxi)\mapsto A(z,\bxi,\bzeta,\Omega)$ is the solution of 
\begin{equation}
\partial_z A =  \frac{i c_o\Omega}{\omega_0^2} |\bxi|^2 A +\frac{\omega_0^2}{4(2\pi)^2c_o^2}
\int_{\RR^2} \hat{C}(\bk) \big[ A(\bxi-\bk) e^{\frac{i c_oz}{\omega_0} \bk \cdot \bzeta} -A(\bxi)\big] d\bk
\label{def:A:2}
\end{equation}
starting from $A(z=0,\bxi,\bzeta,\Omega)=(2\pi)^4 \phi^1_{\sqrt{2}\rho_0}(\bxi)$,
and the function $R^\eps_1 $ satisfies
\begin{equation}
\sup_{z \in [0,Z]} \| R^\eps_1 (z,\cdot,\cdot) \|_{L^1(\RR^2\times \RR^2 )} 
\stackrel{\eps \to 0}{\longrightarrow}  0  ,
\end{equation}
for any $Z>0$.
\end{proposition}

In particular, we have
\begin{equation}
A(z,\bxi,\bzeta,0) = (2\pi)^4 \int_{\RR^2} \exp\Big( - \frac{|\bx|^2}{4\rho_0^2} +\frac{\omega_0^2}{4c_o^2}
\int_0^z C\big(\bx+\frac{c_o \bzeta z'}{\omega_0}\big)-C({\bf 0}) dz' - i \bxi\cdot\bx\Big) d\bx.
\end{equation} 

\begin{proposition}
\label{prop:3:2}%
In the scintillation regime (\ref{sca:sci2}),
the function $\widetilde{M}^\eps_2 (z/\eps,\bxi_1,\bxi_2, \bzeta_1,\bzeta_2 ) $ defined by (\ref{eq:renormhatM2})  can be expanded as
\begin{align}
\nonumber
 \widetilde{M}^\eps_2\big( \frac{z}{\eps} ,\bxi_1,\bxi_2,  \bzeta_1,\bzeta_2 \big)  =&
\frac{1}{4}
\phi^\eps_{r_0} ( \bzeta_1 )\phi^\eps_{r_0} (\bzeta_2) 
A \big( z, \frac{\bxi_2+\bxi_1}{2},   \frac{\bzeta_2+ \bzeta_1}{\eps} ,\Omega_2+\Omega_3 \big) \\
\nonumber
& \hspace*{0.8in} \times
A \big( z, \frac{\bxi_2-\bxi_1}{2},   \frac{\bzeta_2- \bzeta_1}{\eps} ,\Omega_2-\Omega_3 \big) \\
& + R^\eps_2  (z ,\bxi_1,\bxi_2 ,  \bzeta_1 ,\bzeta_2 )   ,
\label{eq:propsci11:2}
\end{align}
where 
the function $(z,\bxi)\mapsto A(z,\bxi,\bzeta,\Omega)$ is the solution of (\ref{def:A:2}),
and the function $R^\eps_2$ satisfies
\begin{equation}
\sup_{z \in [0,Z]} \| R^\eps_2 (z,\cdot,\cdot,\cdot,\cdot) \|_{L^1(\RR^2\times \RR^2\times \RR^2\times \RR^2)} 
\stackrel{\eps \to 0}{\longrightarrow}  0  ,
\end{equation}
for any $Z>0$.
\end{proposition}

In the strongly scattering regime $L\gg \ell_{\rm sca}$, Proposition \ref{prop:1b} is still valid except that
the initial condition for $ {A}_{\rm s}$
 is $ {A}_{\rm s}(z=0,\bxi,\bzeta,\Omega) = (2\pi)^4 \phi^1_{\sqrt{2}\rho_0}(\bxi)$
instead of $ {A}_{\rm s}(z=0,\bxi,\bzeta,\Omega) = (2\pi)^4 \delta(\bxi)$.
As a result, the expression of $\hat{A}_{\rm s}$ given in Proposition \ref{prop:4} has to be updated.
The updated result is given in the following proposition.

\begin{proposition}
\label{prop:4:2}
The partial inverse Fourier transform $\hat{A}_{\rm s}(z,\bx,\bzeta,\Omega) $ has the form
(\ref{def:hatA}) where $(a_\Omega,b_\Omega,c_\Omega,d_\Omega)$ are given by 
(\ref{def:aOmega}-\ref{def:dOmega}) and the functions $\Psi_{a,b,c,d}$ are defined by
\begin{eqnarray}
\label{def:Psia:2}
\Psi_a(s) &=& \ln\Big[\cosh\big( e^{- i \frac{\pi}{4}}s \big)+T_0 \sinh\big( e^{- i \frac{\pi}{4}}s \big)\Big],\\
\Psi_b(s)&=&\frac{T_0+ \tanh(e^{- i \frac{\pi}{4}}s)}{e^{- i \frac{\pi}{4}}s[ 1 + T_0 \tanh(e^{- i \frac{\pi}{4}}s)]},\\
\Psi_c(s) &=& 2i \frac{e^{- i \frac{\pi}{4}}s
\frac{T_0+\tanh(e^{- i \frac{\pi}{4}}s)}{1+ T_0\tanh(e^{- i \frac{\pi}{4}}s)}-1 +[ \cosh(e^{- i \frac{\pi}{4}}s)+T_0 \sinh( e^{- i \frac{\pi}{4}}s )]^{-1}}{ s ^2} ,\\
\nonumber
\Psi_d(s) &=& 1 - \frac{3i}{s^3} \int_0^s \Big( e^{- i \frac{\pi}{4}}s'\frac{T_0+\tanh(e^{- i \frac{\pi}{4}}s')}{1+ T_0\tanh(e^{- i \frac{\pi}{4}}s')} -1\\
&& \hspace*{0.6in}  +[ \cosh(e^{- i \frac{\pi}{4}}s')+T_0 \sinh( e^{- i \frac{\pi}{4}}s' )]^{-1}\Big)^2 ds' ,
\label{def:Psid:2}
\end{eqnarray}
with
\begin{equation}
T_0 =  \frac{2 e^{- i \frac{\pi}{4}} }{\omega_0^2 \rho_0^2} \sqrt{\frac{c_o^3 \Omega}{D}}   .
\end{equation}
\end{proposition}
When $\rho_0 \to+\infty$, we have $T_0=0$ and we recover the result of Proposition \ref{prop:4}.

\begin{proof}
$\hat{A}_{\rm s}$ is given by (\ref{def:hatA})
and the functions $(a_\Omega,b_\Omega,c_\Omega,d_\Omega)$ satisfy the system of differential equations 
(\ref{eq:odeaOmega1}-\ref{eq:odeaOmega4}),
with the initial condition $b_\Omega(0)=1/(4\rho_0^2)$ instead of $b_\Omega(0)=0$.
By solving the differential equations we get the desired result.
\end{proof}

We remark that when $\Omega=0$,  we have
\begin{equation}
a_0(z) = 0  ,\quad
b_0(z) =  \frac{\omega_0^2 D z}{16 c_o^2} +\frac{1}{4\rho_0^2},
\quad
c_o(z) = \frac{\omega_0 D z^2}{16 c_o} ,
\quad
d_0(z)=  \frac{D z^3}{48} 
.
\end{equation}

\section{Time-Harmonic Wave Refocusing}
\label{sec:refocusth}
We address the situation described in Section \ref{sec:1th} in the scintillation regime  (\ref{sca:sci}).
We consider two nearby frequencies $\omega=\omega_0+\eps \Omega$ 
and $\tilde{\omega}=\omega_0-\eps \Omega$.
The goal is to determine the profile of the refocused wave and its signal-to-noise ratio.
We also want to determine for which frequency offset $\Omega$ time-reversal refocusing is still effective.

\subsection{The Mean Refocused Wave}
We first give the general expression of the mean refocused field in the scintillation regime.
\begin{proposition}
In the scintillation regime  (\ref{sca:sci}) the mean refocused field is
\begin{eqnarray*}
&& \EE\big[ \hat{u}_{\rm tr} \big(\frac{\by}{\eps}+\bx;\frac{\by}{\eps}\big)\big]
\stackrel{\eps\to 0}{\longrightarrow} \frac{K(L)}{(2\pi)^4} \int
\frac{r_0^2 }{4\pi}
\exp\Big(  -\frac{r_0^2 |\bzeta|^2}{4} 
+ i \by\cdot\bzeta \Big)d\bzeta \\
&&+ \frac{1}{(2\pi)^4} \iint_{\RR^2\times \RR^2}
\frac{r_0^2 }{4\pi}
\exp\Big( - \frac{r_0^2 |\bzeta|^2}{4} +i \bx\cdot\bxi + i \by \cdot \bzeta
 - i \frac{L c_o }{\omega_0} \bxi \cdot \bzeta  \Big)
A(L,\bxi,\bzeta,\Omega) 
 d\bxi d\bzeta. 
 \end{eqnarray*}
 \end{proposition}
 \begin{proof}
By using (\ref{def:mom1}) 
and by taking $C \to \eps C$, $r_0 \to r_0/\eps$, $\rho_0 \to \rho_0/\eps$, $\by \to \by/\eps$, $L\to L/\eps$,
the mean refocused wave is given by 
\begin{eqnarray*}
\EE\big[ \hat{u}_{\rm tr} \big(\frac{\by}{\eps}+\bx;\frac{\by}{\eps}\big)\big]
&=& {\cal M}_1 \big(\frac{\by}{\eps}+\bx,\frac{\by}{\eps}\big) \\
&=&
M_1^\eps \big(\frac{L}{\eps},\br=\frac{\by}{\eps}+\frac{\bx}{2},\bq=\bx\big) \\
&=&
\frac{1}{(2\pi)^4} \iint_{\RR^2\times \RR^2}
\widetilde{M}_1^\eps  \big(\frac{L}{\eps},\bxi,\bzeta \big)
 \exp \Big( i \bx\cdot\bxi + i (\frac{\by}{\eps}+\frac{\bx}{2})\cdot \bzeta
 - i \frac{L c_o }{\eps \omega_0} \bxi \cdot \bzeta \Big)
 d\bxi d\bzeta .
\end{eqnarray*}
In the limit $\eps \to 0$, we find from Proposition \ref{prop:1} the desired result.
 \end{proof}
 
In the weakly scattering regime $L \ll \ell_{\rm sca}$  (which is equivalent to $\omega_0^2 C({\bf 0})L/c_o^2 \ll 1$), 
we have $K(L)\simeq (2\pi)^4$ and $A \simeq 0$ so 
\begin{eqnarray*}
&& \EE\big[ \hat{u}_{\rm tr} \big(\frac{\by}{\eps}+\bx;\frac{\by}{\eps}\big)\big]
\stackrel{\eps\to 0}{\longrightarrow} \exp\Big( - \frac{|\by|^2}{r_0^2} \Big),
\end{eqnarray*}
which shows that there is no refocusing. This is because the TRM elements are too large and there is no multipathing effect due
to the random medium. 

In the strongly scattering regime $L \gg \ell_{\rm sca}$  (which is equivalent to $\omega_0^2 C({\bf 0})L/c_o^2 \gg 1$), 
we find by Proposition \ref{prop:4} that
\begin{eqnarray*}
\EE\big[ \hat{u}_{\rm tr} \big(\frac{\by}{\eps}+\bx ; \frac{\by}{\eps}\big)\big]
\stackrel{\eps\to 0}{\longrightarrow} 
\frac{e^{-a_\Omega(L)} r_0^2}{4 \pi} \int_{\RR^2} \exp\big( - e_\Omega(L) |\bzeta|^2 -f_\Omega(L) \bx\cdot \bzeta - b_\Omega(L) |\bx|^2 +i \bzeta \cdot \by \big)d\bzeta ,
\end{eqnarray*}
with
\begin{eqnarray}
\nonumber
e_\Omega(z)&=& \frac{r_0^2}{4} +d_\Omega(z)-\frac{c_o z}{\omega_0} c_\Omega(z) +\frac{c_o^2 z^2}{\omega_0^2} b_\Omega(z) \\
&=&  \frac{r_0^2}{4} + \frac{Dz^3}{48} \big( \Psi_d - 3\Psi_c+3\Psi_b\big) \Big(  \sqrt{\frac{D \Omega}{4 c_o}} z \Big)
\label{def:eOmega}
 ,\\
 \nonumber
f_\Omega(z) &=& c_\Omega(z) -\frac{2c_o z}{\omega_0} b_\Omega(z) \\
&=&  \frac{\omega_0 D z^2}{16 c_o}
\big( \Psi_c- 2 \Psi_b\big) \Big(  \sqrt{\frac{D \Omega}{4 c_o}} z \Big)
 .
 \label{def:fOmega}
\end{eqnarray}
This shows that the mean refocused wave has the form
of a Gaussian peak centered at the target location.
More exactly, 
if we consider the case when $\by={\bf 0}$, then we find that the mean refocused wave is
\begin{eqnarray}
\EE\big[ \hat{u}_{\rm tr} \big(\bx;{\bf 0}\big)\big]
=
\frac{e^{-a_\Omega(L)} r_0^2}{4 e_\Omega(L)} \exp\big( - g_\Omega(L) |\bx|^2 \big)  ,
\end{eqnarray}
with
\begin{eqnarray}
\nonumber
g_\Omega(z) &=& b_\Omega(z) - \frac{f_\Omega(z)^2}{4 e_\Omega(z)}\\
&=& \frac{\omega_0^2 D z}{16 c_o^2}
 \Big( \Psi_b - \frac{\frac{D z^3}{16} (\Psi_c-2\Psi_b)^2}{r_0^2 +\frac{Dz^3}{12} (\Psi_d-3\Psi_c+3\Psi_b)} \Big) \Big(  \sqrt{\frac{D \Omega}{4 c_o}} z \Big) .
\label{def:gOmega}
\end{eqnarray}

When $D \Omega L^2 / c_o \ll 1$, we have
\begin{eqnarray}
\EE\big[ \hat{u}_{\rm tr} \big(\bx;{\bf 0}\big)\big]
\simeq
\frac{1}{1+\frac{DL^3}{12 r_0^2}}  \exp
\Big( - \frac{\omega_0^2 DL}{16 c_o^2}  \frac{1+\frac{DL^3}{48 r_0^2}}{1+\frac{DL^3}{12 r_0^2}}
|\bx|^2 \Big) ,
\label{eq:expressmeanutr}
\end{eqnarray}
which is the expression of the mean refocused wave when $\Omega=0$ \cite{garniers5}, which does not depend on the 
array element size $\rho_0$ (which is too large to ensure refocusing), but strongly depends on the properties of the random medium  
(which is scattering enough to ensure the multipathing effect that gives rise to refocusing). 
 We observe a power-law decay of the mean peak amplitude as a function of the propagation distance.

When $D \Omega L^2 / c_o \gg 1$, we have
\begin{eqnarray}
\EE\big[ \hat{u}_{\rm tr} \big(\bx;{\bf 0}\big)\big]
\simeq
\frac{2 \exp\big( - e^{-i\pi/4} \sqrt{\frac{D \Omega}{4c_o}}L\big)}{1 + i \frac{Lc_o}{\Omega r_0^2}}
 \exp \Big( - e^{i\pi/4} \frac{\omega_0^2}{8 c_o^2} \sqrt{ \frac{Dc_o}{ \Omega}} |\bx|^2 \Big) .
 \label{eq:meanrefocustr2}
\end{eqnarray}
We observe an exponential decay of  the mean peak amplitude, of the form
$$
\exp\Big( -   \sqrt{\frac{D \Omega}{8c_o}}L\Big) ,
$$
while the radius of the mean peak becomes equal to
$$
\frac{c_o}{\omega_0} 2^{7/4} \sqrt[4]{ \frac{ \Omega}{Dc_o}} .
$$
These results (concerning the mean refocused wave) do not depend on the array size $r_0$ or array element size $\rho_0$.
They show that the amplitude of the mean refocused wave is noticeable provided $D \Omega L^2/(8c_o)<1$.
We will see in the next section that the signal-to-noise ratio indeed dramatically decays when this condition is not fulfilled.

\subsection{Signal-to-Noise Ratio Analysis}
We now give the general expression of the second-order moment of the refocused field in the scintillation regime.
\begin{proposition}
In the scintillation regime  (\ref{sca:sci}) the second-order moment of the refocused field is
\begin{eqnarray*}
&&\EE\big[ \big| \hat{u}_{\rm tr}\big(\frac{\by}{\eps}+\bx ; \frac{\by}{\eps}\big) \big|^2\big] 
\stackrel{\eps\to 0}{\longrightarrow} \frac{K(L)^2}{(2\pi)^8} \int_{\RR^2}
\frac{r_0^2}{2\pi}
\exp\Big(  -\frac{r_0^2 |\bzeta|^2}{2}+ 2 i \by\cdot\bzeta \Big) d\bzeta  \\
&&+ \frac{2K(L)}{(2\pi)^8}{\rm Re} \iint_{\RR^2\times \RR^2} \frac{r_0^2}{4\pi}
\exp\Big( -\frac{r_0^2 |\bzeta|^2}{4} + i \bx\cdot\bxi + i \by \cdot \bzeta -\frac{|\by|^2}{r_0^2}
\Big)\\&& \quad \times
 \exp\Big(  - i \frac{L c_o }{\omega_0} \bxi \cdot \bzeta\Big)
A(L,\bxi,\bzeta,\Omega) 
 d\bxi d\bzeta  \\
&&+ \frac{2K(L)}{(2\pi)^8}{\rm Re} \iint_{\RR^2\times \RR^2} \frac{\rho_0^2 r_0^2}{2\pi(\rho_0^2+r_0^2)}
\exp\Big( -\frac{r_0^2\rho_0^2 |\bzeta|^2}{2(r_0^2+\rho_0^2)} +   i \frac{2\rho_0^2}{r_0^2+\rho_0^2} \by \cdot \bzeta
-\frac{2 |\by|^2}{r_0^2+\rho_0^2} \Big)\\&& \quad \times
\exp\Big(  - i \frac{L c_o }{\omega_0} \bxi \cdot \bzeta\Big)A(L,\bxi,\bzeta,0) 
 d\bxi d\bzeta  \\
 &&+\Big| \frac{1}{(2\pi)^4} \iint_{\RR^2\times \RR^2} \frac{r_0^2}{4\pi} \exp\Big( -\frac{r_0^2 |\bzeta|^2}{4} + i \bx\cdot\bxi + i \by \cdot \bzeta  - i \frac{L c_o }{\omega_0} \bxi \cdot \bzeta\Big)
A(L,\bxi,\bzeta,\Omega) 
 d\bxi d\bzeta \Big|^2 \\
 &&+\frac{1}{(2\pi)^8}
\iint_{\RR^2\times \RR^2 \times \RR^2\times \RR^2} \frac{r_0^2 \rho_0^2}{(4\pi)^2} 
\exp\Big( -\frac{r_0^2+\rho_0^2}{8} (|\bzeta_a|^2+|\bzeta_b|^2) +\frac{r_0^2-\rho_0^2}{4} \bzeta_a\cdot \bzeta_b
\Big)\\
&&\quad \times
\exp\Big(  i \by \cdot (\bzeta_a-\bzeta_b) - i \frac{L c_o }{\omega_0}( \bxi_a \cdot \bzeta_a-\bxi_b\cdot \bzeta_b) \Big) 
A(L,\bxi_a,\bzeta_a,0) \overline{ A(L,\bxi_b,\bzeta_b,0)} d\bxi_a d\bxi_b d\bzeta_a d\bzeta_b
.
 \end{eqnarray*}
\end{proposition}
\begin{proof}
The second moment of the refocused wave is given by:
\begin{eqnarray*}
\EE\big[ \big| \hat{u}_{\rm tr}\big(\frac{\by}{\eps}+\bx ; \frac{\by}{\eps}\big) \big|^2\big] 
&=& {\cal M}_2 \big(\frac{\by}{\eps}+\bx,\frac{\by}{\eps}+\bx,\frac{\by}{\eps}\big) \\
&=&
M_2^\eps \big(\frac{L}{\eps},\br_1=2\frac{\by}{\eps}+\bx,\br_2={\bf 0}, \bq_1={\bf 0},\bq_2=\bx\big) \\
&=&
\frac{1}{(2\pi)^8} \iint
\widetilde{M}_2^\eps  \big(\frac{L}{\eps},\bxi_1,\bxi_2,\bzeta_1,\bzeta_2 \big)
 \exp \Big( i \bx\cdot\bxi_2 + i (2 \frac{\by}{\eps}+\bx)\cdot \bzeta_1\Big)\\
 &&\times \exp \Big(
 - i \frac{L c_o }{\eps \omega_0} (\bxi_1 \cdot \bzeta_1+\bxi_2\cdot \bzeta_2) \Big)
 d\bxi_1 d\bxi_2 d\bzeta_1 d\bzeta_2  ,
\end{eqnarray*}
with $\Omega_1=\Omega_2=0$ and $\Omega_3=\Omega$. 
In the limit $\eps \to 0$, we find from Proposition \ref{prop:3} the desired result.
\end{proof}

In the weakly scattering regime $L \ll \ell_{\rm sca}$, 
we find
\begin{eqnarray*}
{\rm Var} \big[ \hat{u}_{\rm tr} \big(\frac{\by}{\eps}+\bx ; \frac{\by}{\eps}\big) \big]
&=& \EE\big[ \big|\hat{u}_{\rm tr} \big(\frac{\by}{\eps}+\bx ; \frac{\by}{\eps}\big) \big|^2\big] 
-
\big| \EE\big[ \hat{u}_{\rm tr} \big(\frac{\by}{\eps}+\bx;\frac{\by}{\eps}\big) \big] \big|^2
\\
&\stackrel{\eps\to 0}{\longrightarrow} &
0,
\end{eqnarray*}
which follows since the scattering is negligible and the propagation approximately as in a homogeneous medium.  

In the strongly scattering regime $L \gg \ell_{\rm sca}$, 
we find by Proposition \ref{prop:4}  that
\begin{eqnarray*}
{\rm Var} \big[ \hat{u}_{\rm tr} \big(\frac{\by}{\eps}+\bx ; \frac{\by}{\eps}\big) \big]
&=& \EE\big[ \big|\hat{u}_{\rm tr} \big(\frac{\by}{\eps}+\bx ; \frac{\by}{\eps}\big) \big|^2\big] 
-
\big| \EE\big[ \hat{u}_{\rm tr} \big(\frac{\by}{\eps}+\bx;\frac{\by}{\eps}\big) \big] \big|^2
\\
&\stackrel{\eps\to 0}{\longrightarrow} &
\frac{r_0^2 \rho_0^2}{(4 \pi)^2} \iint_{\RR^2\times \RR^2}
\exp\Big( -\frac{r_0^2+\rho_0^2}{8} (|\bzeta_a|^2+|\bzeta_b|^2) +\frac{r_0^2-\rho_0^2}{4} \bzeta_a\cdot \bzeta_b\Big) \\
&&
\times
\exp\big( - d_0(L) (|\bzeta_a|^2+|\bzeta_b|^2) +i (\bzeta_a-\bzeta_b) \cdot \by \big)d\bzeta_a d\bzeta_b  ,
\end{eqnarray*}
with
\begin{eqnarray*}
d_0(z)&=& \frac{D z^3}{48}  .
\end{eqnarray*}
Note that the variance does not depend on the frequency offset $\Omega$
and we recover the result known in the case $\Omega=0$ \cite{garniers5},
while we have shown above that the amplitude of the main refocused wave decays as $|\Omega|$ increases.
Therefore the signal-to-noise ratio will increase as $|\Omega|$ increases, as we explain below.

If we consider the case when $\by={\bf 0}$, then we find that the variance of the refocused wave has the form
\begin{eqnarray}
{\rm Var} \big[ \hat{u}_{\rm tr} \big(\bx;{\bf 0}\big) \big]
=\frac{1}{\big(1+ \frac{D L^3}{12 r_0^2}\big)\big(1+\frac{D L^3}{12 \rho_0^2}\big)}.
\end{eqnarray}
The signal-to-noise ratio defined by
\begin{equation}
{\rm SNR} = \frac{
\big| \EE\big[ \hat{u}_{\rm tr} \big({\bf 0} ; {\bf 0}\big) \big] \big|^2
}
{
{\rm Var} \big[ \hat{u}_{\rm tr} \big(\bx; {\bf 0}\big)\big]
}
\end{equation}
is given by
\begin{equation}
{\rm SNR} = 
\frac{e^{-2 {\rm Re}[a_\Omega(L)]} r_0^4}{16 |e_\Omega(L)|^2}   \big(1+ \frac{DL^3}{12r_0^2}\big)
\big(1+\frac{DL^3}{12\rho_0^2}\big) .
\end{equation}
When $D \Omega L^2 / c_o \ll 1$, we have
$$
{\rm SNR} \simeq
\frac{1+\frac{DL^3}{12\rho_0^2}}{1+ \frac{DL^3}{12r_0^2}}  .
$$
This result has already been obtained (when $\Omega=0$) in \cite{garniers5}.
When $DL^3 \gg r_0^2,\rho_0^2$, we find that the SNR varies as $r_0^2/\rho_0^2$, that is to say,
as the number of elements of the TRM.

When $D \Omega L^2 / c_o \gg 1$, we have
\begin{eqnarray*}
{\rm SNR}
\simeq
\frac{2 \exp\big( -  \sqrt{\frac{D \Omega}{2c_o}}L\big)}{1 + \frac{L^2c_o^2}{\Omega^2 r_0^4}}
\big(1+ \frac{D L^3}{12 r_0^2}\big)\big(1+\frac{D L^3}{12 \rho_0^2}\big) ,
\end{eqnarray*}
which is dominated by the exponentially decaying term.

{\bf Conclusion.}
To summarize, refocusing can be achieved provided $D \Omega L^2/(8c_o)<1$,
which is a condition that depends only on the frequency offset $\Omega$, 
the coefficient $D$ or paraxial distance $\ell_{\rm par}=3/D$, and the propagation distance $L$.

\subsection{The Scintillation Regime Revisited}
In the scintillation regime (\ref{sca:sci2}),
where $\rho_0$ is of the same order as the correlation length of the random medium,
we find from Proposition \ref{prop:1:2} that the mean refocused wave is
\begin{eqnarray*}
 \EE\big[ \hat{u}_{\rm tr} \big(\frac{\by}{\eps}+\bx;\frac{\by}{\eps}\big)\big]
\stackrel{\eps\to 0}{\longrightarrow}  \frac{1}{(2\pi)^4} \iint_{\RR^2\times \RR^2} 
\frac{r_0^2 }{4\pi}
\exp\Big( - \frac{r_0^2 |\bzeta|^2}{4} +i \bx\cdot\bxi + i \by \cdot \bzeta
 - i \frac{L c_o }{\omega_0} \bxi \cdot \bzeta  \Big) \\
 \times
A(L,\bxi,\bzeta,\Omega) 
 d\bxi d\bzeta. 
 \end{eqnarray*}
 where $A$ is given by (\ref{def:A:2}).
 
In the weakly scattering regime $L \ll \ell_{\rm sca}$  (which is equivalent to $\omega_0^2 C({\bf 0})L/c_o^2 \ll 1$), 
we have $A(L,\bxi,\bzeta,\Omega) =(2\pi)^4 \phi^1_{\sqrt{2}\rho_0}(\bxi) \exp( i c_o \Omega |\bxi|^2 L / \omega_0^2)$ and therefore
\begin{eqnarray*}
 \EE\big[ \hat{u}_{\rm tr} \big(\frac{\by}{\eps}+\bx;\frac{\by}{\eps}\big)\big]
\stackrel{\eps\to 0}{\longrightarrow} 
\iint_{\RR^2\times \RR^2}
\frac{\rho_0^2 r_0^2 }{4\pi^2}
\exp\Big( - \frac{r_0^2 |\bzeta|^2}{4} - \rho_0^2 |\bxi|^2 +i \bx\cdot\bxi + i \by \cdot \bzeta \\
+i\frac{c_o \Omega L}{\omega_0^2} |\bxi|^2 - i \frac{L c_o }{\omega_0} \bxi \cdot \bzeta  \Big) 
 d\bxi d\bzeta.
 \end{eqnarray*}
If $\by={\bf 0}$, then we get
 \begin{eqnarray*}
 \EE\big[ \hat{u}_{\rm tr} \big( \bx;{\bf 0} \big)\big]
\stackrel{\eps\to 0}{\longrightarrow} 
\frac{1}{1+\frac{c_o^2 L^2}{\omega_0^2 \rho_0^2 r_0^2} - i \frac{c_o \Omega L}{\omega_0^2 \rho_0^2}}
\exp\Big( - \frac{|\bx|^2}{4\rho_0^2 \big( 1+\frac{c_o^2 L^2}{\omega_0^2 \rho_0^2 r_0^2} - i \frac{c_o \Omega L}{\omega_0^2 \rho_0^2}\big) }\Big),
\end{eqnarray*}
which shows that we can get refocusing because the TRM element size is small enough. 
The frequency shift $|\Omega|$ should be smaller than $\omega_0^2 \rho_0^2 / (c_o L) $ 
so that the quality of the refocusing is not affected.

In the strongly scattering regime $L \gg \ell_{\rm sca}$  (which is equivalent to $\omega_0^2 C({\bf 0})L/c_o^2 \gg 1$), 
we find by Proposition \ref{prop:4:2} that
\begin{eqnarray*}
\EE\big[ \hat{u}_{\rm tr} \big(\frac{\by}{\eps}+\bx ; \frac{\by}{\eps}\big)\big]
=
\frac{e^{-a_\Omega(L)} r_0^2}{4 \pi} \int \exp\big( - e_\Omega(L) |\bzeta|^2 -f_\Omega(L) \bx\cdot \bzeta - b_\Omega(L) |\bx|^2 +i \bzeta \cdot \by \big)d\bzeta ,
\end{eqnarray*}
with $a_\Omega$ defined by (\ref{def:aOmega}), $(e_\Omega,f_\Omega)$ defined by (\ref{def:eOmega}-\ref{def:fOmega}),
and $(\Psi_a,\Psi_b,\Psi_c,\Psi_d)$ defined by (\ref{def:Psia:2}-\ref{def:Psid:2}).
More exactly, 
if we consider the case when $\by={\bf 0}$, then we find that the mean refocused wave is
\begin{eqnarray}
\EE\big[ \hat{u}_{\rm tr} \big(\bx;{\bf 0}\big)\big]
=
\frac{e^{-a_\Omega(L)} r_0^2}{4 e_\Omega(L)} \exp\big( - g_\Omega(L) |\bx|^2 \big)  ,
\end{eqnarray}
with $g_\Omega$ defined by (\ref{def:gOmega}).

When $D \Omega L^2 / c_o \ll 1$, we have
\begin{equation}
\EE\big[ \hat{u}_{\rm tr} \big(\bx;{\bf 0}\big)\big]
\simeq
\frac{1}{1+\frac{DL^3}{12 r_0^2} + \frac{c_o^2 L^2}{\omega_0^2 \rho_0^2 r_0^2}}  \exp
\Big( -   \frac{\frac{\omega_0^2 DL}{16 c_o^2}\big(1+\frac{DL^3}{48 r_0^2}\big)
+ \frac{1}{4\rho_0^2} \big( 1+\frac{DL^3}{12 r_0^2}\big)}{1+\frac{DL^3}{12 r_0^2}+ \frac{c_o^2 L^2}{\omega_0^2 \rho_0^2 r_0^2}}
|\bx|^2 \Big) .
\label{eq:meanrefocussf:2}
\end{equation}
We can identify the radius $R$ of the mean refocused wave:
\begin{equation}
R^2 =  \frac{1+\frac{DL^3}{12 r_0^2}+ \frac{c_o^2 L^2}{\omega_0^2 \rho_0^2 r_0^2}}{\frac{\omega_0^2 DL}{8 c_o^2}\big(1+\frac{DL^3}{48 r_0^2}\big)
+ \frac{1}{2\rho_0^2} \big( 1+\frac{DL^3}{12 r_0^2}\big)} .
\label{eq:radiussf:2}
\end{equation}
The radius of the mean refocused wave is smaller when $\rho_0$ is smaller and when the random medium is more scattering (i.e., $D$ is larger).
When $\rho_0$ becomes large, we recover the expression (\ref{eq:expressmeanutr}).

When $D \Omega L^2 / c_o \gg 1$, we get the result (\ref{eq:meanrefocustr2}) and we observe again an exponential
decay of the mean peak amplitude. 

Finally, we find from Proposition \ref{prop:3:2} that 
\begin{equation}
\lim_{\eps \to 0}
\EE\big[ \big| \hat{u}_{\rm tr}\big(\frac{\by}{\eps}+\bx ; \frac{\by}{\eps}\big) \big|^2\big] 
=
\lim_{\eps \to 0}
\big| \EE\big[ \hat{u}_{\rm tr}\big(\frac{\by}{\eps}+\bx ; \frac{\by}{\eps}\big) \big] \big|^2.
\end{equation}
The refocused wave is statistically stable in this regime, because there are many elements (of the order of $\eps^{-2}$) in the TRM.

\section{Time Reversal Stability}
\label{sec:refocusbb}
We address the situation described in Section \ref{sec:tr} in the scintillation regime  (\ref{sca:sci}).
We consider a pulse whose bandwidth is small, of order $\eps$:
$$
\hat{f}^\eps(\omega) = \frac{\sqrt{2\pi}}{\eps B } \exp \Big(- \frac{(\omega-\omega_0)^2}{2 \eps^2 B^2} \Big)  .
$$
The goal is to determine the profile of the refocused wave and its signal-to-noise ratio.
In particular we want to determine for which bandwidth $B$ time-reversal refocusing is statistically stable.

\subsection{The Mean Refocused Wave}
In the limit $\eps \to 0$, we find from (\ref{eq:base:t}) and Proposition \ref{prop:1} that
the mean refocused wave is given by:
\begin{eqnarray*}
&& \EE\big[ {u}_{\rm tr} \big( \frac{t}{\eps} ,\frac{\by}{\eps}+\bx ; \frac{\by}{\eps}\big)\big]\\
&&
\stackrel{\eps\to0}{\longrightarrow} \frac{\exp(-B^2 t^2/2)   e^{-i\omega_0 t/\eps} }{(2\pi)^4} \Big\{
K(L)  \exp\Big( - \frac{|\by|^2}{r_0^2}\Big)  
  \\ 
&&+  \iint_{\RR^2\times \RR^2} 
\frac{r_0^2 }{4\pi}
\exp\Big( - \frac{r_0^2 |\bzeta|^2}{4} +i \bx\cdot\bxi + i \by \cdot \bzeta
 - i \frac{L c_o }{\omega_0} \bxi \cdot \bzeta  \Big)
A(L,\bxi,\bzeta,0) 
 d\bxi d\bzeta \Big\} . 
 \end{eqnarray*}
In the weakly scattering regime $L \ll \ell_{\rm sca}$, we find
\begin{eqnarray*}
\EE\big[ {u}_{\rm tr} \big( \frac{t}{\eps} ,\frac{\by}{\eps}+\bx ; \frac{\by}{\eps}\big)\big]
=     \exp(-B^2 t^2/2)  e^{-i\omega_0 t/\eps} \exp\Big( - \frac{|\by|^2}{r_0^2}\Big) ,
\end{eqnarray*}
which shows that there is not refocusing.

In the strongly scattering regime $L \gg \ell_{\rm sca}$, 
we find by Proposition \ref{prop:4}  that
\begin{eqnarray*}
 \EE\big[ {u}_{\rm tr} \big( \frac{t}{\eps} ,\frac{\by}{\eps}+\bx ; \frac{\by}{\eps}\big)\big]
&=&
\frac{r_0^2 \exp(-B^2 t^2/2)  e^{-i\omega_0 t/\eps} }{4 \pi} \\
&& \times \int_{\RR^2} \exp\big( - e_0(L) |\bzeta|^2 -f_0(L) \bx\cdot \bzeta - b_0(L) |\bx|^2 +i \bzeta \cdot \by \big)d\bzeta ,
\end{eqnarray*}
with
\begin{eqnarray*}
b_0(z) = \frac{\omega_0^2 D z}{16 c_o^2} ,\quad \quad 
e_0(z)  =\frac{r_0^2}{4}  + \frac{Dz^3}{48}  ,\quad \quad 
f_0(z) = -  \frac{\omega_0 D z^2}{16 c_o} .
\end{eqnarray*}
In particular, if $\by={\bf 0}$, we find
\begin{eqnarray}
 \EE\big[ {u}_{\rm tr} \big( \frac{t}{\eps} , \bx ; {\bf 0}\big)\big]
 = \frac{  e^{-i\omega_0 t/\eps} }{1+\frac{DL^3}{12 r_0^2}} \exp \Big( - \frac{B^2 t^2}{2} - \frac{\omega_0^2  DL}{16 c_o^2}  \frac{1+\frac{DL^3}{48 r_0^2}}{1+\frac{DL^3}{12 r_0^2}}
|\bx|^2 \Big) ,
 \end{eqnarray}
 which shows that there is refocusing, with a focal spot radius that is all the smaller as the medium is more scattering.

\subsection{Signal-to-Noise Ratio Analysis}
Let us consider the second moment of the refocused wave. 
In the limit $\eps \to 0$, we find from Proposition \ref{prop:3} that
\begin{eqnarray*}
&&\EE\big[ \big|{u}_{\rm tr}\big(\frac{t}{\eps}, \frac{\by}{\eps}+\bx ; \frac{\by}{\eps}\big) \big|^2\big] \\
&&
\stackrel{\eps\to 0}{\longrightarrow}   \frac{K(L)^2 \exp(-B^2 t^2)}{(2\pi)^8} \int_{\RR^2}
\frac{r_0^2  }{2\pi}
\exp\Big(   -\frac{r_0^2 |\bzeta|^2}{2}+ 2 i \by\cdot\bzeta \Big)  d\bzeta   \\
&&+ \frac{2K(L) \exp(-B^2 t^2)}{(2\pi)^8}{\rm Re} \int_\RR \iint_{\RR^2\times \RR^2} \frac{r_0^2}{4\pi}
\exp\Big( -\frac{r_0^2 |\bzeta|^2}{4} + i \bx\cdot\bxi + i \by \cdot \bzeta   -\frac{|\by|^2}{r_0^2}
\Big)\\&& \quad \times
 \exp\Big(  - i \frac{L c_o }{\omega_0} \bxi \cdot \bzeta\Big)
A(L,\bxi,\bzeta,0) 
 d\bxi d\bzeta d\Omega  \\
&&+ \frac{2K(L)}{(2\pi)^8 \sqrt{\pi} B} {\rm Re}  \int_\RR \iint_{\RR^2\times \RR^2} \frac{\rho_0^2 r_0^2}{2\pi(\rho_0^2+r_0^2)}
\exp\Big( -\frac{r_0^2\rho_0^2 |\bzeta|^2}{2(r_0^2+\rho_0^2)} +   i \frac{2\rho_0^2}{r_0^2+\rho_0^2} \by \cdot \bzeta
-\frac{2 |\by|^2}{r_0^2+\rho_0^2} \Big)\\&& \quad \times
\exp\Big(  - i \frac{L c_o }{\omega_0} \bxi \cdot \bzeta -2i \Omega t -\frac{\Omega^2}{B^2} \Big) A(L,\bxi,\bzeta,\Omega) 
 d\bxi d\bzeta  d\Omega \\
 &&+\Big| \frac{\exp(-B^2 t^2/2)}{(2\pi)^4} \iint_{\RR^2\times \RR^2} \frac{r_0^2}{4\pi} \exp\Big( -\frac{r_0^2 |\bzeta|^2}{4} + i \bx\cdot\bxi + i \by \cdot \bzeta  - i \frac{L c_o }{\omega_0} \bxi \cdot \bzeta\Big)
A(L,\bxi,\bzeta,0) 
 d\bxi d\bzeta \Big|^2 \\
 &&+\frac{1}{(2\pi)^8 \sqrt{\pi} B } \int_\RR
\iint_{\RR^2\times \RR^2 \times \RR^2 \times \RR^2} \frac{r_0^2 \rho_0^2}{(4\pi)^2} 
\exp\Big( -\frac{r_0^2+\rho_0^2}{8} (|\bzeta_a|^2+|\bzeta_b|^2) +\frac{r_0^2-\rho_0^2}{4} \bzeta_a\cdot \bzeta_b
\Big)\\
&&\quad \times
\exp\Big(  i \by \cdot (\bzeta_a-\bzeta_b) - i \frac{L c_o }{\omega_0}( \bxi_a \cdot \bzeta_a-\bxi_b\cdot \bzeta_b) \Big) 
A(L,\bxi_a,\bzeta_a,\Omega) \overline{ A(L,\bxi_b,\bzeta_b,\Omega)} 
\\
&& \quad \times \exp\Big(-2 i \Omega t - \frac{\Omega^2}{B^2}\Big) d\bxi_a d\bxi_b d\bzeta_a d\bzeta_b d\Omega
.
 \end{eqnarray*}
In the weakly scattering regime $L \ll \ell_{\rm sca}$,  we get 
\begin{eqnarray*}
&&{\rm Var} \big[{u}_{\rm tr} \big( \frac{t}{\eps} , \frac{\by}{\eps}+\bx ; \frac{\by}{\eps}\big) \big]
= \EE\big[ \big| {u}_{\rm tr} \big(\frac{t}{\eps}, \frac{\by}{\eps}+\bx ; \frac{\by}{\eps}\big) \big|^2\big] 
-
\big| \EE\big[ {u}_{\rm tr} \big(\frac{t}{\eps},\frac{\by}{\eps}+\bx;\frac{\by}{\eps}\big) \big] \big|^2
\\
&&\stackrel{\eps\to 0}{\longrightarrow} 
0.
\end{eqnarray*}
In the strongly scattering regime $L \gg \ell_{\rm sca}$, 
we find by Proposition \ref{prop:4} that
\begin{eqnarray*}
&&{\rm Var} \big[{u}_{\rm tr} \big( \frac{t}{\eps} , \frac{\by}{\eps}+\bx ; \frac{\by}{\eps}\big) \big]
= \EE\big[ \big| {u}_{\rm tr} \big(\frac{t}{\eps}, \frac{\by}{\eps}+\bx ; \frac{\by}{\eps}\big) \big|^2\big] 
-
\big| \EE\big[ {u}_{\rm tr} \big(\frac{t}{\eps},\frac{\by}{\eps}+\bx;\frac{\by}{\eps}\big) \big] \big|^2
\\
&&\stackrel{\eps\to 0}{\longrightarrow} 
\frac{r_0^2 \rho_0^2}{(4 \pi)^2 B \sqrt{\pi}} \int_\RR \iint_{\RR^2\times \RR^2}\exp\Big( -\frac{r_0^2+\rho_0^2}{8} (|\bzeta_a|^2+|\bzeta_b|^2) +\frac{r_0^2-\rho_0^2}{4} \bzeta_a\cdot \bzeta_b\Big) \\
&&\quad 
\times
\exp \big(- a_\Omega(L)-  h_\Omega(L) |\bzeta_a|^2 
- \overline{a_\Omega}(L) - \overline{h_\Omega}(L) |\bzeta_b|^2  \big)\\
&& \quad \times \exp \Big(i (\bzeta_a-\bzeta_b) \cdot \by -2 i \Omega t -\frac{\Omega^2}{B^2} \Big) d\bzeta_a d\bzeta_b d\Omega ,
\end{eqnarray*}
with
\begin{eqnarray}
\nonumber
h_\Omega(z) &=&
\frac{c_o^2 z^2}{\omega_0^2} b_\Omega(z) -\frac{c_o z}{\omega_0} c_\Omega(z) +d_\Omega(z) \\
&=&  \frac{Dz^3}{48} \big( \Psi_d - 3\Psi_c+3\Psi_b\big) \Big( \sqrt{\frac{D \Omega}{4 c_o}} z \Big) .
\end{eqnarray}
For $\by ={\bf 0}$,  this gives
 \begin{eqnarray*}
{\rm Var} \big[{u}_{\rm tr} \big( \frac{t}{\eps} , \bx ;{\bf 0} \big) \big]
&=&
\frac{1}{ \sqrt{\pi}} \int_\RR \hat{a}_s  \exp \big( -s^2 -2i Bt s\big)
\Big(1+\frac{DL^3 \hat{h}_s}{12\rho_0^2}\Big)^{-1}
\Big(1+\frac{DL^3\hat{h}_s}{12r_0^2}\Big)^{-1} d s ,
\end{eqnarray*}
 with $\hat{a}_s = \exp(-2 {\rm Re}(a_{Bs}(L)) )$ and $ \frac{DL^3}{12} \hat{h}_s = 4{\rm Re}( h_{Bs}(L)) $.
More explicitely,
\begin{eqnarray*}
\hat{a}_s &=& \hat{\cal A}  \big( \frac{D B L^2|s|}{4 c_o}\big),\quad\quad \hat{\cal A}(s)=     \frac{2}{\big(\cos +\cosh\big)\big( \sqrt{2s} \big)} ,\\
\hat{h}_s &=&\hat{\cal H}  \big( \frac{D B L^2|s|}{4 c_o}\big),\quad\quad \hat{\cal H}(s) = {\rm Re}\Big\{
 \big(\Psi_d -3\Psi_c+3\Psi_b\big)  \big(  \sqrt{s} \big)\Big\}.
\end{eqnarray*}
The SNR defined by
$$
{\rm SNR} = \frac{\big| \EE [ {u}_{\rm tr} \big( 0 , {\bf 0} ;{\bf 0}\big) ] \big|^2}{{\rm Var} \big[{u}_{\rm tr} \big( 0 , \bx ;{\bf 0}\big) \big]}
$$
is therefore
\begin{equation}
\label{snr:tr1}
{\rm SNR}^{-1} = \frac{2}{ \sqrt{\pi}} \int_0^\infty \hat{\cal A}  \big( \frac{D B L^2s}{4 c_o}\big) 
\frac{\big( 1+\frac{DL^3}{12 r_0^2}\big)^2  \exp \big( -s^2 \big)}{\big(1+\frac{DL^3}{12 \rho_0^2}\hat{\cal H}  \big( \frac{D B L^2s}{4 c_o}\big) \big)
\big(1+\frac{DL^3}{12 r_0^2}\hat{\cal H}  \big( \frac{D B L^2s}{4 c_o}\big) \big) } d s .
\end{equation}
We can observe that there is a complicated interplay between spatial and frequency effects,
that depends on three dimensionless parameters: $\frac{DL^3}{12 r_0^2}$,
$\frac{DL^3}{12 \rho_0^2}$, and $\frac{D B L^2}{4 c_o}$.
We plot in Figure \ref{fig:snr} the SNR for different values of these three parameters,
where we can see that the SNR increases with these three parameters,
and analyze below its asymptotic behavior.

\begin{figure}
\begin{center}
\begin{tabular}{cc}
\includegraphics[width=6.0cm]{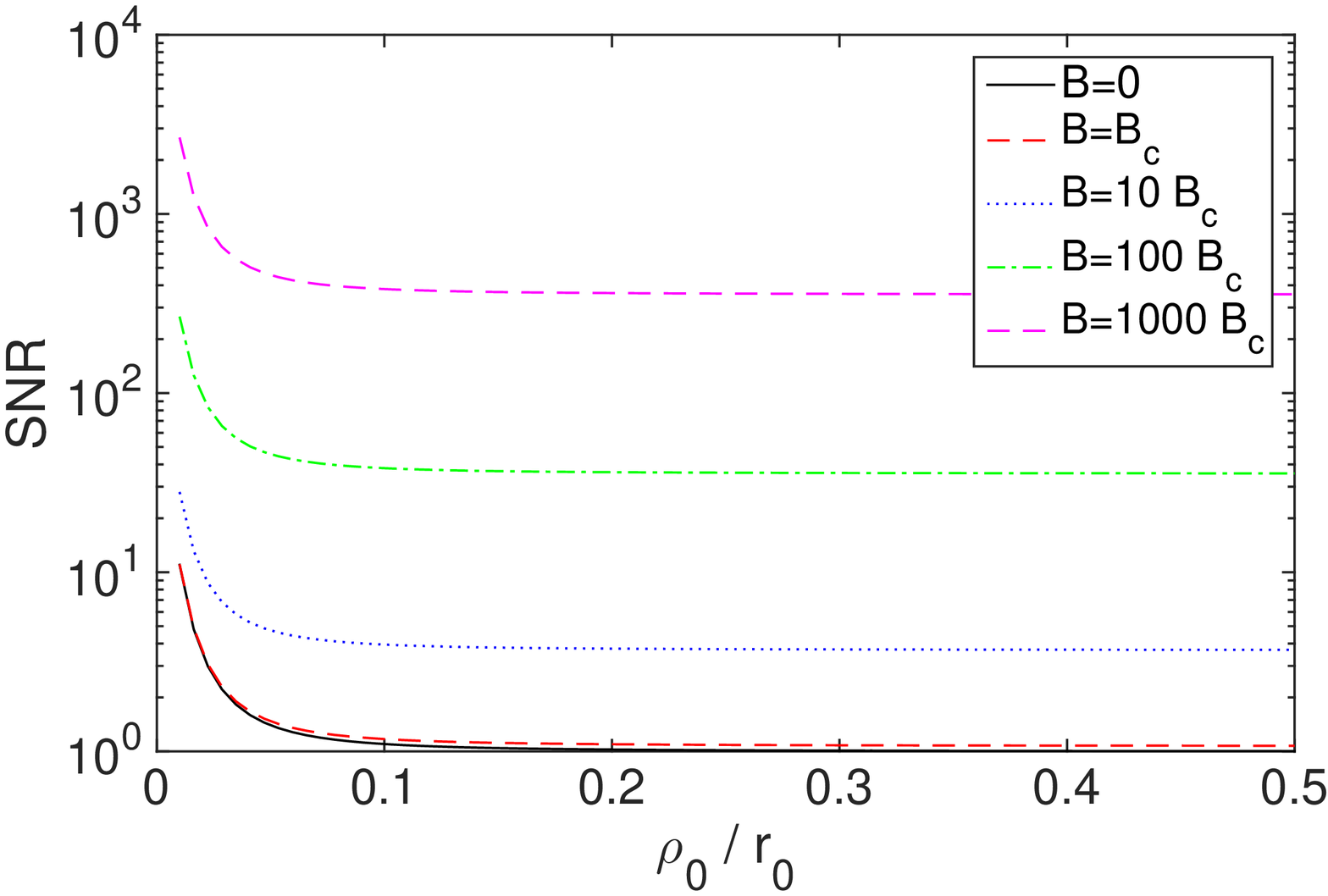}&  
\includegraphics[width=6.0cm]{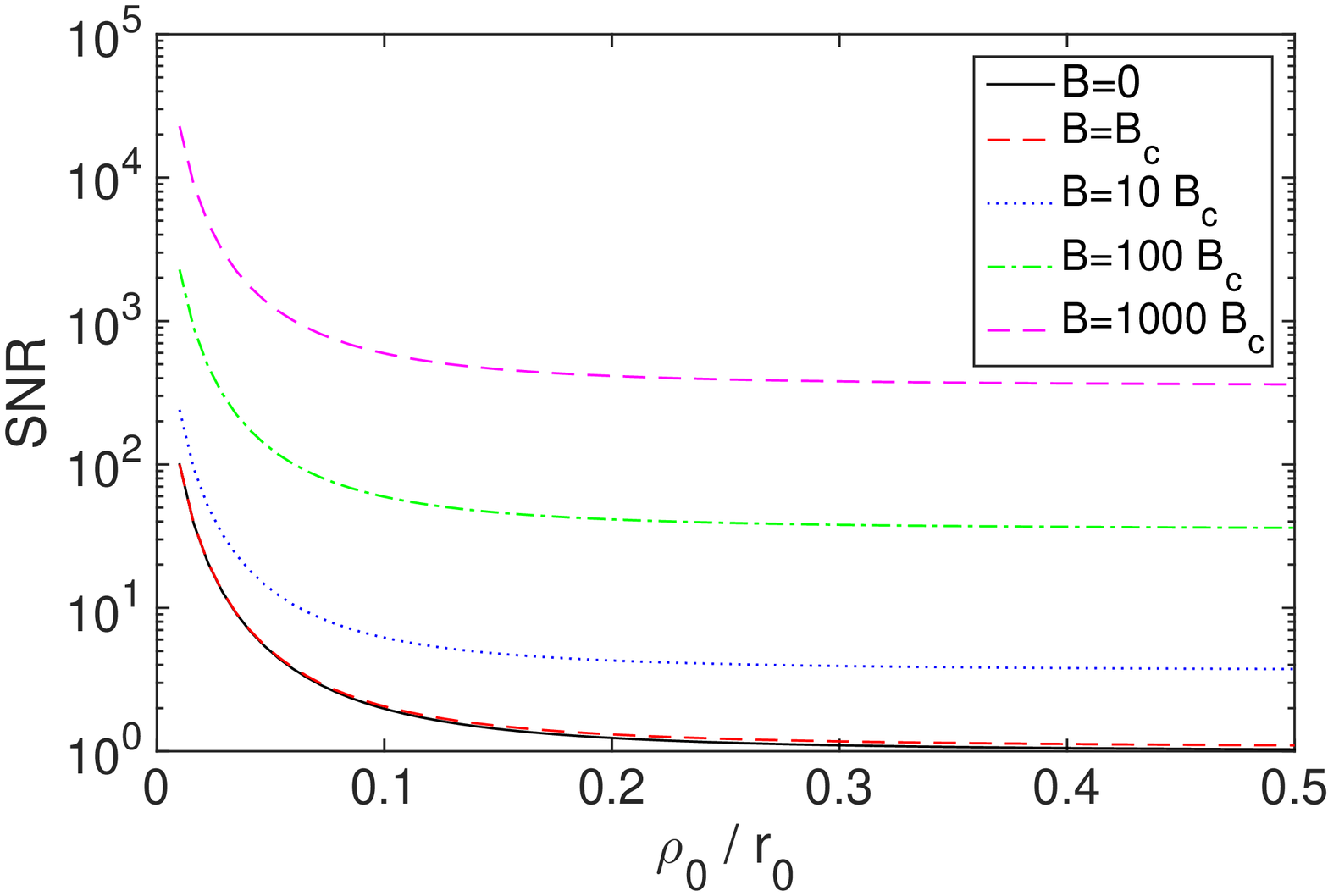}  \\
$\frac{DL^3}{12r_0^2}=0.001$ & $\frac{DL^3}{12r_0^2}=0.01$ \\
\includegraphics[width=6.0cm]{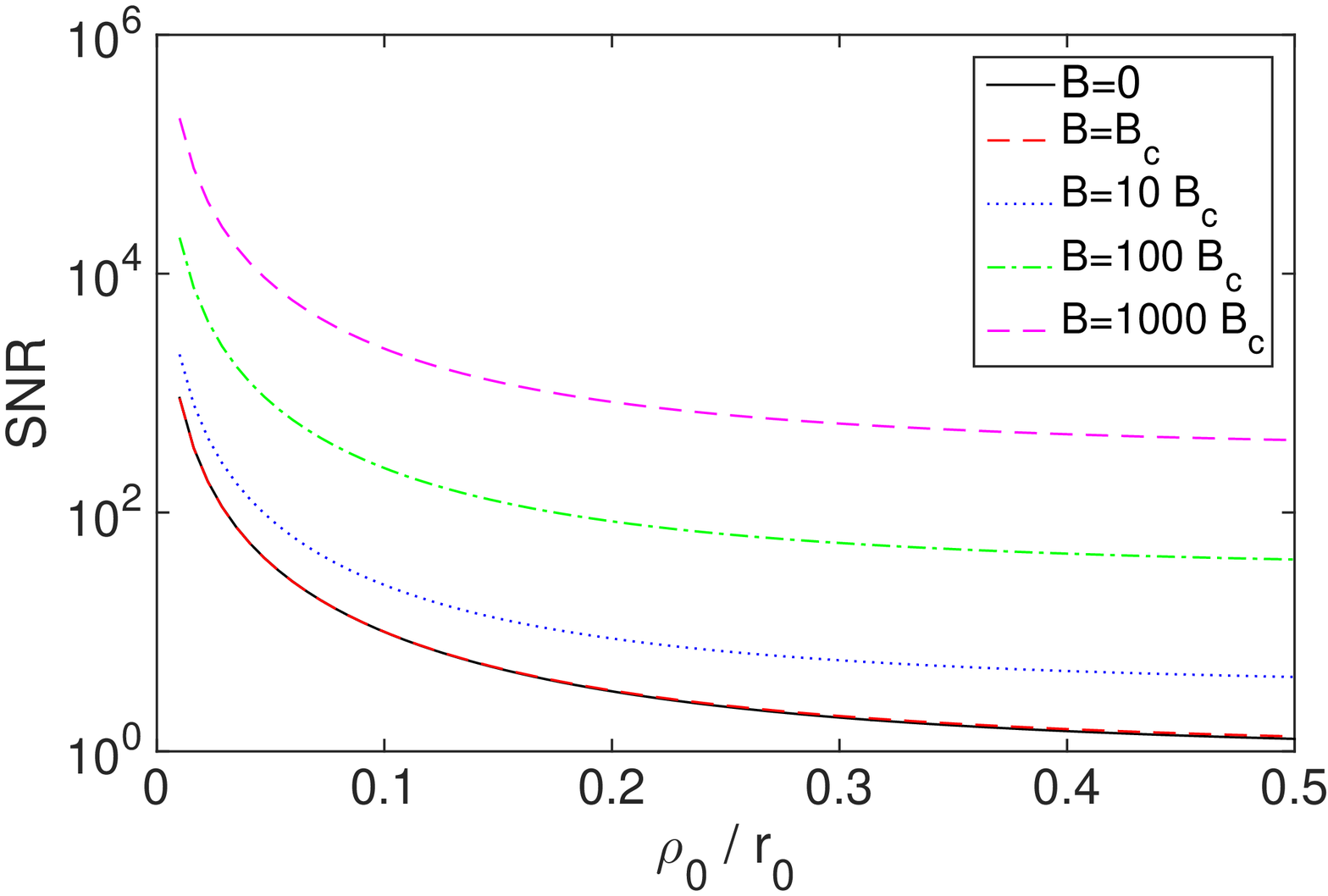}&  
\includegraphics[width=6.0cm]{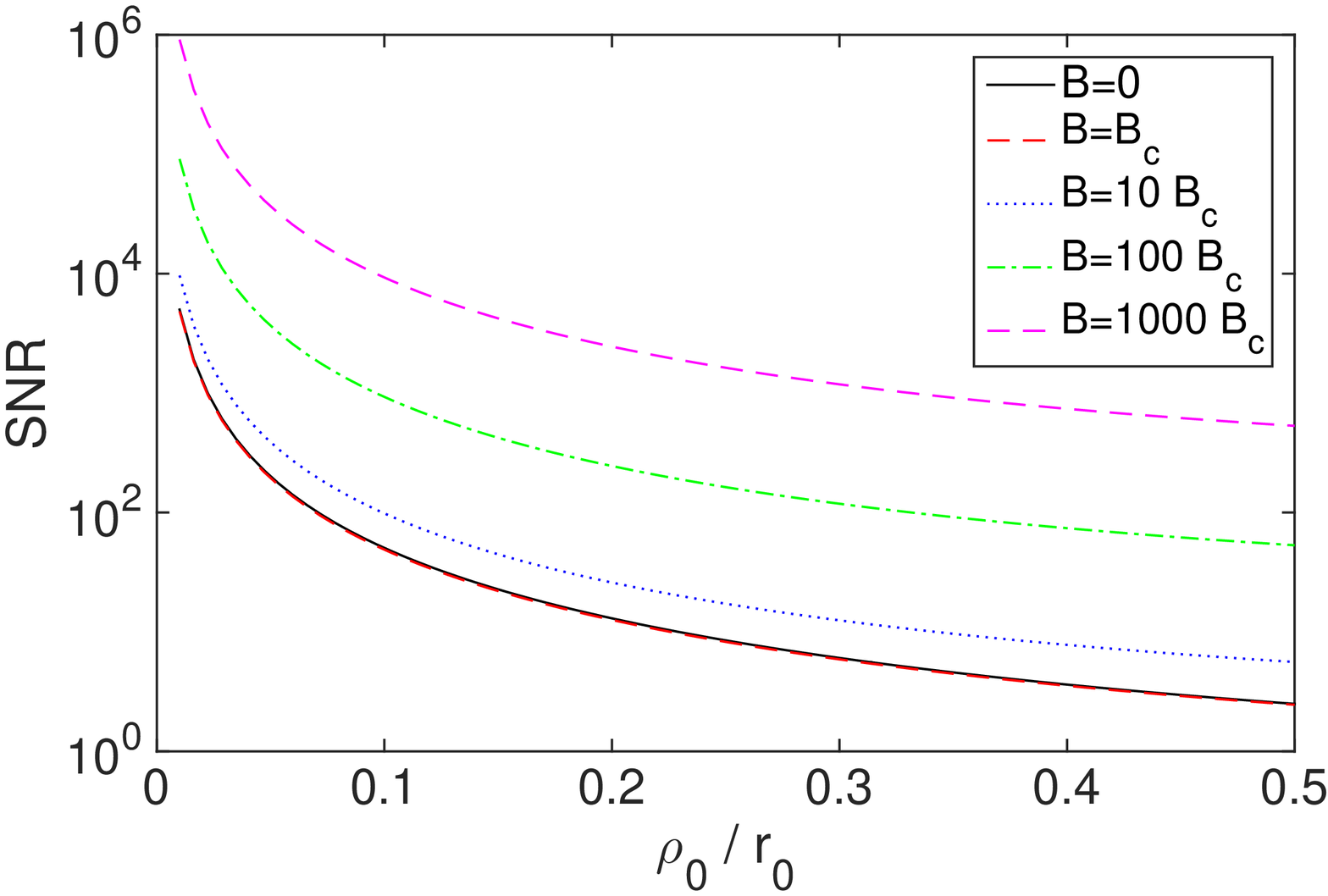}  \\
$\frac{DL^3}{12r_0^2}=0.1$ & $\frac{DL^3}{12r_0^2}=1$ \\
\includegraphics[width=6.0cm]{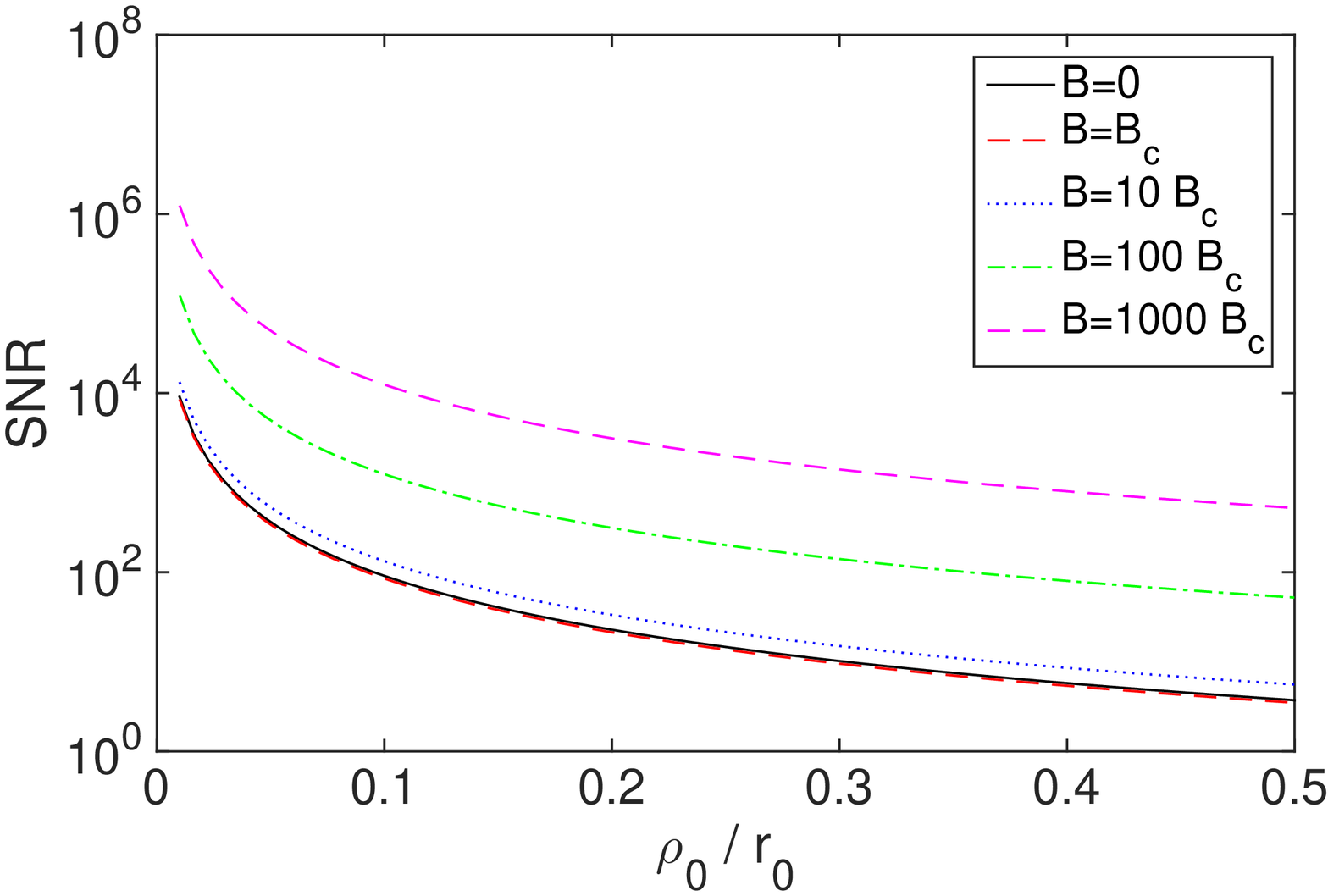}&  
\includegraphics[width=6.0cm]{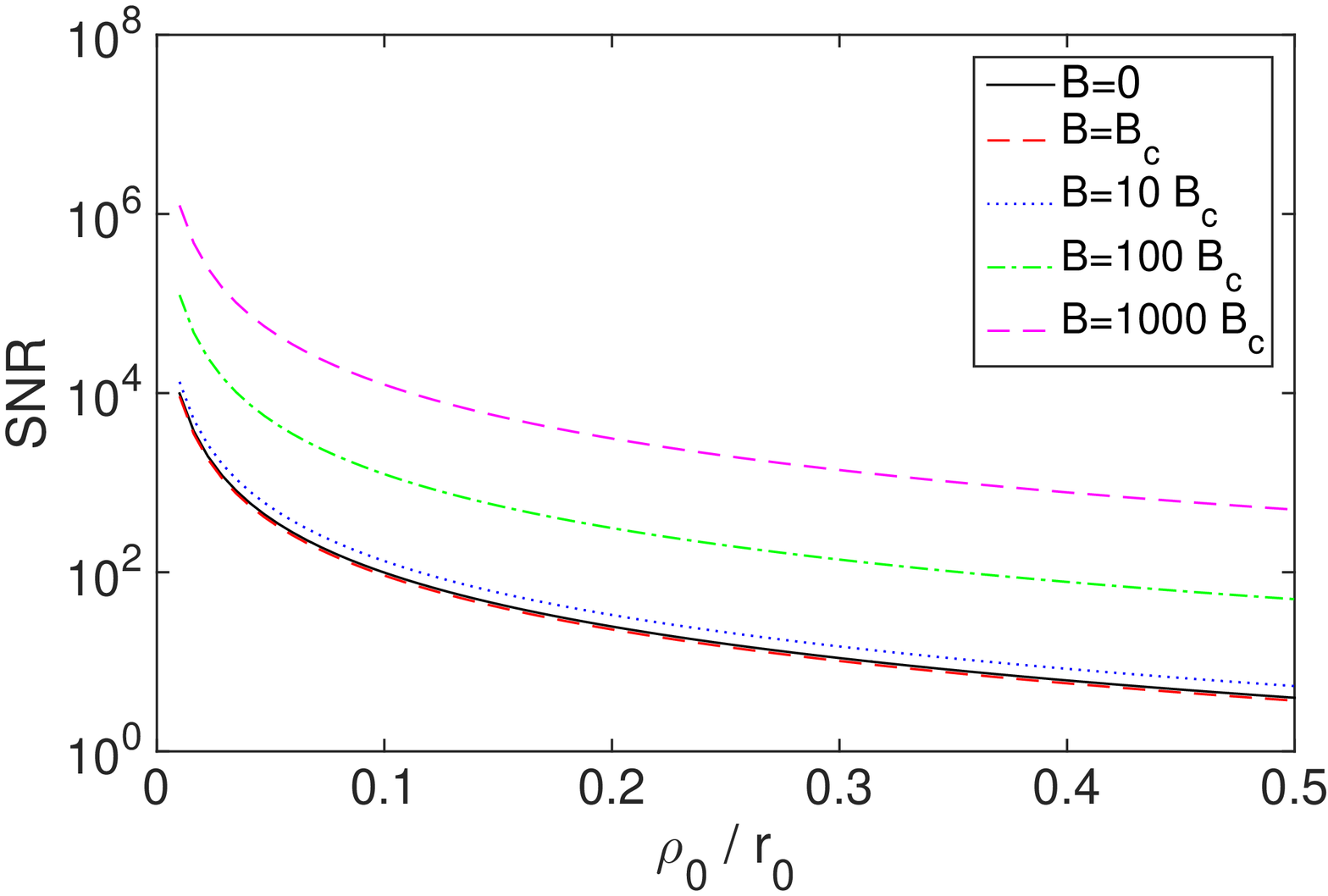}  \\
$\frac{DL^3}{12r_0^2}=10$ & $\frac{DL^3}{12r_0^2}=100$ 
\end{tabular}
\end{center}
\caption{Signal-to-noise ratio (\ref{snr:tr1}) of the time-reversed refocused wave. We denote $B_c = \frac{4c_o}{D  L^2}$.
The value of the SNR for $B=0$ is (\ref{eq:snrB0}).}
\label{fig:snr}
\end{figure}

The functions $\hat{\cal A}$ and $\hat{\cal H}$ satisfy
\begin{eqnarray*}
\hat{\cal A}(s) \simeq  \left\{
\begin{array}{ll}
1 & \mbox{ if } s\ll 1  \\
4 \exp \big( -    \sqrt{2 s}\big)& \mbox{ if } s\gg 1
\end{array}
\right. ,
\quad \quad
\hat{\cal H}(s) \simeq    \left\{
\begin{array}{ll}
1 & \mbox{ if } s\ll 1  \\
\frac{3}{4 \sqrt{2}} s^{-3/2}& \mbox{ if } s\gg 1
\end{array}
\right.
 .
\end{eqnarray*}

Therefore, if $B$ is much smaller than  $4 c_o/(DL^2)$, 
then 
 \begin{eqnarray}
{\rm SNR}
&\simeq&
\frac{1+\frac{DL^3}{12\rho_0^2}}{
1+\frac{DL^3}{12r_0^2}} ,
\label{eq:snrB0}
\end{eqnarray}
which shows that the source bandwidth does not affect the statistical stability of the refocused wave under these conditions.
We have
\begin{equation}
\label{snr:tr2}
{\rm SNR} \simeq
\left\{
\begin{array}{ll} 
\displaystyle
1 & 
\displaystyle
\mbox{ if } \rho_0^2 > \frac{DL^3}{12} ,\\
\displaystyle
\frac{DL^3}{12\rho_0^2} &
\displaystyle
 \mbox{ if } r_0^2 > \frac{DL^3}{12} > \rho_0^2,\\
\displaystyle
\frac{r_0^2}{\rho_0^2} &
\displaystyle
 \mbox{ if }  \frac{DL^3}{12} > r_0^2 .
\end{array}
\right.
\end{equation}
In particular, we recover the fact that, when $DL^3/12 \gg r_0^2$, the SNR is equal to 
 the number $r_0^2/\rho_0^2$ of elements of the TRM.

If $B$ is much larger than  $4 c_o/(DL^2)$, then
 \begin{eqnarray}
{\rm SNR}^{-1}
&\simeq&  \frac{4c_o}{D B L^2}
\frac{2}{ \sqrt{\pi}} \int_0^\infty \hat{\cal A} (s)
\frac{\big( 1+\frac{DL^3}{12 r_0^2}\big)^2}{\big(1+\frac{DL^3}{12 \rho_0^2}\hat{\cal H} (s) \big)
\big(1+\frac{DL^3}{12 r_0^2}\hat{\cal H} (s) \big) } d s ,
\end{eqnarray}
and we find
\begin{equation}
\label{snr:tr3}
{\rm SNR} \simeq
\left\{
\begin{array}{ll} 
\displaystyle
\frac{DL^2 B}{4c_o {\cal A}_1} & 
\displaystyle
\mbox{ if } \rho_0^2 > \frac{DL^3}{12} ,\\
\displaystyle
\frac{DL^2 B}{4c_o {\cal A}_2} 
\frac{DL^3}{12\rho_0^2} &
\displaystyle
 \mbox{ if } r_0^2 > \frac{DL^3}{12} > \rho_0^2,\\
\displaystyle
\frac{DL^2 B}{4c_o {\cal A}_3} 
\frac{r_0^2}{\rho_0^2} &
\displaystyle
 \mbox{ if }  \frac{DL^3}{12} > r_0^2 ,
\end{array}
\right.
\end{equation}
where
$$
{\cal A}_j  = \frac{2}{ \sqrt{\pi}}  \int_0^\infty \frac{\hat{\cal A}(s)}{\hat{\cal H}(s)^{j-1}} ds ,
$$
or more explicitly ${\cal A}_1   \simeq 2.81$, ${\cal A}_2 \simeq 4.40$, and
${\cal A}_3 \simeq 8.05$.
This shows that the source bandwidth improves the statistical stability of the refocused wave,
provided it is larger than $4c_o/(DL^2)$.
In particular, if scattering is so strong that both $DL^3/12 \gg r_0^2$ and $4c_o/(DL^2) \ll B$, 
then the SNR is proportional to the number of elements $r_0^2 / \rho_0^2$
of the TRM times the number of uncorrelated frequency components $(D B L^2)/(4c_o)$
(i.e. the ratio of $B$ over the coherence frequency $4c_o/(DL^2)$).
The equations (\ref{snr:tr2}) and (\ref{snr:tr3}) give the SNR in the different cases and quantify the usual assertion found 
in the literature that the profile of the time-reversed field is self-averaging by independence of the frequency
components of the wave field.

\subsection{The Scintillation Regime Revisited}
In the scintillation regime (\ref{sca:sci2})
(in which $\rho_0$ is of the same order as the correlation length of the random medium),
we find from Proposition \ref{prop:1:2} that the mean refocused field is
\begin{eqnarray*}
&& \EE\big[ {u}_{\rm tr} \big( \frac{t}{\eps} ,\frac{\by}{\eps}+\bx ; \frac{\by}{\eps}\big)\big]\stackrel{\eps\to0}{\longrightarrow}
 \frac{\exp(-B^2 t^2/2) e^{-i\omega_0 t/\eps}}{(2\pi)^4}  \\
&&\times \iint_{\RR^2\times \RR^2} 
\frac{r_0^2 }{4\pi}
\exp\Big( - \frac{r_0^2 |\bzeta|^2}{4} +i \bx\cdot\bxi + i \by \cdot \bzeta
 - i \frac{L c_o }{\omega_0} \bxi \cdot \bzeta  \Big) A(L,\bxi,\bzeta,0) 
 d\bxi d\bzeta ,
 \end{eqnarray*}
  where $A$ is given by (\ref{def:A:2}).

In the weakly scattering regime $L \ll \ell_{\rm sca}$, we find
\begin{eqnarray*}
 \EE\big[ {u}_{\rm tr} \big( \frac{t}{\eps} ,\frac{\by}{\eps}+\bx ; \frac{\by}{\eps}\big)\big]\stackrel{\eps\to0}{\longrightarrow}
 \exp(-B^2 t^2/2)   e^{-i\omega_0 t/\eps}  
\iint_{\RR^2\times \RR^2}
\frac{\rho_0^2 r_0^2 }{4\pi^2}
\exp\Big( - \frac{r_0^2 |\bzeta|^2}{4} - \rho_0^2 |\bxi|^2   \\
+i \bx\cdot\bxi+ i \by \cdot \bzeta - i \frac{L c_o }{\omega_0} \bxi \cdot \bzeta  \Big) 
 d\bxi d\bzeta.
 \end{eqnarray*}
 More exactly, if $\by={\bf 0}$, we get
 \begin{eqnarray*}
 \EE\big[  {u}_{\rm tr} \big( \frac{t}{\eps} ,\bx;{\bf 0} \big)\big]
\stackrel{\eps\to 0}{\longrightarrow} 
\frac{ \exp(-B^2 t^2/2)  e^{-i\omega_0 t/\eps}  }{1+\frac{c_o^2 L^2}{\omega_0^2 \rho_0^2 r_0^2} }
\exp\Big( - \frac{|\bx|^2}{4\rho_0^2 \big( 1+\frac{c_o^2 L^2}{\omega_0^2 \rho_0^2 r_0^2} \big) }\Big),
\end{eqnarray*}
which shows that we get refocusing because the TRM element size $\rho_0$ is small enough. 
When $\rho_0$ becomes very small, i.e. smaller than $c_o L /(\omega_0 r_0)$, then the radius 
of the refocused wave is $\sqrt{2} c_o L/(\omega_0 r_0)$, 
which is the diffraction limit or Rayleigh resolution formula.

In the strongly scattering regime $L \gg \ell_{\rm sca}$, 
we find by Proposition \ref{prop:4:2}  that
\begin{eqnarray*}
 \EE\big[ {u}_{\rm tr} \big( \frac{t}{\eps} ,\frac{\by}{\eps}+\bx ; \frac{\by}{\eps}\big)\big]
&=&
\frac{r_0^2 \exp(-B^2 t^2/2)  e^{-i\omega_0 t/\eps}  }{4 \pi} \\
&& \times \int_{\RR^2} \exp\big( - e_0(L) |\bzeta|^2 -f_0(L) \bx\cdot \bzeta - b_0(L) |\bx|^2 +i \bzeta \cdot \by \big)d\bzeta ,
\end{eqnarray*}
with
\begin{eqnarray*}
b_0(z) = \frac{\omega_0^2 D z}{16 c_o^2}  +\frac{1}{4\rho_0^2} ,\quad \quad 
e_0(z)  =\frac{r_0^2}{4}  + \frac{Dz^3}{48}  +\frac{c_o^2 z^2}{4\omega_0^2 \rho_0^2} ,\quad \quad 
f_0(z) = -  \frac{\omega_0 D z^2}{16 c_o}  - \frac{c_o z}{2\omega_0 \rho_0^2}.
\end{eqnarray*}
In particular, if $\by={\bf 0}$, we find
\begin{align}
\nonumber
 \EE\big[ {u}_{\rm tr} \big( \frac{t}{\eps} , \bx ; {\bf 0}\big)\big]
 =& \frac{   e^{-i\omega_0 t/\eps}  }{1+\frac{DL^3}{12 r_0^2} +\frac{c_o^2 L^2}{\omega_0^2 \rho_0^2 r_0^2} } \exp \Big( - \frac{B^2 t^2}{2} \Big)\\
 &\times\exp\Big(  - \frac{ \frac{\omega_0^2 DL}{16 c_o^2}  \big(1+\frac{DL^3}{48 r_0^2}\big)
 +\frac{1}{4\rho_0^2}  \big(1+\frac{DL^3}{12 r_0^2}\big)
 }{1+\frac{DL^3}{12 r_0^2} +\frac{c_o^2 L^2}{\omega_0^2 \rho_0^2 r_0^2}}
|\bx|^2 \Big) ,
 \end{align}
 which makes it possible to identify the amplitude and the radius $R$ of the refocused wave (as in (\ref{eq:meanrefocussf:2}-\ref{eq:radiussf:2})):
\begin{equation}
 R^2 = \frac {1+\frac{DL^3}{12 r_0^2} +\frac{c_o^2 L^2}{\omega_0^2 \rho_0^2 r_0^2}}
{\frac{\omega_0^2 DL}{8 c_o^2}  \big(1+\frac{DL^3}{48 r_0^2}\big)
 +\frac{1}{2\rho_0^2}  \big(1+\frac{DL^3}{12 r_0^2}\big)
 }.
\end{equation}
 The radius 
is smaller when the TRM element size $\rho_0$ is smaller (we have $\partial R/\partial \rho_0 >0$) and when the random medium is more scattering
(we have $\partial R / \partial D<0$).
It is not surprising  that time-reversal refocusing is improved when the TRM has many array elements and better resolve
the  wave field on the mirror, moreover,
it is well-known that random scattering improves time-reversal refocusing by multipathing \cite{blomgren,book1}.
When $\rho_0$ becomes very small,  i.e. smaller than  $c_o L /(\omega_0 r_0)$
and  $c_o / (\omega_0 DL)$, then the radius of the refocused wave
is equal to 
$$
R= 
\frac{\sqrt{2} c_o L }{\omega_0 \sqrt{ r_0^2 +\frac{DL^3}{12} } } ,
$$
which is the Rayleigh resolution formula but with the enhanced 
TRM radius $r_{\rm eff} = \sqrt{ r_0^2 +\frac{DL^3}{12} }$.
This result can be found in the literature \cite{blomgren}.
 
Finally, we find from Proposition \ref{prop:3:2} that 
\begin{equation}
\lim_{\eps \to 0}
\EE\big[ \big| {u}_{\rm tr} \big( \frac{t}{\eps} ,\frac{\by}{\eps}+\bx ; \frac{\by}{\eps}\big) \big|^2\big] 
=
\lim_{\eps \to 0}
\big| \EE\big[ {u}_{\rm tr} \big( \frac{t}{\eps} ,\frac{\by}{\eps}+\bx ; \frac{\by}{\eps}\big)\big] \big|^2.
\end{equation}
The refocused wave is statistically stable in this regime, because there are many elements (of the order of $\eps^{-2}$) in the TRM.

\section{Conclusion}
In this paper we have analyzed the fourth-order moment of the random paraxial Green's function 
at four different frequencies. 
We have obtained  a complete characterization in the scintillation regime,
which  makes it possible to quantify the speckle memory effect in the frequency domain  
in terms of the propagation distance through the scattering medium and statistics of the medium fluctuations.
Using this result we have also been able to obtain for the  first time a quantitative  characterization of  the 
statistical stability  in the classic  time-reversal refocusing experiment. This characterization depends on  
the radius of the time-reversal mirror, the size of its elements, and the source bandwidth,
as well as  the statistics of the medium fluctuations.
As anticipated and observed in experiments \cite{derode01,lerosey04},
when the medium is strongly scattering, the signal-to-noise ratio of the time-reversed refocused wave
is given by the number of elements of the time-reversal mirror times the number of independent
frequency components in the source bandwidth.

 \section*{Acknowledgements}

JG was supported by the Agence Nationale pour la Recherche under Grant No. ANR-19-CE46-0007 (project ICCI), 
 and Air Force Office of Scientific Research under grant FA9550-18-1-0217.
 \\
KS was supported by the Air Force Office of Scientific Research under grant FA9550-18-1-0217,  and   the National Science Foundation under grant DMS-2010046.
 
\appendix 
 
\section{The White-Noise Paraxial Regime and the Scintillation Regime}
\label{app:a}%
In this paper we consider a primary scaling regime in which the solutions of the Helmholtz equation  
(\ref{eq:helm})
can be approximated in terms of the Green's function solving the  It\^o-Schr\"odinger equation  (\ref{eq:model}). 
This is the  white-noise paraxial regime where the propagation distance is large compared to the 
correlation length of the medium which is on the same scale as the beam radius (or source width),
which in turn is large compared to the wavelength.
The  It\^o-Schr\"odinger  description allows us to get explicit
expressions for the second-order moments of the wave field 
at a fixed frequency. 
In this paper we use the second moment  at two frequencies  to describe the mean refocused  wave field in 
time reversal  when we average with respect to the random medium in (\ref{eq:rm})
corresponding to averaging with respect to the driving Brownian motion $B$ in (\ref{eq:model}).   
 It is also important to describe the statistical
stability of empirical covariances or time reversed fields when 
 formed from one realization of the medium.  
 Such statistical stability or signal-to-noise ratio analysis requires expressions for the fourth
moment of the wave field with the wave field components in the moment evaluated at different frequencies.
In this paper we consider a secondary scaling regime,
the scintillation regime, which allows us to get explicit expressions for the multi-frequency 
moments of the wave field.  The scintillation regime 
is valid in the paraxial white-noise regime when, additionally, the correlation length of the medium is small compared to the beam radius as described in Section \ref{sec:regime}.  

In this appendix we  discuss  these two scaling regimes,
the paraxial white-noise and scintillation regimes, and the relation to the  It\^o-Schr\"odinger equation, and we refer to \cite{garniers1,garniers4} for the full derivation.  
 Consider  $\hat{u}(z,\bx)$ satisfying  the Helmholtz equation
(\ref{eq:helm}). 
Let $\sigma$ be the standard deviation of the fluctuations of the index of refraction
$n$ in this equation.   Moreover, assume here that the random fluctuations of the index of
refraction  is isotropic and denote by $l_{\rm c}$ the correlation length of the fluctuations, 
by $\lambda$ the  wavelength,
by $L$ the typical propagation distance, 
and by $r_o$ the transverse radius of the initial beam,
which in this paper corresponds to the dimension of the time-reversal mirror. 
We introduce the wavenumber defined by
\begin{equation}
   k = \frac{\omega}{c_o} = \frac{2\pi}{\lambda},
\end{equation}
with $c_o$ the background wave speed.  
In this framework the variance $C({\bf 0})$ of the Brownian field in the It\^o-Schr\"odinger equation
(\ref{eq:model}) is of order $\sigma^2 l_{\rm c}$ and the transverse scale of variation of the covariance function
$C(\bx)$ in (\ref{def:C2}) is of order $l_{\rm c}$.
   
First,  we consider the primary (paraxial white-noise) scaling that leads to
the It\^o-Schr\"odinger equation (\ref{eq:model2}), which corresponds 
to zooming in on a high-frequency beam that propagates
over a distance that is large relative to the correlation length of the medium, which is itself large
relative to the wavelength,
moreover, the medium fluctuations are small. 
Explicitly, we assume the primary scaling  when
\begin{eqnarray*}
   \frac{l_{\rm c}}{r_o} \sim 1  \, ,   \quad\quad
      \frac{l_{\rm c}}{L}  \sim   \theta\, ,    \quad \quad
     \frac{l_{\rm c}}{\lambda}     \sim    \theta^{-1} \,   ,     \quad  \quad
     \sigma^2 \sim \theta^3 \,  , 
\end{eqnarray*}
where $\theta$ is a small dimensionless parameter.
We introduce dimensionless coordinates by:
\begin{eqnarray*}
\bx =   { l_{\rm c} }    { \bx' }  , \quad  \quad  
z =  L z', \quad\quad  
k = \frac{k' }{  l_{\rm c} \theta }, \quad \quad  
 \nu ( L z',  l_{\rm c} \bx') = \theta^{3/2}  \nu' \Big(\frac{z'}{\theta} ,  \bx'  \Big)   ,
\end{eqnarray*}
with $\nu$ being the relative fluctuations of the random medium  (\ref{eq:rm}). 
Then dropping  `primes'   we find that in  dimensionless coordinates
the Helmholtz equation reads
$$
\left({\theta^2} \partial_z^2+\Delta_\bx \right) \hat{u}^\theta + \frac{k^2}{\theta^2} \left(1 
+ \theta^{3/2} \nu \left( \frac{z}{\theta},\bx \right) \right) \hat{u}^\theta= 0 .
$$  
We look for the behavior of the slowly-varying envelope $v^\theta$
for  propagation distances of order one in the dimensionless coordinates:
$$
 \hat{u}^\theta  ({z},\bx  )  = \exp\Big( i \frac{k z}{\theta^2} \Big)
v^\theta ( z, \bx  ) 
$$
that satisfies (by the chain rule)
$$
{\theta^2} \partial_{z}^2 v^\theta
+ 
\left( 2 i k \partial_z 
v^\theta+ \Delta_\bx v^\theta +  \frac{k^2}{\theta^{1/2}}  \nu\Big(  \frac{z}{\theta} , \bx\Big) 
 v^\theta \right)=0  .
$$
Heuristically, when $\theta \ll 1$ the backscattering term ${\theta^2} \partial_{z}^2  v^\theta$
can be neglected and we obtain a Schr\"odinger-type equation 
in which the potential fluctuates in $z$ on the  
scale $\theta$  and is of amplitude
$ \theta^{-1/2}$. 
This diffusion approximation then gives the
It\^o-Schr\"odinger equation or white-noise limit
driven by a  Brownian field:
\begin{equation}\label{eq:model2}
2 i k d v +  \Delta_\bx v \, dz +   k^2 
 v\circ dB(z,\bx) =0 ,
\end{equation}
or  (\ref{eq:model}) when written in terms of the Green's function.
This heuristic derivation can be made rigorous as shown in \cite{garniers1}.
This equation is written in  Stratonovich form as represented by the $\circ$ symbol.
This reflects the fact that we arrive at this description as a  scaling limit of a physical model where the fluctuating random field $\nu$ multiplying  
the wave field $u$  has  a finite correlation length in the $z$-direction.      
The Stratonovich  stochastic integral can be interpreted in the simplest case as 
 the limit when the integrand is evaluated at the midpoint of the interval of increment
 of the driving Brownian field and thus naturally appears in the diffusion 
 limit when  $\nu$ is replaced by a driving Brownian field.  
 Note that the It\^o interpretation 
 of   (\ref{eq:model2}) has the form  
\begin{equation}\label{eq:model3}
d v =  \frac{i}{2k} 
 \Delta_\bx v \, dz  - \frac{k^2 C({\bf 0})}{8} v dz+   \frac{i k}{2 }   v \, dB(z,\bx) .
\end{equation}
In this representation the last term integrates to a zero-mean martingale term
and  the added damping term is the Stratonovich corrector. We then  have for
the mean  field $\bar{v}=\EE[v]$:
  \begin{equation}\label{eq:model4}
 \partial_z \bar{v}  = \frac{i}{2k} 
 \Delta_\bx \bar{v}   - \frac{k^2 C({\bf 0})}{8} \bar{v} . 
  \end{equation} 
Here the damping term reflects scattering and transfer of energy from the coherent
part of the wave field to the incoherent part so that the mean field is exponentially damped.   
Indeed the reciprocal of the damping parameter was referred to as the scattering mean free path in (\ref{def:lsca})
and characterizes the distance a coherent wave can travel before  wave energy is scattered 
to the incoherent part.  

The representation (\ref{eq:model3}) gives closed equations for moments of all orders. 
 We can easily solve explicitly the first-order moment in (\ref{eq:model4}) 
  and  also the second-order moment equations at a single frequency. 
  As mentioned however there is no explicit solution for the  fourth moment equations.
  We discuss now the secondary scaling limit that we refer to as the scintillation regime
  where we can solve explicitly for the fourth moment both in the single frequency case 
  (and also in the  multi-frequency case up to a second-order lateral scattering function 
  that can be explicitly characterized in the case of relatively strong scattering). 
In the scintillation regime the correlation 
length of the medium $l_{\rm c}$ is smaller than the initial beam radius $r_o$.
Moreover, the medium fluctuations are weak, and the beam propagates deep
into the medium. We  then get the modified scaling picture
\begin{equation}
\label{eq:scaling app}
    \frac{l_{\rm c}}{r_o} \sim  \eps  \, ,  \quad\quad
     \frac{l_{\rm c}}{L}  \sim   \theta \eps \, ,   \quad\quad
     \frac{l_{\rm c}}{\lambda} \sim \theta^{-1}   \,   ,    \quad\quad
     \sigma^2 \sim \theta^3 \eps \,      ,
\end{equation}
and we assume $  \theta \ll \eps \ll 1$.
This means that the paraxial white-noise limit $\theta \to 0$ is taken first, 
and we find
$$ 
2ik d {v}^\eps      
    +\Delta_{\bx}   {v}^\eps \, dz
   +  k^2   {v}^\eps   \circ  d{B}^\eps(z,\bx) 
 =0 , 
$$
where the radius $r_o^\eps$ of the initial condition is of  order $\eps^{-1}$,  
the variance $C^\eps({\bf 0})$ of the Brownian field $B^\eps$ is of order $\eps$,
and the propagation distance $L^\eps$ is of order $\eps^{-1}$.  
Then the limit  $\eps\to 0$ is applied, corresponding to the scintillation regime. 
In the regime (\ref{eq:scaling app}) the effective strength $k^2 C^\eps({\bf 0}) L^\eps$ of the 
Brownian field is of order one since $  \sigma^2 l_{\rm c} L/\lambda^2   \sim 1$.
Moreover,  $L^\eps \lambda/ (r_o^\eps)^2$ is of order $\eps$.
That is,   the typical propagation distance is smaller than the Rayleigh length of the initial beam.
Here the Rayleigh length corresponds to the distance when
the transverse radius of the beam has roughly  doubled by diffraction
in the homogeneous medium case and it  is given by $r_o^2/\lambda$.
Indeed,  it is seen  in Section \ref{sec:4} that the propagation distance at which relevant phenomena arise 
in the random case is of the order
of $r_o l_{\rm c} / \lambda$,  which is smaller than the Rayleigh distance of
the homogeneous medium $r_o^2/\lambda$.

\end{document}